\documentclass{amsart}

\usepackage{bm,color,epsfig,enumerate,caption}
\usepackage{amsmath,amssymb}
\usepackage{amsthm}
\usepackage{amsrefs}
\usepackage{indentfirst}
\usepackage{mathrsfs}
\usepackage{graphicx}
\usepackage{hyperref}
\usepackage{xcolor}
\usepackage{fourier}
\usepackage[margin=1.2in]{geometry}
\usepackage{algpseudocode}
\usepackage{algorithm}
\usepackage{algorithmicx}

\theoremstyle{plain}
\newtheorem{theorem}{Theorem}
\newtheorem{lemma}{Lemma}

\newtheorem{algo}[theorem]{Algorithm}
\newtheorem{algomain}{Algorithm}

\theoremstyle{definition}
\newtheorem{defn}[theorem]{Definition}

\theoremstyle{remark}

\newcommand{\grad}{\nabla}

\newcommand{\ud}{\,\mathrm{d}}

\newcommand{\RR}{\mathbb{R}}
\newcommand{\NN}{\mathbb{N}}
\newcommand{\ZZ}{\mathbb{Z}}
\newcommand{\TT}{\mathrm{T}}

\newcommand{\Or}{\mathcal{O}}

\newcommand{\Angle}{\text{Angle}}
\newcommand{\TE}{\text{TE}}
\newcommand{\BD}{\text{BD}}
\newcommand{\Vol}{\text{Vol}}
\newcommand{\SNR}{\text{SNR}}

\newcommand{\wt}[1]{\widetilde{#1}}

\IfFileExists{mathabx.sty}%
  {\DeclareFontFamily{U}{mathx}{\hyphenchar\font45}%
   \DeclareFontShape{U}{mathx}{m}{n}{<->mathx10}{}%
   \DeclareSymbolFont{mathx}{U}{mathx}{m}{n}%
   \DeclareFontSubstitution{U}{mathx}{m}{n}%
   \DeclareMathAccent{\widebar}{0}{mathx}{"73}%
}{%
  \PackageWarning{mathabx}{%
    Package mathabx not available, therefore\MessageBreak substituting
    widebar with overline\MessageBreak }%
  \newcommand{\widebar}[1]{\overline{#1}}%
}

\newcommand{\eps}{\epsilon}

\newcommand{\VAR}{\text{Var}}
\newcommand{\abs}[1]{\lvert#1\rvert}

\renewcommand{\Re}{\mathfrak{Re}}

\title{Crystal image analysis using $2D$ synchrosqueezed transforms}

\author{Haizhao Yang}
\address{Department of Mathematics, Duke University}
\email{haizhao@math.duke.edu}

\author{Jianfeng Lu}
\address{Departments of Mathematics, Physics, and Chemistry, Duke University} 
\email{jianfeng@math.duke.edu}

\author{Lexing Ying}
\address{Department of Mathematics and Institute for Computational \& Mathematical Engineering, Stanford University}
\email{lexing@math.stanford.edu}

\date{\today}

\begin{document}

\begin{abstract}

We propose efficient algorithms  based on a band-limited version of $2D$ synchrosqueezed transforms to extract mesoscopic and microscopic information from atomic crystal images. The methods analyze atomic crystal images as an assemblage of non-overlapping segments of $2D$ general intrinsic mode type functions, which are superpositions of non-linear wave-like components. 
In particular, crystal defects are interpreted as the irregularity of local energy; crystal rotations are described as the angle deviation of local wave vectors from their references; the gradient of a crystal elastic deformation can be obtained by a linear system generated by local wave vectors. Several numerical examples of synthetic and real crystal images are provided to illustrate the efficiency, robustness, and reliability of our methods.

\end{abstract}



\maketitle

\section{Introduction}


In materials science, crystal image analysis on a microscopic length
scale has become an important research topic recently
\cites{ZengGipson:07, SingerSinger:06, Berkels:08, Berkels:10,
  Strekalovskiy:11, ElseyWirth:13, ElseyWirth:MMS, StukowskiAlbe1:10,
  StukowskiAlbe2:10}. The development of image acquisition techniques
(such as high resolution transmission electron microscopy (HR-TEM) \cite{King:98}) and the advancement of atomic simulation of molecular
dynamics \cite{Abraham2002} or mean field models like phase field
crystals \cites{ElderGrantMartin:04, ElseyWirth:12} create
data of large scale crystalline solids with defects at an atomic
resolution. This provides unprecedented opportunities in understanding
materials properties at a microscopic level. Given that defects such as
dislocations and grain boundaries are of crucial importance to materials
properties, it is essential to have efficient tools to analyze large
scale images and to extract information about defects in the system.
Previously, this was done through analyzing the microscopic image
manually; this is becoming impractical due to the extraordinarily large magnitude of the
measurements and simulations we are dealing with nowadays, in particular,
in the case of analyzing a time series of crystal images during an
evolution process. Therefore, this calls for the development of
efficient and automatic crystal image analysis tools.

In this work, we will limit our scope to $2D$ images of (slices of)
polycrystalline materials and aim at extracting mesoscopic and
microscopic information from the given images. This involves the
identification of point defects, dislocations, deformations, grains
and grain boundaries.  Grains are material regions that are composed
of a single crystal, possibly with different orientations (which will
be referred as crystal rotations later, since we are working in $2D$);
they are usually slightly deformed due to defects and
interactions with neighboring grains at grain boundaries.  Grains might contain point defects like vacancies and interstitials,
for which we would like identify their positions.  Crystal analysis should also be able to locate cores and Burgers vectors for dislocations,
which play an important role in crystal plasticity. We refer the
readers to \cite{Bravais} for more background details of
polycrystals and crystal defects. Our proposed method provides a reliable and efficient way for extracting this information from crystal images. 


\subsection{Our contribution}

Due to the lattice structure on the microscopic scale, crystal images are highly oscillatory (see Figure \ref{fig:osc} (right) as an example). Inspired by this, we introduce a new characterization of grains by studying $2D$ general shape functions and $2D$ general intrinsic mode type functions, which are superpositions of non-linear and non-stationary wave-like components (more precise definitions will be given in Section~\ref{sec:model}). Using these two concepts, a crystal image can be considered as an assemblage of $2D$ general intrinsic mode type functions with non-overlapping supports, specified propagating directions and smoothly varying local wave vectors (see Figure \ref{fig:osc} (left) as an example). In this model, crystal defects and grain boundaries can be detected through the discontinuity and irregularity of these components; crystal rotations and crystal deformations are estimated from a linear system provided by local wave vectors of underlying wave-like components.

\begin{figure}[ht!]
  \begin{center}
    \begin{tabular}{cc}
   \includegraphics[height=1.8in]{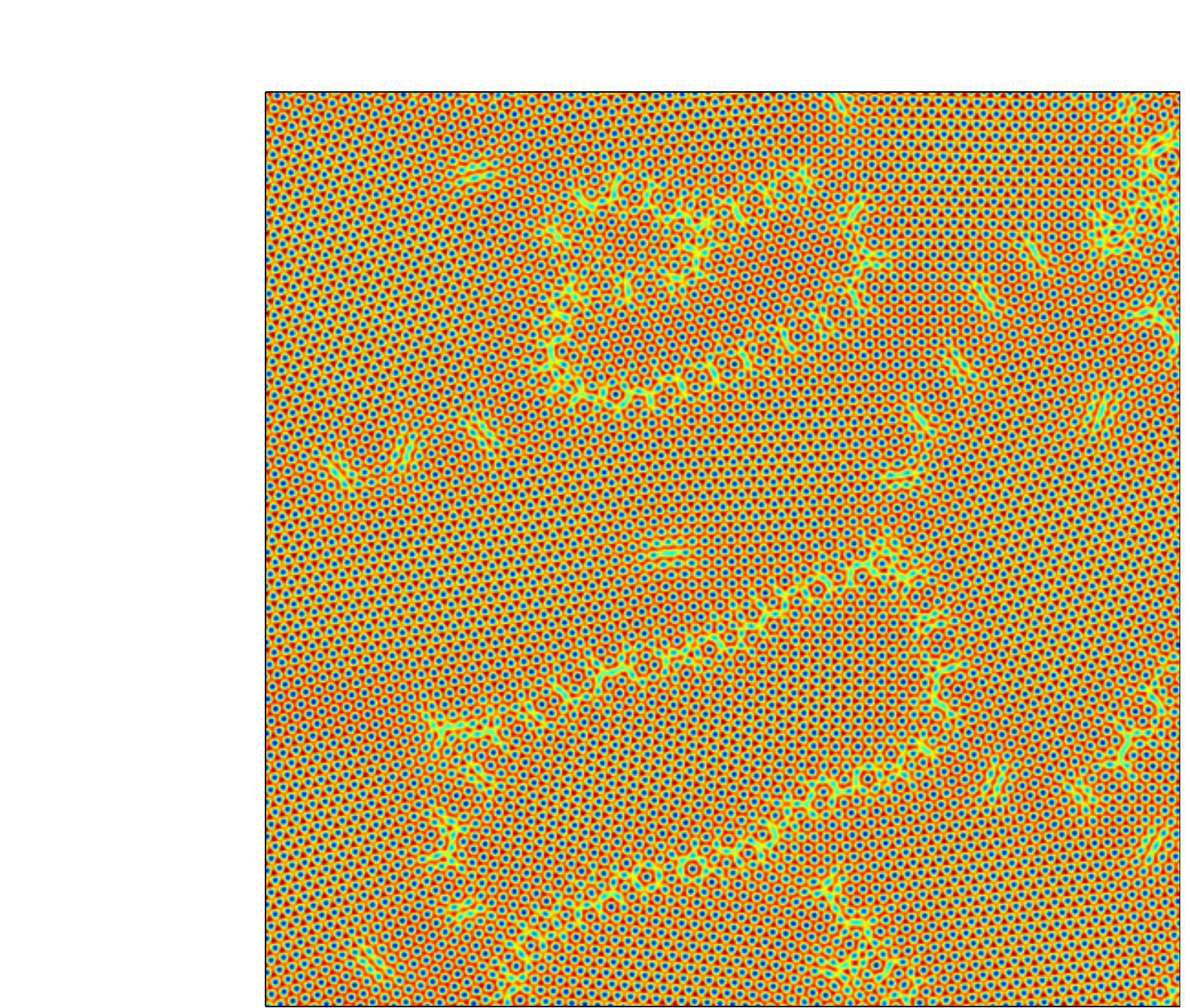}&     \includegraphics[height=1.8in]{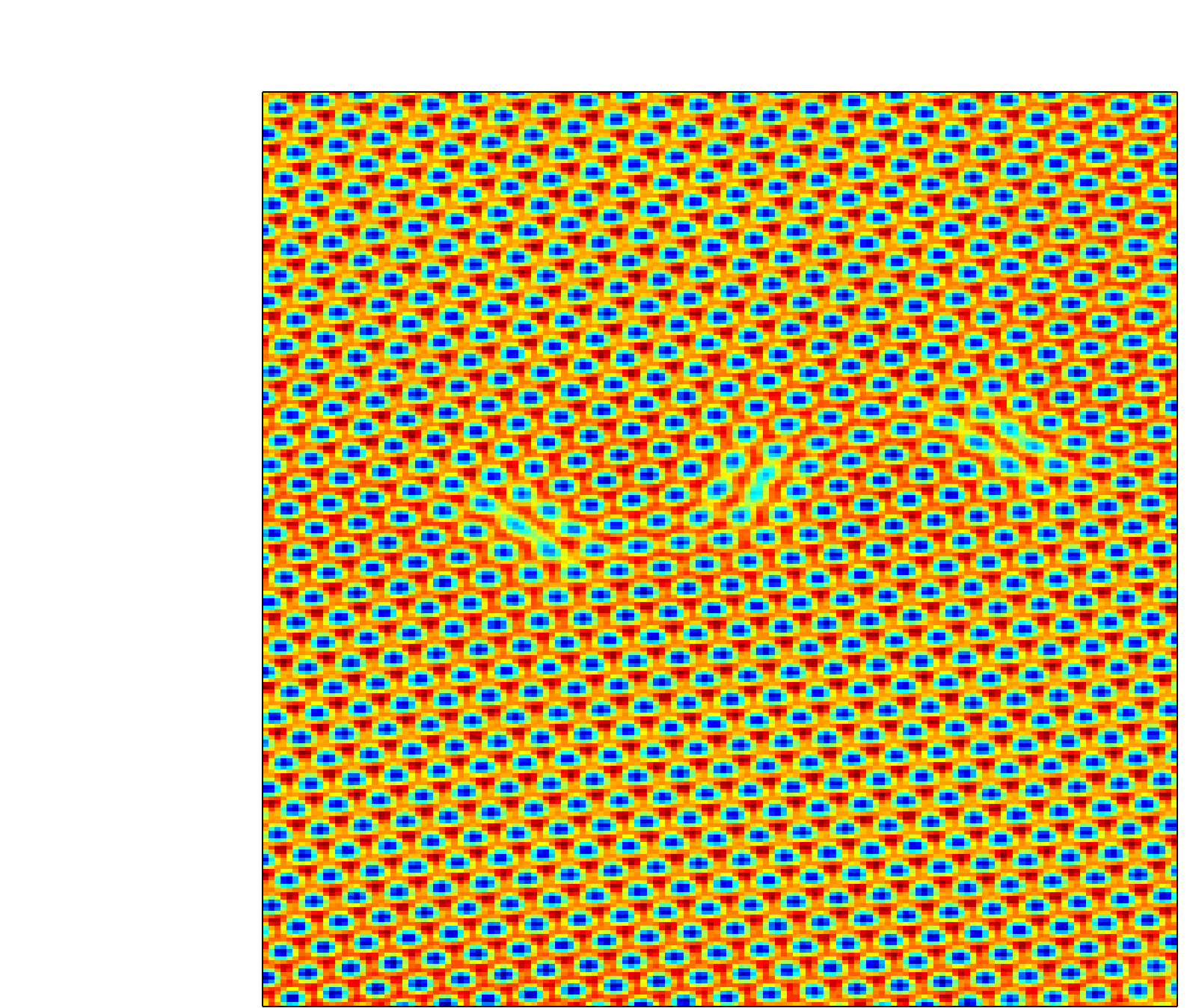}   
    \end{tabular}
  \end{center}
  \caption{A phase field crystal (PFC) image and its zoomed-in image. Courtesy of Benedikt Wirth \cite{ElseyWirth:MMS}.}
  \label{fig:osc}
\end{figure}

The study of superpositions of non-linear and non-stationary wave-like components has been an active line of analysis in the past two decades. The synchrosqueezing technique proposed in \cite{DaubechiesMaes:96} is applied to analyze and decompose superpositions of $1D$ signals using the synchrosqueezed wavelet transform in \cite{DaubechiesLuWu:11}. Following the same methodology, synchrosqueezed wave packet transforms \cite{YangYing:2013} and synchrosqueezed curvelet transforms \cite{YangYing:preprint} are proposed to analyze and decompose superpositions of $2D$ components. $2D$ synchrosqueezed transforms give a concentrated spectrum on the phase (spatial frequency) plane, which can help to obtain information like local wave vectors and a spectral density in the $4D$ phase space domain. Recent study on the robustness of these synchrosqueezed transforms in \cite{YangYingPreprint:2014} supports their applications to real problems where noise is ubiquitous. These transforms have been applied to analyze seismic wave images with great success in \cites{YangYing:2013, YangYing:preprint}. The numerical experiements in \cite{Canvas} show that these transforms have better statistical stability than the window Fourier transform in the application of canvas thread counting.

In this paper, we propose efficient algorithms to analyze crystal images by adapting $2D$ synchrosqueezed transforms to the problems at hand. First, $2D$ synchrosqueezed transforms are applied to obtain the synchrosqueezed energy distributions of underlying wave-like components. Second, since the synchrosqueezed energy is concentrating on local wave vectors, these local wave vectors can be estimated by averaging the supports of the synchrosqueezed energy distribution. Third, the irregularity of each wave-like component can be measured by the irregularity of its corresponding synchrosqueezed energy distribution. Finally, with such information ready, the crystal image can be analyzed as discussed previously.

\subsection{Previous works}

One important class of methods for crystal image analysis is variation based. General variational methods for texture
classification and segmentation have been extensively studied (see \cites{Meyer:2001, Chan:01, Vese:02, Sandberg:02, Vese:03,
  Aujol2006}, for example).  

The method in \cite{Berkels:08} proposed to
segment crystal images into disjoint regions with different constant
crystal rotations using the Chan-Vese level-set approximation in
\cite{Chan:01} of the piecewise constant Mumford-Shah segmentation in
\cite{Mumford}. The method involves search for a global deformation
$\phi:\Omega\rightarrow\RR^2$ acting on all grains. To speed up the
expensive optimization in \cite{Berkels:08}, the authors in
\cite{Berkels:10} proposed a convex relaxation via functional lifting
and the authors in \cite{Strekalovskiy:11} proposed a more efficient
version by penalizing the segmentation interfaces according to jumps
in crystal rotations. Although a corresponding GPU implementation of
these methods is very fast, a bottleneck still exists due to the large
memory cost coming from the additional dimension for the functional
lifting process.

More recently, Matt Elsey and Benedikt Wirth
\cites{ElseyWirth:13,ElseyWirth:MMS} proposed a variational model based
on finding a tensor map, the gradient of the inverse deformation
$\grad (\phi^{-1}):\Omega\rightarrow\RR^{2\times 2}$ and developed an
efficient $L^1-L^2$ regularization scheme. Crystal defects, rotations,
grain boundaries and strain can be recovered by the information hidden
in $\grad (\phi^{-1})$. In addition, a corresponding GPU
implementation is proposed and it shortens the runtime for a $1024^2$
image to about $40$ seconds.

Another class of methods for texture classification and segmentation are based on a local, direction sensitive frequency analysis \cites{Unser:95, SingerSinger:06}. \cite{Unser:95} constructed an over-complete wavelet frame for texture feature extraction and segmented textures in a reduced feature space by clustering techniques. However, the frame is not sensitive to crystal rotations and local defects. Hence, it cannot distinguish two grains with a small angle boundary and cannot detect local defects. This method is neither capable of providing estimates of crystal deformations. The work  \cite{SingerSinger:06} develops a heuristic method that uses a ``wavelet like'' patch according to a given reference crystal and quantify the similarity between local crystal patches with the reference. By looking up a prefabricated table of crystal rotation angles and their corresponding similarity, crystal rotations of each crystal patch can be estimated. However, in the case of deformed crystals, this method may be problematic. 

The method to be presented in this paper follows a different spirit from those methods. It is based on a novel model characterizing deformed periodic textures using $2D$ general intrinsic mode type functions and an efficient phase space representation method, $2D$ synchrosqueezed transforms proposed recently. An analytic characterization of periodic textures allows rigorous analysis and indeed it is proved that $2D$ synchrosqueezed transforms can estimate the local wave vectors of underlying wave-like components of textures (grains in this paper) precisely under certain conditions. Most of all, non-linear deformations of crystals are available by solving a simple linear system provided by local wave vectors.

\medskip

The rest of this paper is organized as follows. In Section \ref{sec:model}, we introduce a crystal image model on the microscopic length scale based on $2D$ general intrinsic mode type functions and prove that $2D$ synchrosqueezed transforms are able to estimate the local properties of the underlying wave-like components of general intrinsic mode type functions. In Section \ref{sec:alg}, two efficient algorithms based on $2D$ discrete band-limited synchrosqueezed transforms are proposed to detect crystal defects, estimate crystal rotations and elastic deformations. In Section \ref{sec:examples}, several numerical examples of synthetic and real crystal images are provided to demonstrate the robustness and the reliability of our methods. Finally, we conclude with some discussion on future works in Section \ref{sec:conclusion}.

\section{Crystal image models and theory}
\label{sec:model}

In this section, at first, we describe a new model to characterize atomic crystal images. Inspired by the periodicity of crystal images, $2D$ general intrinsic mode type functions are defined as a key ingredient of the characterization of a perfect crystal image. Second, $2D$ synchrosqueezed transforms are briefly recalled to analyze underlying wave-like components in $2D$ general intrinsic mode type functions. We will prove that the $2D$ synchrosqueezed transforms accurately estimate the local wave vectors of these wave-like components, and hence provides a useful tool for crystal image analysis.  

\subsection{$2D$ general intrinsic mode type functions}

A perfect crystal image, i.e., a single undeformed grain without defects, is characterized by a periodic function with two space
variables. We will limit ourselves to simple crystals (Bravais
lattices). In $2D$ space domain, there are five kinds of Bravais
lattices: oblique, rectangular, centered rectangular (rhombic),
hexagonal, and square \cite{Bravais}. The lattices and corresponding
unit cells are shown in Figure \ref{fig:Bravais}.
\begin{figure}[ht!]
  \begin{center}
    \begin{tabular}{c}
      \includegraphics[height=3.2in]{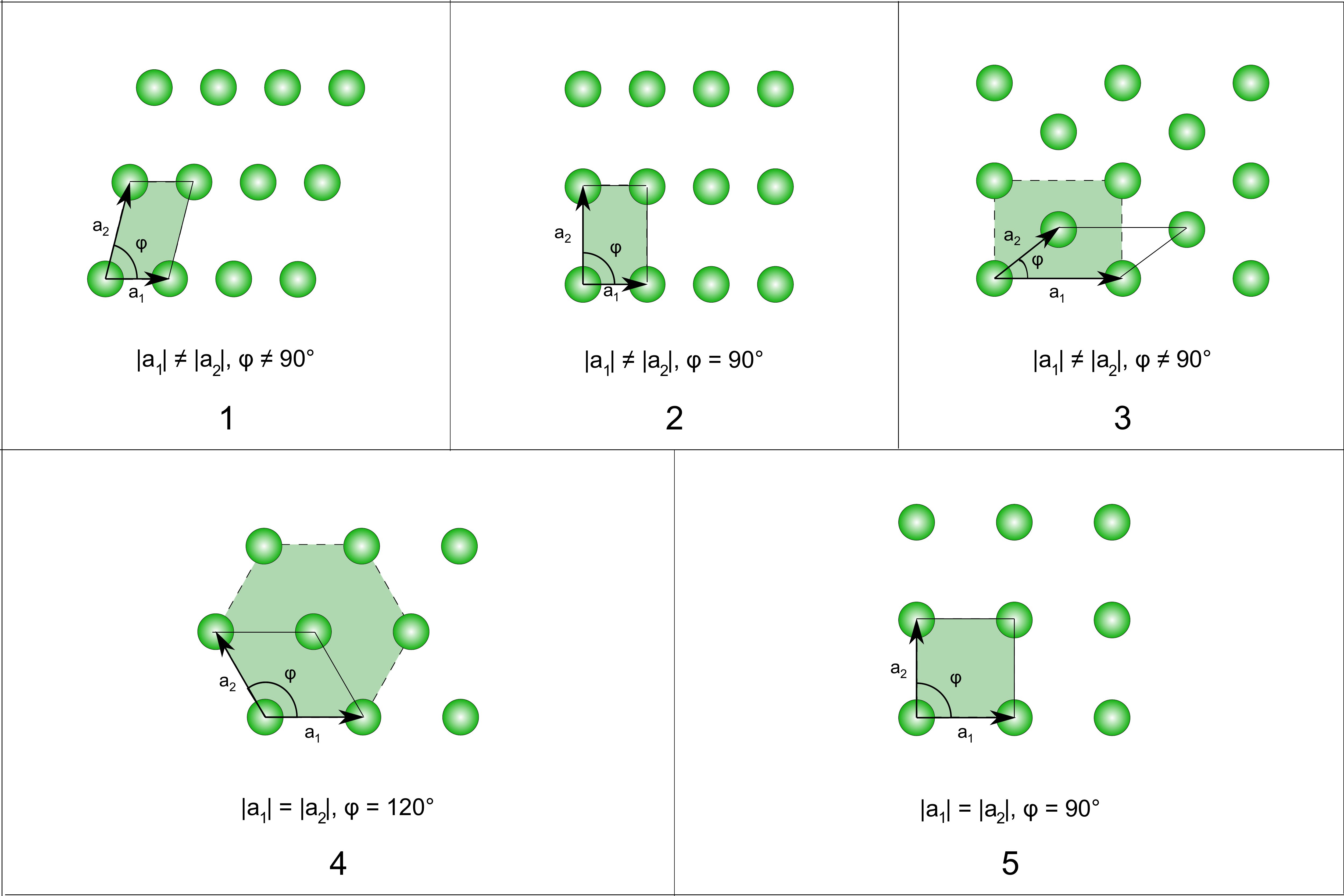}
    \end{tabular}
  \end{center}
  \caption{Five fundamental $2D$ Bravais lattices: 1 oblique, 2 rectangular, 3 centered rectangular (rhombic), 4 hexagonal, and 5 square. Courtesy of Wikipedia.}
  \label{fig:Bravais}
\end{figure}
For each lattice type, through an affine transform, we can transform the unit cell to a square. As an example, for the hexagonal lattice, the transform is given by 
\begin{equation*}
  x \mapsto F x; \quad F = 
  \begin{pmatrix}
    1 & -\frac{\sqrt{3}}{3} \\ 
    0 & \frac{2\sqrt{3}}{3}
  \end{pmatrix}. 
\end{equation*}
Hence, by introducing the matrix $F$, we can set up a reference
configuration function as
\[
f(x)=\alpha S(2\pi F x)+c,
\]
for $x\in \RR^2$. Here, $S(x)$ is a $2\pi$ periodic general shape
function (the rigorous definition is given later), which has a unit $L^2([-\pi, \pi]^2)$-norm and a zero mean. $\alpha$ and $c$ are two parameters. 

Allowing a rotation and a translation, a crystal image function for an undeformed grain is then modeled by 
\[
f(x)=\alpha S(2\pi N F(R_{\theta} x+z))+c, 
\]
where $N$ is the reciprocal of the
lattice parameter, $R_{\theta}$ is the rotation matrix
\[R_\theta=\left( \begin{array}{cc}
\cos\theta & -\sin\theta\\
\sin\theta& \cos\theta\end{array}\right)\]
corresponding to a rotation angle $\theta$, 
and $z\in \RR^2$ gives the translation. 
In the case of multi-grains, it is expected that
\[
f(x)=\sum_{k=1}^M \chi_{\Omega_k} (x)\left( \alpha_k S_k\left(2\pi N_k F_k (R_{\theta_k} x+z_k)\right)+c_k\right),
\]
where $\chi_{\Omega_k}(x)$ is the indicator function defined as
  \begin{equation}
    \chi_{\Omega_k}(x)=
    \begin{cases}
      1,
      & x\in\Omega_k\\
      0, 
      & \text{otherwise},
    \end{cases}
  \end{equation}
and $\Omega_k$ is the domain of the $k$th grain with $\Omega_k\cap \Omega_j=\emptyset$, if $k\neq j$. With these notations, grain boundaries are  interpreted as $\cup\partial \Omega_k$ (in real crystal images, grain boundaries would be a thin transition region instead of a sharp boundary $\cup\partial \Omega_k$). In the presence of local defects, e.g., an isolated defect and a terminating line of defects, $\cup\partial \Omega_k$ may include irregular boundaries and may contain point boundaries inside $\cup\Omega_k$.

Considering an uneven distribution of atoms on the mesoscopic length scale and possible reflection of light when generating crystal images, the amplitudes $\alpha_k$ and the global trends $c_k$ are assumed to be smooth functions $\alpha_k(x)$ and $c_k(x)$, respectively, in the domain $\Omega_k$. 

Notice that the rotation matrix $R_{\theta_k}$ and the translation position $z_k$ act as a linear transformation $\psi_k$ from $x$ to $\psi_k(x)=R_{-\theta_k}(x-z_k)$. In a more complicated case, a smooth non-linear deformation $\psi_k:\RR^2\to \RR^2$ transferring an atom from position $x$ to $\psi_k(x)$ is introduced. Let $\phi_k(x)=\psi_k^{-1}(x)$ defined in $\Omega_k$, then the crystal image function becomes 
\begin{equation}
f(x)=\sum_{k=1}^M \chi_{\Omega_k} (x)\left( \alpha_k(x) S_k\left(2\pi N_k F_k \phi_k(x)\right)+c_k(x)\right).
\label{eqn:crystal}
\end{equation}
This motivates the definition of $2D$ general shape functions and $2D$ general intrinsic mode type functions as follows.

\begin{defn}[$2D$ general shape function]
\label{def:GSF}

The $2D$ general shape function class ${S}_M$ consists of periodic functions $S(x)$ with a periodicity $(2\pi,2\pi)$, a unit $L^2([-\pi,\pi]^2)$-norm, and an $L^\infty$-norm bounded by $M$ satisfying the following conditions:
\begin{enumerate}
\item The $2D$ Fourier series of $S(x)$ is uniformly convergent;
\item $\sum_{n \in \ZZ^2} |\widehat{S}(n)|\leq M$ and $\widehat{S}(0,0)=0$;
\item Let $\Lambda$ be the set of integers $\{|n_1|\in \NN: \widehat{S}(n_1,n_2)\neq 0\text{ or }\widehat{S}(n_2,n_1)\neq 0\text{ for some }n_2\in\ZZ\}$. 
The greatest common divisor of all the elements in $\Lambda$ is $1$.
\end{enumerate}
\end{defn}
The requirement that $\widehat{S}(0,0)=0$, which is equivalent to a zero mean over $[-\pi,\pi]^2$, guarantees a well-separation between the oscillatory part $\alpha_k(x) S_k\left(2\pi N_k F_k \phi_k(x)\right)$ and the smooth trend $c_k(x)$ in \eqref{eqn:crystal}, when $N_k$ is sufficiently large. The third condition implies the uniqueness of similar oscillatory patterns in ${S}_M$ up to a scaling, i.e., if $S(x)\in{S}_M$, $S(Nx)\notin{S}_M$ for any positive integer $N>1$. 

\begin{defn}[$2D$ general intrinsic mode type function (GIMT)]
  \label{def:GIMTF}
  A function $f(x)=\alpha(x)s(2\pi N F \phi(x))$ is a $2D$ GIMT of type $(M,N,F)$, if $S(x)\in {S}_M$, $\alpha(x)$ and $\phi(x)$ satisfy the conditions below.
\begin{align*} 
&\alpha(x)\in C^\infty, \quad |\grad\alpha|\leq M, \quad 1/M \leq \alpha\leq M, \\
 &\phi(x)\in C^\infty, \quad 1/M \leq \left|  \nabla (n^{\TT} F \phi) / \abs{n^{\TT} F} \right| \leq M,\quad \text{and}\\
  &\left| \nabla^2 (n^{\TT} F \phi) / \abs{n^{\TT}F} \right|\leq M, \quad \forall n \in \ZZ^2 \quad \text{s.t.}\quad \widehat{s}(n)\neq 0.
   \end{align*}
\end{defn}

Hence, in the domain $\Omega_k$ of each grain, the crystal image is a superposition of a $2D$ general intrinsic mode type function and a smooth trend. Applying the $2D$ Fourier series of each $2D$ general shape function $S_k(x)$, it holds that
\begin{equation}
\begin{aligned}
  f(x) &=\sum_{k=1}^M \chi_{\Omega_k} (x)\left( \alpha_k(x) S_k\left(2\pi    N_kF_k\phi_k(x)\right)+c_k(x)\right) \\
  &=\sum_{k=1}^M \chi_{\Omega_k} (x)\left(
    \sum_{n \in \ZZ^2}\alpha_k(x)
    \widehat{S_k}(n)e^{2\pi iN_k n^{\TT} F_k
      \phi_k(x)}+c_k(x)\right),
\label{eqn:imagefunction}
\end{aligned}
\end{equation}
In the domain $\Omega_k$, each underlying wave-like component \[\alpha_k(x) \widehat{S_k}(n)e^{2\pi iN_k n^{\TT} F_k \phi_k(x)}\] is an intrinsic mode type function studied in \cites{YangYing:2013,YangYing:preprint}. Hence, if they satisfy the well-separation condition defined in \cites{YangYing:2013,YangYing:preprint} and each $N_k$ is large enough, the $2D$ synchrosqueezed wave packet transforms and the synchrosqueezed curvelet transforms are expected to estimate the local wave vectors $N_k \grad (n^{\TT} F_k \phi_k(x))$ accurately for $x$ away from $\partial\Omega_k$.
%

Based on the estimates of local wave vectors, it is possible to define and analyze crystal rotations as follows.

\begin{defn} 
  \label{def:LRF}
  Given a reference configuration $s(2\pi N Fx)$, a deformation
  $\phi(x)$, an amplitude function $\alpha(x)$ and a
  trend function $c(x)$ such that $f(x)=\alpha(x)s(2\pi N
  F\phi(x))+c(x)$ is a $2D$ GIMT of type $(M,N,F)$, suppose
  $\widehat{S}(n)\neq 0$, then the reference configuration has a local
  wave vector $v(n)=N n^{\TT} F $ and $f(x)$ has a local wave vector
  $v_\phi =N n^{\TT} F \grad \phi(x)$. The local rotation function of
  $f(x)$ with respect to $v(n)$ is defined as
  \[
  \beta(f)(x)= \arg(v_\phi(x))-\arg(v(n)),
  \]
  where $\arg(v)$ means the argument of a vector $v$.
\end{defn}
In the case of an undeformed crystal image $f(x)=\alpha(x)s(2\pi N F (R_\theta x+z))+c(x)$, $\phi(x)=R_\theta x+z$ is a composition of a rotation and a translation. Then the local rotation function $\beta(f)(x)=\theta$ with respect to any local wave vector $v(n)$. This agrees with an intuition of a global crystal rotation. However, a global crystal rotation is not well defined in a real crystal image due to a non-linear crystal deformation. This motivates the definition of a local crystal rotation in Definition \ref{def:LRF}. Because the non-linear deformation $\phi(x)$ is a smooth function with $\det\left(\grad\phi(x)\right)\approx 1$, local rotation functions $\beta(f)(x)$  vary smoothly and are approximately the same with respect to any local wave vector.

\subsection{$2D$ synchrosqueezed transforms}
For the sake of convenience, we adopt the name $2D$ synchrosqueezed transforms instead of specifying $2D$ synchrosqueezed wave packet transform or $2D$ synchrosqueezed curvelet transforms. Actually, when the scaling parameters $s$ and $t$ in the $2D$ synchrosqueezed curvelet transforms are equal, these transforms become $2D$ synchrosqueezed wave packet transforms. For the same reason, we prefer a uniform name general wave packet transforms, instead of the wave packet transforms in \cite{YangYing:2013} and the general curvelet transforms in \cite{YangYing:preprint}. Before a brief review of these transforms, here are some notations. 
\begin{enumerate}
\item The scaling matrix
\[A_a=\left( \begin{array}{cc}
a^t&0\\
0&a^s 
\end{array}\right),\]
where $a$ is the distance from the center of one general wave packet to the origin in the Fourier domain.
\item A unit vector $e_\theta=(\cos \theta,\sin \theta)^T$ with a rotation angle $\theta$.
\item $\theta_\alpha$ represents the argument of a given vector $\alpha$.
\item $w(x)$ of $x\in \RR^2$ denotes a mother wave packet, which is in the Schwartz class and has a non-negative, radial, real-valued, smooth Fourier transform
$\widehat{w}(\xi)$ with a support equal to a ball $B_d(0)$ centered at the origin with a radius $d\leq 1$ in the Fourier
domain. The mother wave packet is required to obey the admissibility condition: $\exists 0<c_1<c_2<\infty$ such that
\[
c_1\leq \int_0^{2\pi} \int_1^\infty a^{-(t+s)} |\widehat{w}(A^{-1}_a R^{-1}_\theta (\xi-a\cdot e_\theta))|^2 a\ud a\ud\theta \leq c_2
\]
for any $|\xi|\geq 1$.
\end{enumerate}
With the notations above, it is ready to define a family of general wave packets through scaling, modulation, and translation as follows, controlled by geometric parameters $s$ and $t$.
\begin{defn}
  \label{def:GC2D}
For $\frac{1}{2}<s\leq t<1$, define $\widehat{w_{a\theta b}}(\xi)=\widehat{w}(A^{-1}_a R^{-1}_\theta (\xi-a\cdot e_\theta))e^{-2\pi ib\cdot\xi}a^{-\frac{t+s}{2}}$ as a general wave packet in the Fourier domain. Equivalently, in the space domain, the corresponding general wave packet is
\begin{align*}
w_{a\theta b}(x)&=\int_{R^2}\widehat{w}(A^{-1}_a R^{-1}_\theta (\xi-a\cdot e_\theta)) e^{-2\pi ib\cdot \xi}e^{2\pi i\xi\cdot x}a^{-\frac{t+s}{2}}\ud\xi\nonumber\\
&=a^{\frac{t+s}{2}}\int_{R^2}\widehat{w}(y)e^{-2\pi ib\cdot(R_\theta A_a y+a\cdot e_\theta)}e^{2\pi ix\cdot(R_\theta A_a y+a\cdot e_\theta)}\ud y\nonumber\\
&=a^{\frac{t+s}{2}}e^{2\pi ia(x-b)\cdot e_\theta}w(A_a R^{-1}_\theta(x-b)).
\end{align*}
In such a way, a family of general wave packets $\{w_{a\theta b}(x),a\in [1,\infty),\theta\in[0,2\pi),b\in \RR^2\}$ is constructed. 
\end{defn}

Similar to classical time-frequency transforms, the general wave packet transform is defined to be the inner product of a given signal and each general wave packet as follows.
\begin{defn}
  \label{def:GCT}
The general wave packet transform of a function $f(x)$ is a function 
\begin{align*}
  W_f(a,\theta,b)&=\langle w_{a\theta b},f\rangle = \int_{\RR^2}\overline{w_{a\theta b}(x)}f(x)\ud x\\
  &=\langle\widehat{w_{a\theta
      b}},\widehat{f}\rangle=\int_{\RR^2}\overline{\widehat{w_{a\theta
        b}(\xi)}}\widehat{f}(\xi)\ud\xi
\end{align*}
for $a\in [1,\infty)$, $\theta\in[0,2\pi)$, $b\in \RR^2$.
\end{defn}

\begin{defn} 
  \label{def:LW}
  The local wave-vector estimation of a function $f(x)$ at
  $(a,\theta,b)$ is
  \begin{equation}
    v_f(a,\theta,b)=\frac{ \grad_b W_f(a,\theta,b) }{ 2\pi i W_f(a,\theta,b)} \label{eq:IWE}
  \end{equation}
  for $a\in [1,\infty)$, $\theta\in [0,2\pi)$, $b\in\RR^2$ such that $W_f(a,\theta,b)\not=0$.
\end{defn}

In \cites{YangYing:2013,YangYing:preprint}, it has been proved that $v_f(a,\theta,b)$ accurately estimates local wave-vectors independently of the amplitude functions $\alpha_k$ or the position $b$ from a finite superposition of wave-like components without boundaries. Hence, if the coefficients with the same $v_f$ are reallocated together, then the nonzero energy would be concentrating around local wave-vectors of $f(x)$. Mathematically speaking, the synchrosqueezed energy distribution is defined as follows.
\begin{defn}
  Given $f(x)$, $W_f(a,\theta,b)$, and $v_f(a,\theta,b)$, the synchrosqueezed energy
  distribution $T_f(v, b)$ is 
  \begin{equation}
    T_f(v,b) = \int |W_f(a,\theta,b)|^2 \delta(\Re v_f(a,\theta,b)-v) a\ud a\ud\theta \label{eq:SED}
  \end{equation}
  for $v\in \RR^2$, $b\in \RR^2$.
\end{defn} 

If the Fourier transform $\widehat{f}(\xi)$ vanishes for $|\xi|< 1$, one can check the following $L^2$-norms equivalence up to a uniform constant factor following the proof of \cite{CandesII}*{Theorem $1$}, i.e.,
\begin{equation}
\label{eqn:normeq}
\int|W_f(a,\theta,b)|^2 a\ud a\ud\theta \ud b = \Theta\left( \int |f(x)|^2 \ud x\right).
\end{equation}

Our goal is to apply and adapt the $2D$ synchrosqueezed transforms to analyze the $2D$ general intrinsic mode type functions in the image function \eqref{eqn:imagefunction}. This problem is similar to but different from the $1D$ general mode decomposition problems studied in \cites{Hau-Tieng:2013,Yang:preprint}, where $1D$ synchrosqueezed transforms are applied to estimate the instantaneous properties of $1D$ general intrinsic mode type functions. Generalizing the conclusions in \cites{Yang:preprint,YangYing:preprint}, the theorem below shows that the $2D$ synchrosqueezed transforms precisely estimate the local wave vectors of the wave-like components in \eqref{eqn:imagefunction} at the points away from boundaries.

\begin{theorem}
  \label{thm:main}
  For a $2D$ general intrinsic mode type function $f(x)$ of type $(M,N,F)$ with $\left|F\right|\geq 1$, any $\eps>0$ and any $r>1$, we define
  \begin{equation*}
  R_{\eps} = \left\{(a,\theta,b): |W_f(a,\theta,b)|\geq a^{-\frac{s+t}{2}}\sqrt \eps,\quad a\leq 2MNr \right\}
  \end{equation*}
  and 
  \[
  Z_{n} = \left\{(a,\theta,b):\left|A^{-1}_a R^{-1}_\theta\left(a\cdot e_\theta-N \nabla(n^{\TT} F  \phi(b)) \right)\right|\leq d,\quad a\leq 2MNr \right\}
  \]
For fixed $M$, $r$, $s$, $t$, $d$, and $\epsilon$, there exists $N_0(M,r,s,t,d,\eps)>0$ such that for any $N>N_0(M,r,s,t,d,\eps)$ and a $2D$ GIMT $f(x)$ of type $(M,N,F)$, the following statements
  hold.
  \begin{enumerate}[(i)]
  \item $\left\{Z_n: \widehat{S}(n)\neq 0 \right\}$ are disjoint and $R_{\eps}
    \subset \bigcup_{\widehat{S}(n)\neq 0} Z_{n}$;
  \item For any $(a,\theta,b) \in R_{\eps} \cap Z_{n}$, 
    \[
    \frac{\left|v_f(a,\theta,b)-N \nabla(n^{\TT} F  \phi(b))  \right|}{ \left|N \nabla(n^{\TT} F  \phi(b))  \right|}\lesssim\sqrt \eps.
    \]
  \end{enumerate}
\end{theorem}

For simplicity, the notations $\Or(\cdot)$, $\lesssim$ and $\gtrsim$ are used when the implicit constants may only depend on $M$, $s$, $t$, $d$ and $K$. The proof of the theorem is similar to those theorems in \cites{YangYing:preprint,Yang:preprint} and relies on two lemmas in \cite{YangYing:preprint}. Although the parameter $d$ is set to be $1$ in \cite{YangYing:preprint}, it is easy to extend these two lemmas for a general value of $d$. We state the generalizations here, leaving the proofs to the reader.

\begin{lemma}
  \label{lem:A}
Suppose $f(x)=\sum_{k=1}^K f_k(x)=\sum_{k=1}^{K}e^{-(\phi_k(x)-c_k)^2/\sigma_k^2}\alpha_k(x)e^{2\pi iN \phi_k(x)}$ is a well-separated superposition of type $(M,N,K)$ (refer to \cite{YangYing:preprint}). Set \[\Omega=\left\{ (a,\theta):a\in\left(\frac{N}{2M},2MN\right),\exists k\ s.t.\ \left|\theta_{\grad\phi_k(b)}-\theta\right|<\theta_0\right\},\] where $\theta_0=\arcsin((\frac{M}{N})^{t-s})$. For any $\eps>0$, there exists $N_0(M,s,t, d, \eps)$ such that the following estimation of $W_f(a,\theta,b)$ holds for any $N>N_0$.
 \begin{enumerate}[(1)]
  \item If $(a,\theta)\in \Omega$,\[W_f(a,\theta,b)=a^{-\frac{s+t}{2}} \left( \sum_{k:\ |\theta_{\grad\phi_k(b)}-\theta|<\theta_0} f_k(b)\widehat{w}\left(A^{-1}_a R^{-1}_\theta \left(a\cdot e_\theta-N\grad\phi_k(b)\right)\right)+\Or(\eps) \right);\]
  \item Otherwise,\[W_f(a,\theta,b)=a^{-\frac{s+t}{2}}\Or(\eps).\]
  \end{enumerate}
\end{lemma}

\begin{lemma}
  \label{lem:B}
  Suppose $f(x)=\sum_{k=1}^K f_k(x)=\sum_{k=1}^{K}e^{-(\phi_k(x)-c_k)^2/\sigma_k^2}\alpha_k(x)e^{2\pi iN \phi_k(x)}$ is a well-separated superposition of type $(M,N,K)$. For any $\eps>0$, there exists $N_0(M,s,t, d, \eps)$ such that the following estimation of $\grad_b W_f(a,\theta,b)$ holds for any $N>N_0$.
  \begin{eqnarray*}
    \grad_b W_f(a,\theta,b)=a^{-\frac{s+t}{2}}\left(2\pi i N \sum_{k:\ |\theta_{\grad\phi_k(b)}-\theta|<\theta_0}\grad\phi_k(b)f_k(b)\widehat{w}\left(A^{-1}_a R^{-1}_\theta(a\cdot e_\theta-N\grad\phi(b))\right) +\Or(\eps)\right),
  \end{eqnarray*}
when \[(a,\theta)\in\Omega=\left\{ (a,\theta):a\in\left(\frac{N}{2M},2MN\right),\exists k\ s.t.\ \left|\theta_{\grad\phi_k(b)}-\theta\right|<\theta_0\right\}.\]
\end{lemma}

In the scope of this paper, we have $\sigma_k = \infty$ for all $k$. Now, we are ready to prove Theorem \ref{thm:main}.
\begin{proof}[Proof of Theorem~\ref{thm:main}]
By the uniform convergence of the $2D$ Fourier series of general shape functions, we have
\begin{equation*}
W_f(a,\theta,b)= \sum_{n\in \ZZ^2}W_{f_{n}}(a,\theta,b),
  \end{equation*}
where $f_{n}(x)=\widehat{S}(n)\alpha(x)e^{2\pi iN n^{\TT} F \phi(x)}$. 
Introduce the short hand notation,  $\wt{\phi}_{n}(x)=n^{\TT} F \phi(x) / \abs{n^{\TT} F}$, then 
\begin{equation*}
  f_{n}(x)=\widehat{S}(n)\alpha(x)e^{2\pi iN \abs{n^{\TT}F} \wt{\phi}_{n}(x)}.
\end{equation*}
By the property of $2D$ general intrinsic mode functions, $f_{n}(x)$ is a well-separated superposition of type $(M,N\abs{n^{\TT}F},1)$ defined in \cite{YangYing:preprint}.

For each $n$, we estimate $W_{f_{n}}(a,\theta,b)$. By Lemma \ref{lem:A}, there exists a uniform $N_1(M,s,t,d,\eps)$ independent of $n$ such that, if $N\abs{n^{\TT}F}>N_1$, 
\begin{equation*}
W_{f_{n}}(a,\theta,b)=a^{-\frac{s+t}{2}} \left( f_{n}(b)\widehat{w}\left(A^{-1}_a R^{-1}_\theta \left(a\cdot e_\theta-N \abs{n^{\TT} F} \grad\wt{\phi}_{n}(b)\right)\right)+\left| \widehat{S}(n)\right|\Or(\eps) \right)
\end{equation*}
for $(a,\theta)\in \Omega_{n}$, and
\begin{equation*}
W_{f_{n}}(a,\theta,b)=a^{-\frac{s+t}{2}}\left| \widehat{S}(n)\right|\Or(\eps)
\end{equation*}
for $(a,\theta)\not\in \Omega_{n}$. Here 
\[
\Omega_{n}=\left\{ (a,\theta):a\in\left(\frac{N\abs{n^{\TT}F}}{2M},2MN\abs{n^{\TT}F}\right),\left|\theta_{\grad\wt{\phi}_{n}(b)}-\theta\right|<\theta_0\right\},
\]
and $\theta_0=\arcsin((\frac{M}{N\abs{n^{\TT}F}})^{t-s})$. Let $\Gamma_{a\theta}=\{n \in \ZZ^2 :(a,\theta)\in\Omega_{n}\}$, then
\begin{align*}
W_{f}(a,\theta,b)&=a^{-\frac{s+t}{2}} \left(\sum_{n\in \Gamma_{a\theta}} f_{n}(b)\widehat{w}\left(A^{-1}_a R^{-1}_\theta \left(a\cdot e_\theta-N\abs{n^{\TT}F}\grad\wt{\phi}_{n}(b)\right)\right)+\sum_{n}\left| \widehat{S}(n)\right|\Or(\eps) \right)\\
&=a^{-\frac{s+t}{2}} \left(\sum_{n \in \Gamma_{a\theta}} f_{n}(b)\widehat{w}\left(A^{-1}_a R^{-1}_\theta \left(a\cdot e_\theta-N\abs{n^{\TT}F}\grad\wt{\phi}_{n}(b)\right)\right)+\Or(\eps) \right).
\end{align*}

Notice that for any $n \neq \wt{n}$, the distance between the local wave vectors $N\abs{n^{\TT} F}\grad\wt{\phi}_{n}(b)$ and $N\abs{\wt{n}^{\TT} F}\grad\wt{\phi}_{\wt{n}}(b)$ are bounded below. In fact,
\begin{align*}
\left|N \abs{n^{\TT} F} \grad\wt{\phi}_{n}(b)-N\abs{\wt{n}^{\TT} F}\grad\wt{\phi}_{\wt{n}}(b)\right|
&=N\left|(n -\wt{n })^{\TT} F \grad\phi(b) \right|\\
&\geq\frac{N}{M}\abs{(n - \wt{n})^{\TT} F}
\geq\frac{N}{M}.
\end{align*}
The first inequality above is due to the definition of $2D$ general intrinsic mode type function. Observe that the support of a general wave packet centered at $(a\cos(\theta),a\sin(\theta))$ is within a disk with a radius of length $da^t$. Because the range of $a$ of interest is $a\leq 2MNr$, the general wave packets of interest have supports of size at most $2d(2MNr)^t$. Hence, if $\frac{N}{M}\geq 2d(2MNr)^t$, which is equivalent to $N\geq (2^{1+t}M^{1+t}dr^t)^{\frac{1}{1-t}}$, then for each $(a,\theta,b)$ of interest, there is at most one $n \in \ZZ^2$ such that
\[
\left|A^{-1}_a R^{-1}_\theta\left(a\cdot e_\theta-N \nabla(n^{\TT} F  \phi(b))  \right)\right|\leq d.
\]
This implies that $\{Z_{n}\}$ are disjoint sets. Notice that $\widehat{w}(x)$ decays when $|x|\geq d$. The above statement also indicates that there is at most one $n\in\Gamma_{a\theta}$ such that 
\[f_{n}(b)\widehat{w}\left(A^{-1}_a R^{-1}_\theta \left(a\cdot e_\theta-N \abs{n^{\TT} F} \grad\wt{\phi}_{n}(b)\right)\right)\neq 0.\]
Hence, if $(a,\theta,b)\in R_\eps$, there must be some $n$ such that $\widehat{S}(n)\neq 0$ and
\begin{eqnarray}
\label{eqn:Wf}
W_{f}(a,\theta,b)=a^{-\frac{s+t}{2}} \left(f_{n}(b)\widehat{w}\left(A^{-1}_a R^{-1}_\theta \left(a\cdot e_\theta-N\abs{n^{\TT} F}\grad\wt{\phi}_{n}(b)\right)\right)+\Or(\eps) \right).
\end{eqnarray}
By the definition of $Z_{n}$, we see $(a,\theta,b)\in Z_{n}$. So, $R_{\eps} \subset \bigcup_{\widehat{S}(n)\neq 0} Z_{n}$ and $(1)$ is proved.

Now, we estimate $\grad_b W_f(a,\theta,b)$. Suppose $(a,\theta,b)\in R_{\eps} \cap Z_{n}$. Similarly to the estimate of $W_f(a,\theta,b)$, by Lemma \ref{lem:B}, there exits $N_2(M,s,t,d,\eps)$ such that if $N>N_2$ then
 \begin{align*}
   \grad_b & W_f(a,\theta,b)\\
 &= a^{-\frac{s+t}{2}}\left(2\pi i N \sum_{n\in\Gamma_{a\theta}}
\abs{n^{\TT} F}\grad\wt{\phi}_{n}(b)f_{n}(b)\widehat{w}\left(A^{-1}_a R^{-1}_\theta(a\cdot e_\theta-N\abs{n^{\TT} F} \grad\wt{\phi}_{n}(b))\right) +\Or(\eps)\right)\\
&= a^{-\frac{s+t}{2}}\left(2\pi i N \abs{n^{\TT} F}\grad\wt{\phi}_{n}(b)f_{n}(b)\widehat{w}\left(A^{-1}_a R^{-1}_\theta(a\cdot e_\theta-N\abs{n^{\TT}F}\grad\wt{\phi}_{n}(b))\right) +\Or(\eps)\right)
  \end{align*}
for the same $n$ in \eqref{eqn:Wf}.

Let $g=f_{n}(b)\widehat{w}\left(A^{-1}_a R^{-1}_\theta(a\cdot e_\theta-N\abs{n^{\TT}F}\grad\wt{\phi}_{n}(b))\right)$, then
\[
v_f(a,\theta,b)=\frac{N\abs{n^{\TT} F}\grad\wt{\phi}_{n}(b)g+\Or(\eps)}{g+\Or(\eps)}.
\]
Since $|W_f(a,\theta,b)|\geq a^{-\frac{s+t}{2}}\sqrt{\eps}$ for $(a,\theta,b)\in R_{\eps}$, then $|g|\gtrsim\sqrt{\eps}$. So 
\begin{align*}
\frac{\left|v_f(a,\theta,b)-N \nabla(n^{\TT} F  \phi(b)) \right|}{ \left|
N \nabla(n^{\TT} F  \phi(b)) \right|}
&=\frac{\left|v_f(a,\theta,b)-N\abs{n^{\TT} F} \grad\wt{\phi}_{n}(b)\right|}{\left|N\abs{n^{\TT} F}\grad\wt{\phi}_{n}(b)\right|}\\
&\lesssim \left|\frac{\Or(\eps)}{g+\Or(\eps)}\right|\\
&\lesssim\sqrt{\eps}.
\end{align*}
The above estimate holds for $N>N_0 = \max\{N_1,N_2, (2^{1+t}M^{1+t}dr^t)^{\frac{1}{1-t}}\}$. The proof is complete.
\end{proof}

Theorem \ref{thm:main} indicates that the underlying wave-like components $\widehat{S}(n)\alpha(x)e^{2\pi iN n^{\TT} F \phi(x)}$ of the $2D$ general intrinsic mode type function $f(x)=\alpha(x)S(2\pi N F \phi(x))$ are well-separated, if they are within a scale $a\leq 2MNr$ and $N$ is sufficiently large. By the definition of the synchrosqueezed energy distribution, $T_f(a,\theta,b)$ would concentrate around their local wave vectors $N \nabla(n^{\TT} F  \phi(b)) $. The energy around each local wave vectors reflects the energy of the corresponding wave-like component by the norm equivalence \eqref{eqn:normeq}.

In the crystal image analysis, grains would have local defects, irregular boundaries and smooth trends caused by reflection etc. 
First, it is important to know where we can keep the local wave vector estimates accurate. Suppose a $2D$ general intrinsic mode type function $\widehat{S}(n)\alpha(x)e^{2\pi iN n^{\TT} F\phi(x)}$ is defined in a domain $\Omega$ with a boundary $\partial\Omega$. The smallest scale used to estimate local wave vectors is of order $\frac{N}{M}$. Hence, the general wave packets used have supports of size at most $\frac{M^t}{dN^t}$ by $\frac{M^t}{dN^t}$. Let us denote a perfect interior of $\Omega$ as
\[
\wt{\Omega}=\left\{x\in\Omega:\left|x-y\right|>\frac{M^t}{dN^t},\forall y\in\partial\Omega \right\},
\]
then the estimates of local wave vectors remain accurate in $\wt{\Omega}$. As the numerical results in \cites{DaubechiesLuWu:11,YangYing:2013,YangYing:preprint,Yang:preprint} show, synchrosqueezed transforms can still approximately recover instantaneous or local properties near $\partial\Omega$.

The second concern is the influence of the smooth trends $c_k(x)$ in \eqref{eqn:crystal}. By the method of stationary phase, the Fourier transform $\widehat{c_k}(\xi)$ would decay quickly as $|\xi|$ increases. Hence, the influence of smooth trends is essentially negligible when the synchrosqueezed transforms are applied to estimate local wave vectors $N n^{\TT} F \grad\phi(x))$ with a sufficiently large wave number.

\section{Crystal defect analysis algorithms and implementations}
\label{sec:alg}

This section introduces several algorithms based on the features of crystal images and $2D$ synchrosqueezed transforms for a fast analysis of local defects, crystal rotations and deformations. As discussed in the introduction, we will focus on the analysis of an image with only one type of crystals, i.e., suppose
\begin{align*}
f(x)&=\sum_{k=1}^M \chi_{\Omega_k} (x)\left( \alpha_k(x) S\left(2\pi NF\phi_k(x)\right)+c_k(x)\right)\\
&=\sum_{k=1}^M \chi_{\Omega_k} (x)\left( \sum_{n\in\ZZ^2}\widehat{S}(n)\alpha_k(x) e^{2\pi iNn^{\TT}F\phi(x)}+c_k(x)\right).
\end{align*}

In the $2D$ space domain, there are five kinds
of 
Bravais lattices, which are oblique, rectangular, centered rectangular
(rhombic), hexagonal, and square \cite{Bravais} as shown in Figure
\ref{fig:Bravais}. Accordingly, there are five kinds of $2D$ Fourier
power spectra $|\widehat{S}(n)|$. These spectra have a few dominant
wave vectors in terms of energy, i.e., a few $|\widehat{S}(n)|$ with
large values. At each interior point of a grain, a sufficiently
localized Fourier transform is able to recover an approximate
distribution of the $2D$ Fourier power spectrum, based on which the
corresponding reference configuration of this grain can be
specified. Hence, for the simplicity of a presentation, we will
restrict ourselves to the analysis of images of hexagonal crystals
(e.g.~Figure \ref{fig:Toy_fast_spec} (left)). A generalization of our
algorithms to other kinds of simple crystal images is straightforward.
In the case of a complex lattice, there might be numerous local wave
vectors with relatively large energy. Feature extraction and dimension
reduction techniques should be applied to provide a few crucial local
wave vectors. This would be an interesting future work.

We will introduce two fast algorithms for crystal image
analysis. Algorithm~$\ref{alg:undeformed}$ provides estimates of grain
boundaries and crystal rotations; while Algorithm~$\ref{alg:deformed}$
further identifies point defects, dislocations, and deformations.  To
make our presentation more transparent, the algorithms and
implementations are introduced with toy examples. For
Algorithm~$\ref{alg:undeformed}$, we will use the example in
Figure~\ref{fig:Toy_fast_spec} (left) which contains two undeformed
grains with a straight grain boundary. While
  Figure~\ref{fig:Toy_fast_spec} is a synthetic example, it
  illustrates nicely the key feature of atomic crystal images and the
  idea of local phase plane spectrum. In this example, the reciprocal
lattice parameter $N=120$ and the crystal rotations are given by $15$
and $52.5$ degrees on the left and right respectively.  We introduce
Algorithm~$\ref{alg:deformed}$ using strained examples of a small
angle boundary (see Figure \ref{fig:GB5GB8} (left)) and with some
isolated dislocations (see Figure \ref{fig:GB5GB8}
(right)). These are examples from phase field crystal simulations
  \cite{ElseyWirth:12}.

\subsection{Frequency band detection (bump detection)}
\label{sec:BumpDetection}

Typically, each grain 
\[
\chi_{\Omega_k} (x)\left( \alpha_k(x) S\left(2\pi NF\phi_k(x)\right)+c_k(x)\right)
\]
in a polycrystalline crystal image can be identified as a $2D$ general intrinsic mode type function of type $(M,N,F)$ with a small $M$ near $1$, unless the strain is too large. Hence, the $2D$ Fourier power spectrum of a multi-grain image would have several well-separated nonzero energy annuli centered at the origin due to crystal rotations (see an example shown in Figure \ref{fig:Toy_fast_spec} (middle)). Suppose the radially average Fourier power spectrum is defined as
\[
E(r)=\frac{1}{r}\int_0^{2\pi}\left|\widehat{f}(r,\theta)\right| d\theta,
\]
where $\widehat{f}(r,\theta)$ is the Fourier transform in the radial coordinate: $\widehat{f}(r,\theta)=\widehat{f}(\xi)$ and $\xi=(r\cos(\theta),r\sin(\theta))^T\in \RR^2$. Then $E(r)$ would have several well-separated energy bumps (see Figure \ref{fig:Toy_fast_spec} (right)) for the same reason.

As we can see in Figure \ref{fig:Toy_fast_spec} (middle), a hexagonal crystal image with a single grain \[f(x)=\alpha(x) S\left(2\pi NF\phi(x)\right)+c(x)\] has six dominant local wave vectors close to $v_j(\theta(x))$, $j=0,1,\dots,5$, which are the vertices of a hexagon centered at the origin in the Fourier domain, i.e.,
\[
v_j(\theta(x))=\left(N\cos(\theta(x)+\frac{j\pi}{3}),N\sin(\theta(x)+\frac{j\pi}{3})\right)^\TT,\quad j=0,1,\dots,5,
\]
for $\theta(x)\in[0,\frac{\pi}{3})$. Suppose $S\left(2\pi NF (x)\right)$ is a reference configuration, then the local rotation function with respect to each vertex $v_j(0)$ is
\begin{equation}
\label{eqn:LRF}
\beta(f)(x)\approx\arg(v_j(\theta(x)))-\arg(v_j(0))=\theta(x).
\end{equation}
Actually, $f(x)$ can approximately be considered as a rotated version of $\alpha S\left(2\pi NF (x+z)\right)+c$ by an angle $\theta(x)+\frac{k\pi}{3}$ for any $k\in\ZZ$ due to the crystal symmetry. The restriction $\theta(x)\in[0,\frac{\pi}{3})$ guarantees a unique local rotation function $\beta(f)(x)$. 

\begin{figure}[ht!]
  \begin{center}
    \hspace{-4em}
      \includegraphics[height=1.8in]{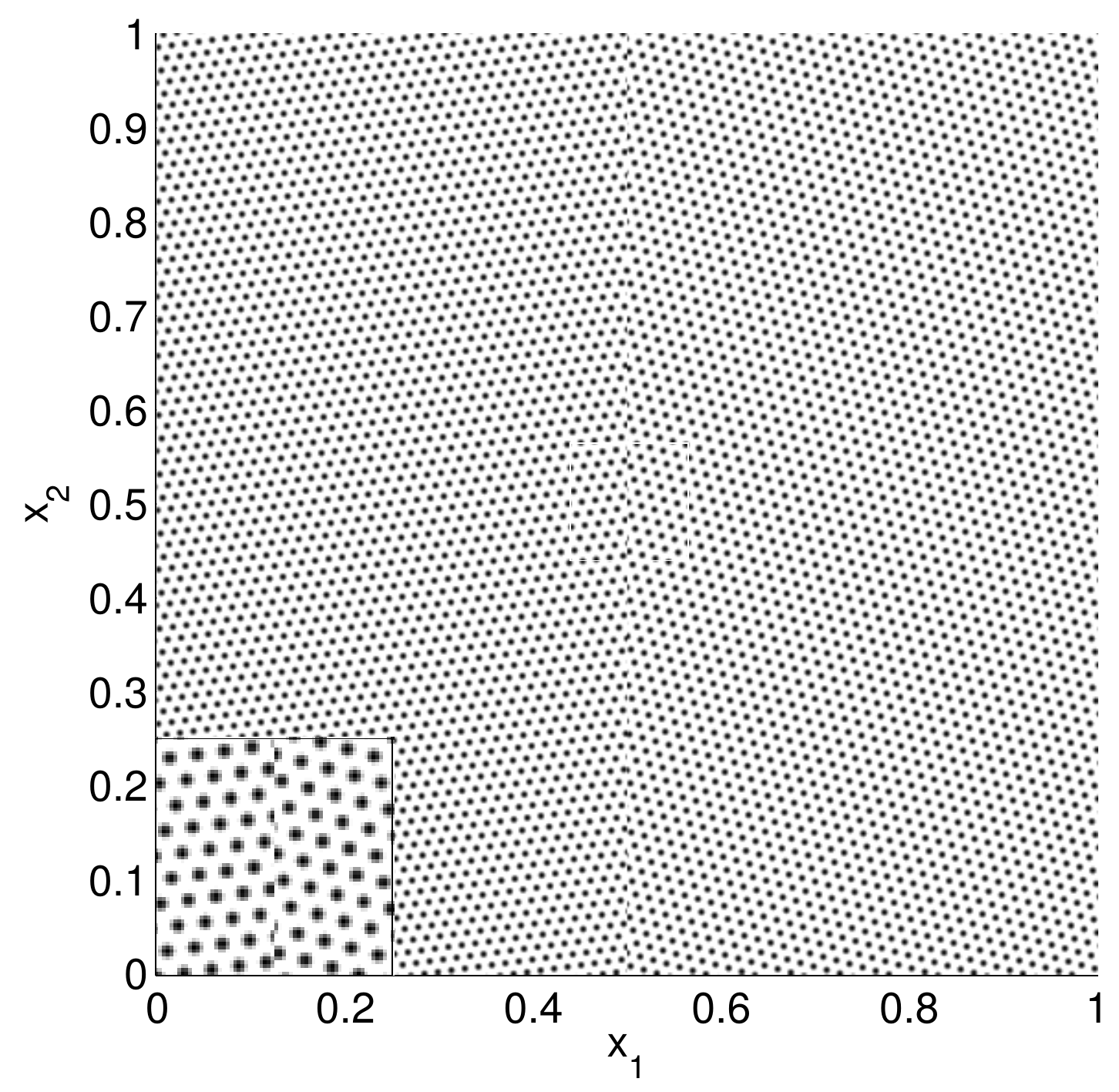} \; \includegraphics[height=1.8in]{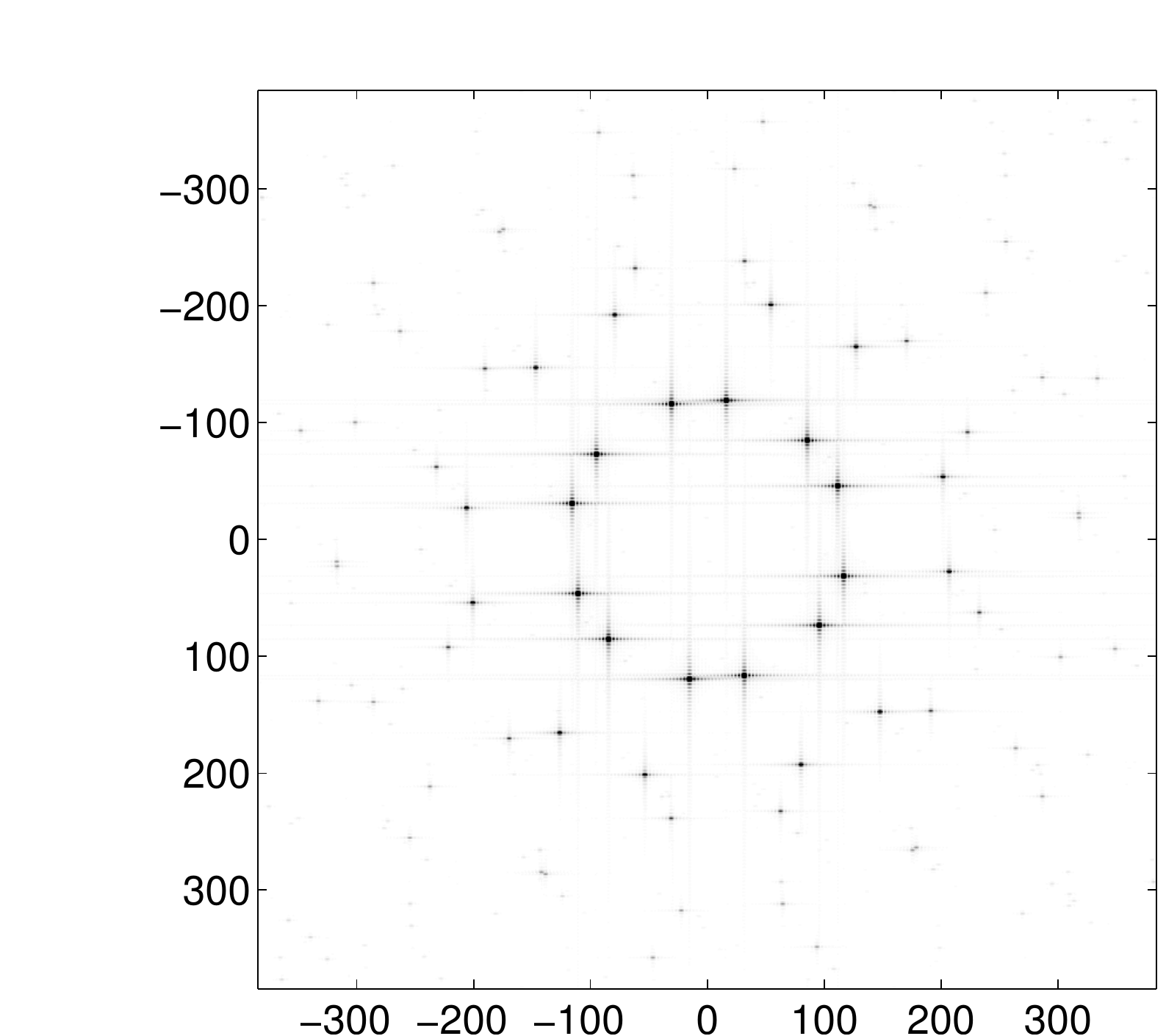} \includegraphics[height=1.8in]{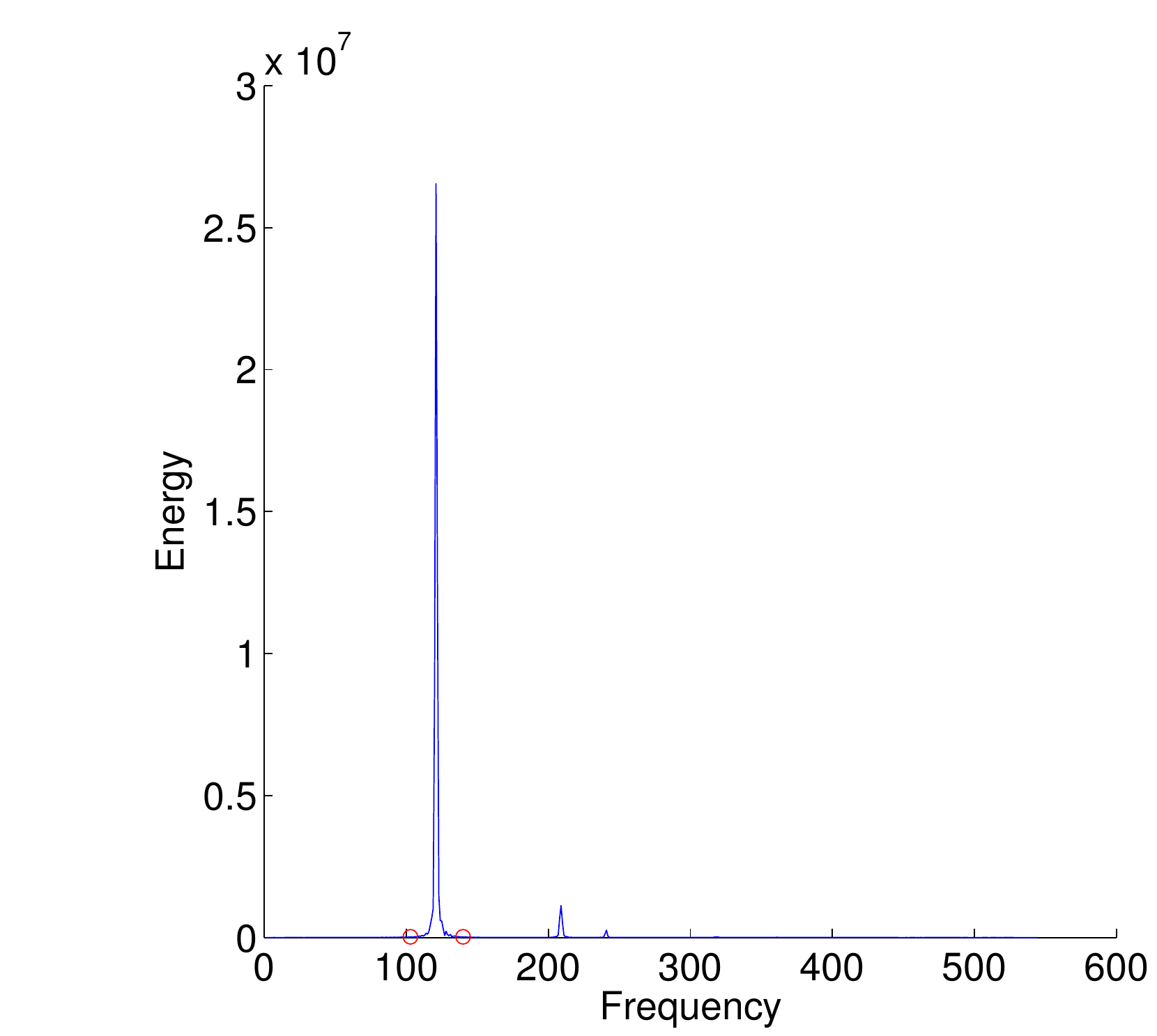}
  \end{center}
  \caption{Left: An undeformed example of two hexagonal grains with a vertical line boundary. Middle: Its Fourier power spectrum. The number of dominant local wave vectors is $12$ due to the superposition of spectrum of the two grains. Right: Its radially average Fourier power spectrum with the identified most dominant energy bump indicated by two red circles.}
  \label{fig:Toy_fast_spec}
\end{figure}

The discussion above shows that it is sufficient to compute the local rotation function using the dominant local wave vectors. To reduce the computational cost of $2D$ SST, the support of the most dominant energy bump in the radially average Fourier power spectrum should be identified (see Figure \ref{fig:Toy_fast_spec} (right)). Suppose the support (i.e. frequency band) is $[r_1,r_2]$. Then a band-limited fast $2D$ SST as introduced later is applied to estimate the local wave vectors with wave numbers in $[r_1,r_2]$.

In what follows, we would walk through the steps to determine the dominant frequency band $[r_1,r_2]$ with a time complexity $O(L^2\log L)$ for an $L$ by $L$ image. For simplicity, we consider images that are periodic over the unit square $[0,1)^2$ in $2D$. If it is not the case, the images will be periodized by padding zeros around the image boundary. Let
\[
X = \{ (n_1/L,n_2/L): 0 \le n_1, n_2<L,\, n_1,n_2 \in \ZZ\}
\]
be the $L\times L$ spatial grid at which crystal image functions are sampled.
The corresponding $L\times L$ Fourier grid is
\[
\Xi = \{ (\xi_1,\xi_2): -L/2\le \xi_1,\xi_2<L/2,\, \xi_1,\xi_2 \in \ZZ\}.
\]
For a function $f(x)\in \ell^2(X)$, we apply the $2D$ FFT to obtain $\widehat{f}(\xi)\in \ell^2(\Xi)$ first. Then the energy $\left| \widehat{f}(\xi)\right|$ should be stacked up and averaged to obtain the radially average Fourier power spectrum $E(r)$. To realize this step, a grid of step size $\Delta$ is generated to discretize the domain $[0,\infty)$ in variable $r$ as follows:
\[
R = \{n\Delta: n\in \NN\}.
\]
At each $r = n\Delta \in R$, we associate a cell $D_r$
started at $r$
\[
D_r = 
\left[n\Delta,(n+1)\Delta\right).
\]
Then $E(r)$ is estimated by
\[
E(r) = \frac{1}{r}\sum_{\xi\in\Xi: \left|\xi\right|\in D_r }\left| \widehat{f}(\xi)\right|.
\]

The component task to identify the most dominant energy bump in $E(r)$ here belongs to a simple case of geometric object identification problems which have been extensively studied in \cites{Illingworth:1988,GotzD95,Yu04detectingcircular,Arias:2005}. For a $1D$ discrete signal of length $L$, it has been proved that the optimal complexity to detect bumps in this signal is $O(L)$ in \cite{Arias:2005}. Since $E(r)$ has only a few well-separated and sharp energy bumps with relatively small noise, we would adopt the following straightforward but still effective algorithm with $O(L)$ complexity to detect the most dominant energy bump in $E(r)$.

\begin{algo}\label{alg:bump}
Bump detection algorithm
\begin{algorithmic}[1]
\State Input: Vector $E=(E_i)$, $i=0,\dots,L-1$, parameters $c_1$ and $c_2$.
\State Output: End points of the identified bump $r_1$ and $r_2$.
\Function {}{} {\tt BumpDetection}{$(E,c_1,c_2)$}
\State Find $p_0$ s.t. $E_{p_0}=\max_i E_i$ and set up a threshold $\delta = E_{p_0}c_1$.
\State Compute $\wt{E}_i=\min(E_i,\delta)$ for $i=1,\dots,L-1$.
\State Find the largest $p_2< p_0$ s.t. $\wt{E}_{p_2} \geq \wt{E}_{p_2+1}$. If $p_2$ does not exist, let $p_2=0$.
\State Find the largest $p_1< p_0$ s.t. $\wt{E}_{p_1} > \wt{E}_{p_1+1}$. If $p_1$ exists, let $r_1=(p_1+p_2)/2$. Otherwise, let $r_1=p_2$.
\State Find the smallest $p_1> p_0$ s.t. $\wt{E}_{p_1} \geq \wt{E}_{p_1-1}$. If $p_1$ does not exist, let $p_1=L-1$.
\State Find the smallest $p_2> p_0$ s.t. $\wt{E}_{p_2} > \wt{E}_{p_2-1}$. If $p_2$ exists, let $r_2=(p_1+p_2)/2$. Otherwise, let $r_2=p_1$.
\State Let $r_1=\max\{0,r_1(1-c_2)\}$ and $r_2=\min\{L-1,r_2(1+c_2)\}$.
\State \Return $[r_1,r_2]$
\EndFunction
\end{algorithmic}
\end{algo}

When the dominant energy bump is located in the low frequency part, i.e. $r_1=0$, the energy comes from the smooth trend functions of the crystal image. Hence, it is reasonable to eliminate this bump from $E(r)$ and apply Algorithm \ref{alg:bump} again to update new $[r_1,r_2]$, which is the frequency band of dominant local wave vectors. Since the complexity of each searching procedure is at most $O(L)$, the complexity of Algorithm \ref{alg:bump} is $O(L)$. Because the time complexity of $2D$ FFT is $O(L^2\log L)$, the total time complexity for frequency band detection is $O(L^2\log L)$.

\subsection{Band-limited $2D$ fast SST}
\label{sec:SST}

The $2D$ synchrosqueezed transforms were originally proposed in \cites{YangYing:2013,YangYing:preprint} for $2D$ mode decomposition problems. For an image of size $L$ by $L$, the time complexity for these transforms is $O(L^{2+t-s}\log(N))$, where $t$ and $s$ are the geometric scaling parameters for the general wave packet transforms. To reduce the complexity, a band-limited $2D$ fast SST is introduced following a similar methodology to estimate local wave vectors with wave numbers in $[r_1,r_2]$ provided by the frequency band detection.

The band-limited $2D$ SST is based on a band-limited $2D$ general wave packet transform. The key idea is to restrict the class of general wave packets
\[
\{w_{a\theta b}(x),a\in [1,\infty),\theta\in[0,2\pi),b\in \RR^2\}
\]
to a band-limited class
\[
\{w_{a\theta b}(x),a\in [r_1,r_2],\theta\in[0,2\pi),b\in \RR^2\}.
\]
Meanwhile, a discrete analog of the band-limited class of general wave packets is constructed by specifying a set of tilings covering the annulus $\{\xi\in \RR^2:r_1\leq \left|\xi\right| \leq r_2\}$. To be more specific, we need to specify how to decimate the Fourier domain $(a,\theta)$ and the position space $b$ as follows.

In the Fourier domain, the decimation follows two steps. First, we follow the discretization and construction of the discrete general curvelet transform in \cite{YangYing:preprint} to construct the discrete general wave packets here. Notice that the support parameter $d$ in \cite{YangYing:preprint} is $1$. The sizes of tilings in this paper should be multiplied by $d$. Second, we keep those discrete general wave packets with supports overlapping the annulus $\{\xi\in \RR^2:r_1\leq \left|\xi\right| \leq r_2\}$ to obtain the band-limited class of discrete general wave packets.

To decimate the position space $b$, we discretize it with an $L_B\times L_B$ uniform grid:
\[
B = \{ (n_1/L_B,n_2/L_B): 0 \le n_1, n_2<L_B, n_1,n_2 \in \ZZ\}.
\]
It is required that $L_B$ is large enough so that a sampling grid of size $L_B\times L_B$ can cover the supports of general wave packets of interest in the Fourier domain. For a crystal image
\[
f(x)=\sum_k\chi_{\Omega_k} (x)\left( \alpha_k(x) S\left(2\pi NF \phi_k(x)\right)+c_k(x)\right)
\]
defined in $[0,1)^2$, its fundamental wave number is $O(N)$, i.e., $r_1=O(N)$ and $r_2=O(N)$. Hence, the largest support of general wave packets in the Fourier domain is of size $O(N^t)$, implying that $L_B$ is at least $O(N^t)$. 

With the band-limited class of general wave packets ready, we follow the fast algorithms in \cite{YangYing:preprint} to compute the forward general wave packet transform $W_f(a,\theta,b)$, the local wave vector estimate $v_f(a,\theta,b)$ and the synchrosqueezed energy distribution $T_f(a,\theta,b)$. Notice that the width of the frequency band $r_2-r_1=O(1)$ compared to the fundamental wave number, which implies that the cardinality of the band-limited class is $O(N^{1-s})$. A straightforward calculation shows that the time complexity of the band-limited general wave packet transform is $O(L^2\log L+N^{1-s}L_B^2\log L_B)$. Because the synchrosqueezing procedure takes a complexity of $O(N^{1-s}L_B^2)$, the total time complexity for the band-limited $2D$ fast SST is $O(L^2\log L+N^{1-s}L_B^2\log L_B)$.

As an example, Figure \ref{fig:Toy_fast_ss} shows the synchrosqueezed energy distribution $T_f(a,\theta,b)$ in a polar coordinate at three different positions. Because the crystal image is real, it is enough to compute the synchrosqueezed energy distribution for $\theta\in[0,\pi)$. The results show that the essential support of $T_f(a,\theta,b)$ can accurately estimate local wave vectors when location $b$ is not at the boundary. When $b$ is at the boundary, the essential support of $T_f(a,\theta,b)$ can still provide some information, e.g. crystal rotations.

\begin{figure}[ht!]
  \begin{center}
    \includegraphics[height=1.7in]{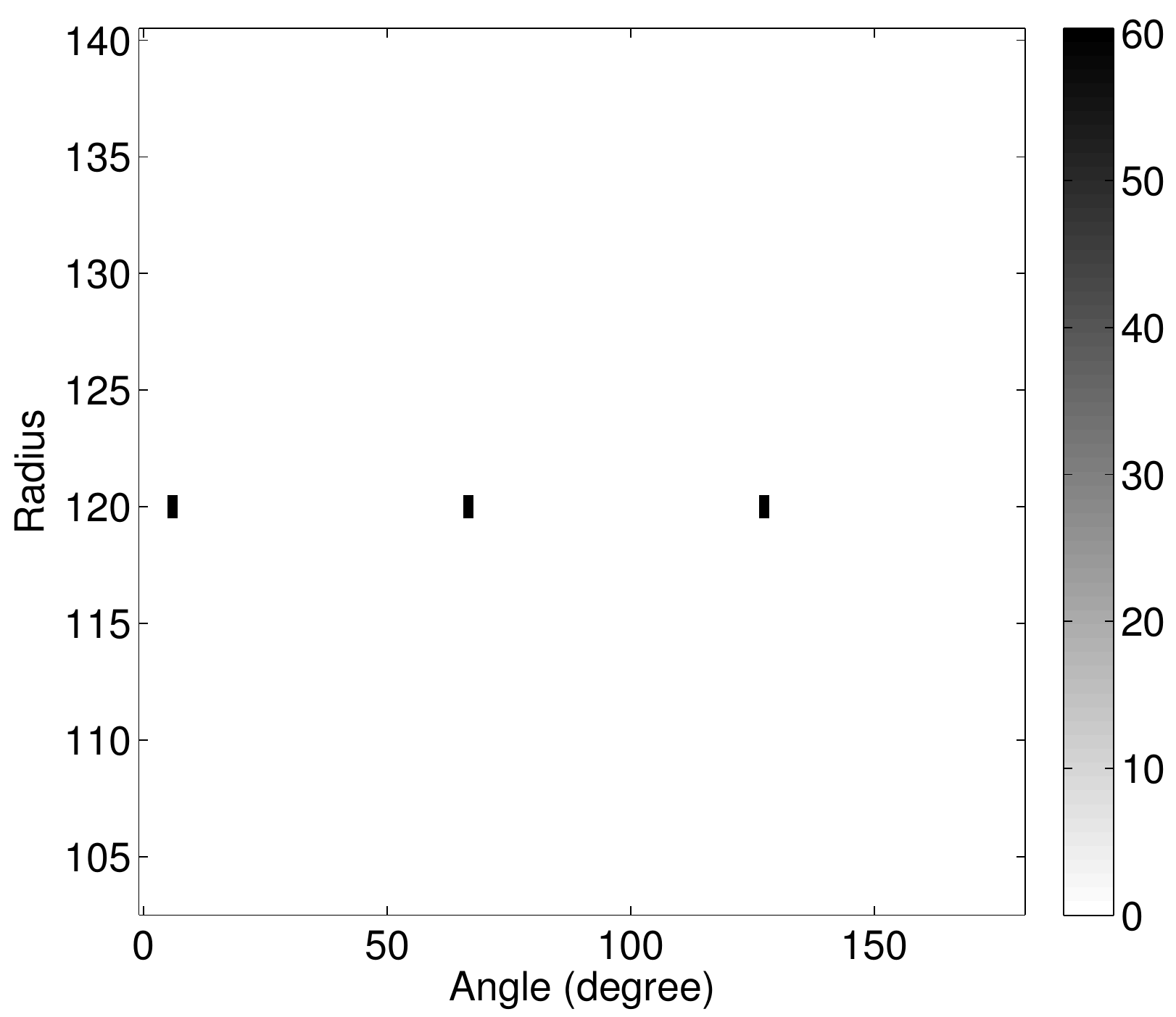}   \includegraphics[height=1.7in]{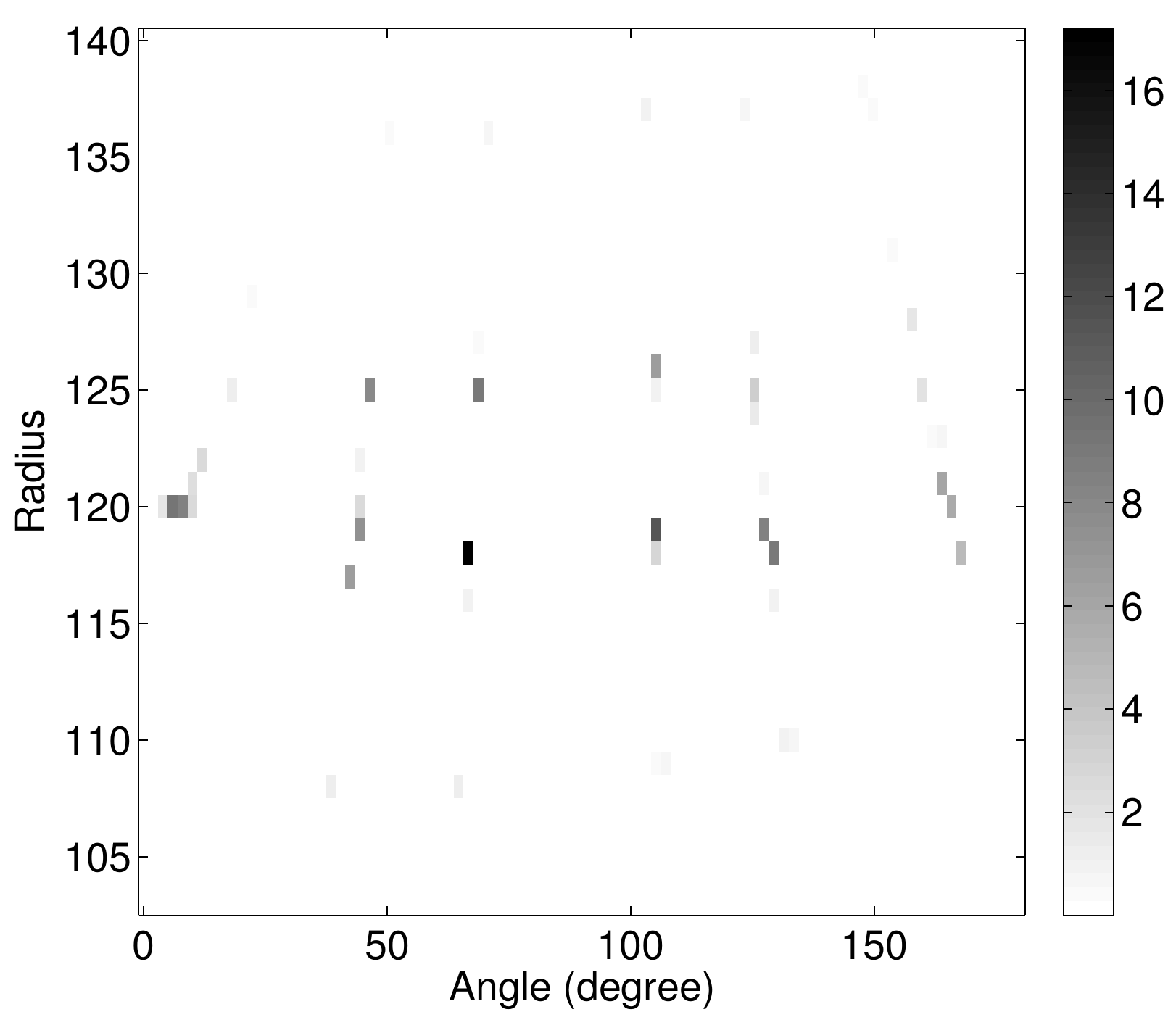} \includegraphics[height=1.7in]{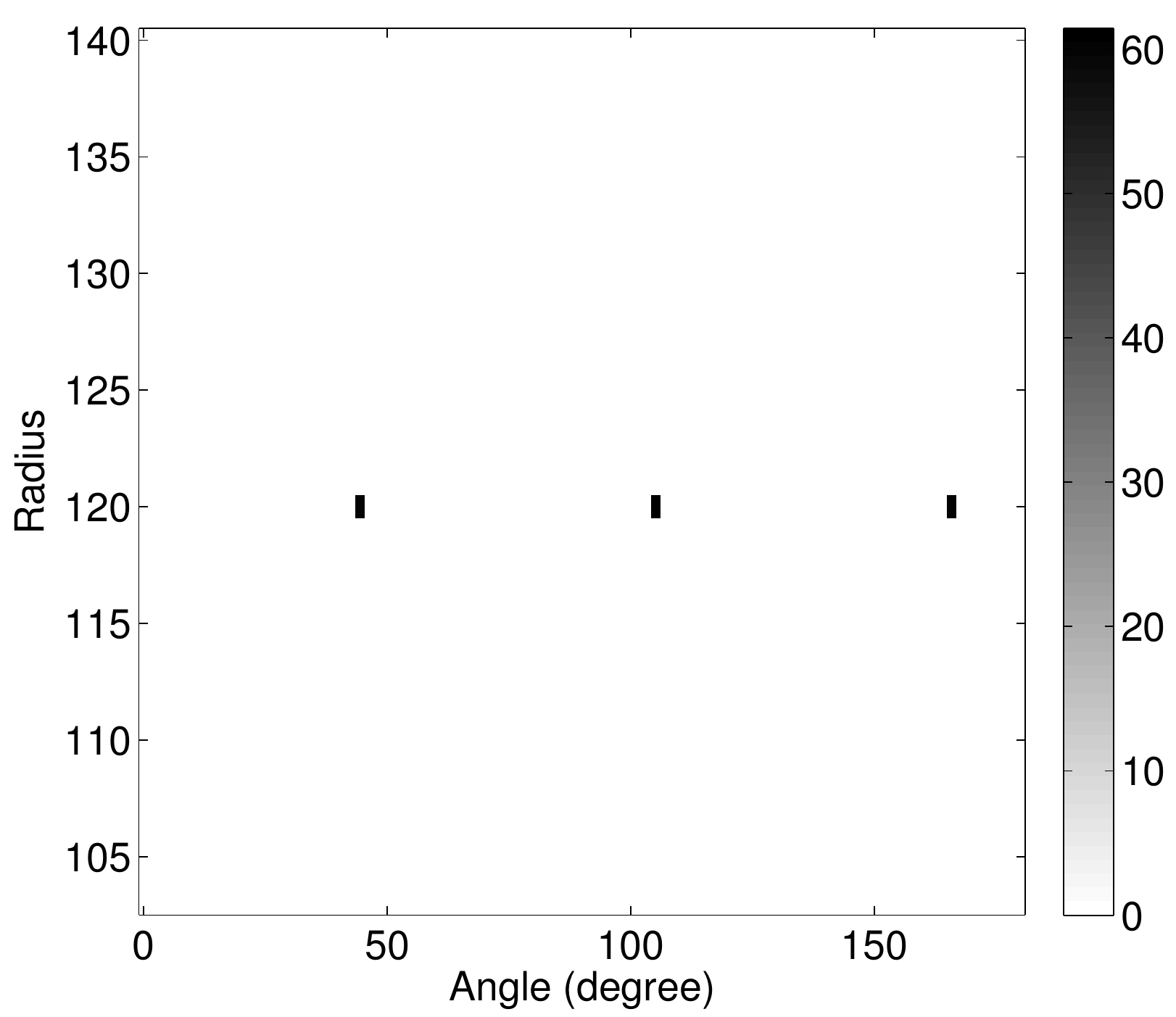}
  \end{center}
  \caption{The synchrosqueezed energy distribution $T_f(a,\theta,b)$ of Figure~\ref{fig:Toy_fast_spec} (left) at three different points $b_i$, $i=1,2,3$. Left: $b_1=(0.25,0.5)$ is in the middle of the left grain. Middle: $b_2=(0.5,0.5)$ is at the boundary. Right: $b_3=(0.75,0.5)$ is in the middle of the right grain. Note that the scale of colorbar is different in the middle panel.}
  \label{fig:Toy_fast_ss}
\end{figure}

\subsection{Defect detection algorithms}

As we can see in Figure \ref{fig:Toy_fast_ss}, the synchrosqueezed energy around each local wave vector $\grad\left(Nn^\TT F\phi(x)\right)$ is stable and of order $|\widehat{S}(n)|\alpha_k(x)$ when $x$ is in the perfect interior $\wt{\Omega}_k$. Moreover, the energy would decrease fast near the boundary $\partial\Omega_k$ and becomes zero soon outside $\Omega_k$. This motivates the application of synchrosqueezed energy distribution to identify grain boundaries by detecting the irregularity of energy distribution as follows.

\begin{algomain}[Fast defect detection algorithm based on stacked
  synchrosqueezed energy]\
\label{alg:undeformed}

\begin{itemize}
\item[\textbf{Step 1:}] Stack the synchrosqueezed energy distribution and define a stacked synchrosqueezed energy distribution as
\[
\wt{T}_f(a,\theta,b)=T_f(a,\theta,b)+T_f(a,\theta+\frac{\pi}{3},b)+T_f(a,\theta+\frac{2\pi}{3},b)
\]
for $\theta\in[0,\frac{\pi}{3})$. 

\item[\textbf{Step 2:}] Compute an angular total energy distribution \[E_a(\theta,b)=\int_{r_1}^{r_2}\wt{T}_f(a,\theta,b)da.\]

\item[\textbf{Step 3:}] For each $b$, apply Algorithm \ref{alg:bump} to identify the most dominant energy bump in $E_a(\theta,b)$. Denote the range of this bump as $[\theta_{11}(b),\theta_{12}(b)]$.

\item[\textbf{Step 4:}]  Compute the total energy of the first identified bump as
\[
\TE_1(b)=\int_{\theta_{11}(b)}^{\theta_{12}(b)}E_a(\theta,b)d\theta.
\]

\item[\textbf{Step 5:}] Compute a weighted average angle of the first bump as
\[
\Angle(b)=\frac{1}{\TE_1(b)} \int_{\theta_{11}(b)}^{\theta_{12}(b)} E_a(\theta,b)\theta d\theta.
\]

\item[\textbf{Step 6:}] For each $b$, update $E_a$ such that $E_a(\theta,b)=0$, if $\theta\in[\theta_{11}(b),\theta_{12}(b)]$. Apply Algorithm~\ref{alg:bump} again to update the most dominant energy bump in $E_a$. Denote the range of this bump as $[\theta_{21}(b),\theta_{22}(b)]$.

\item[\textbf{Step 7:}] Compute the total energy of the second identified bump by
\[
\TE_2(b)=\int_{\theta_{21}(b)}^{\theta_{22}(b)}E_a(\theta,b)d\theta.
\]

\item[\textbf{Step 8:}] Compute the boundary indicator function
\[
\BD(b)=\frac{1}{\sqrt{\TE_1(b)-\TE_2(b)+1}}.
\]
\end{itemize}
\end{algomain}

\begin{figure}[ht!]
  \begin{center}
    \hspace{-3em}
      \includegraphics[height=1.8in]{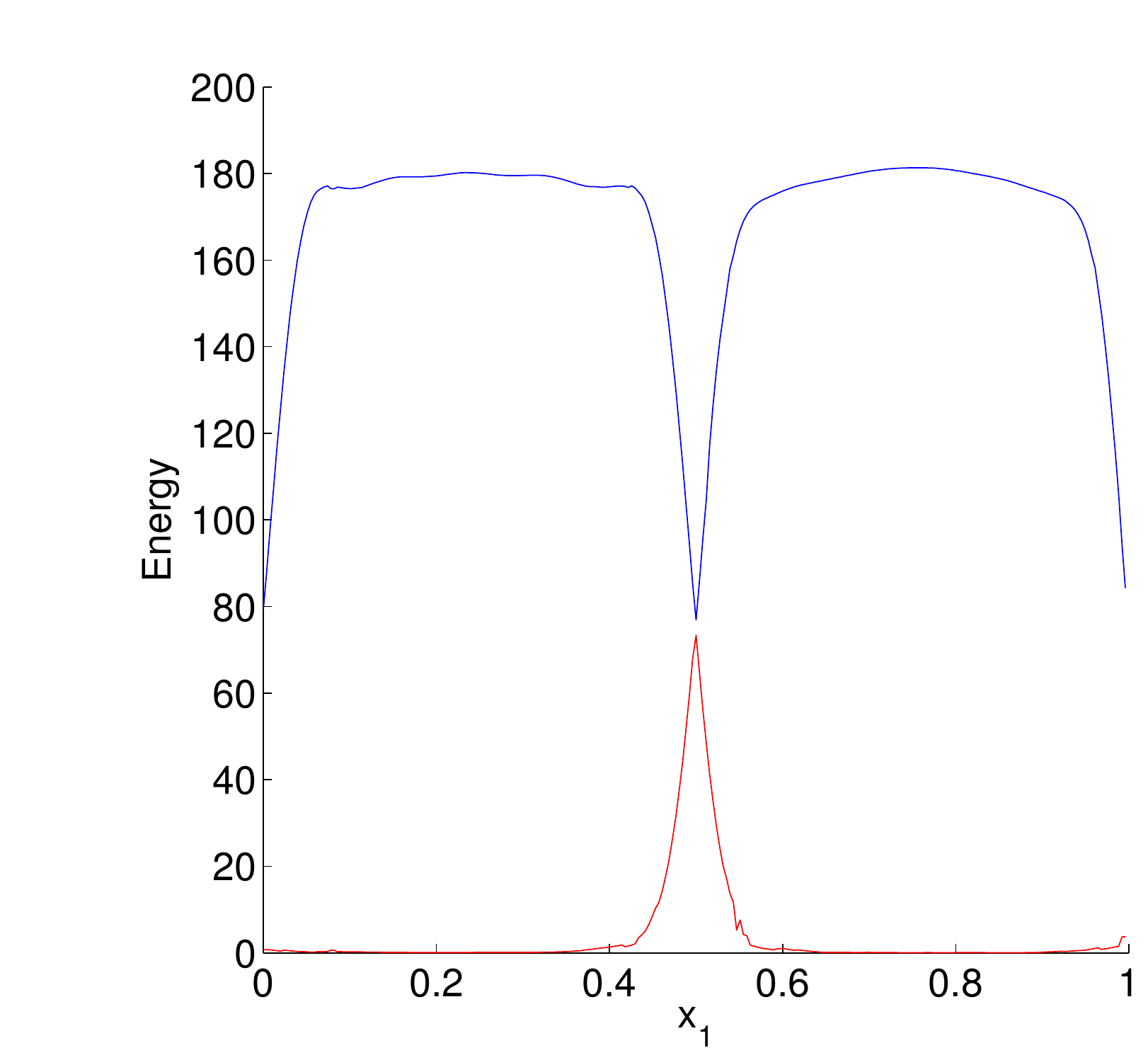}  \includegraphics[height=1.8in]{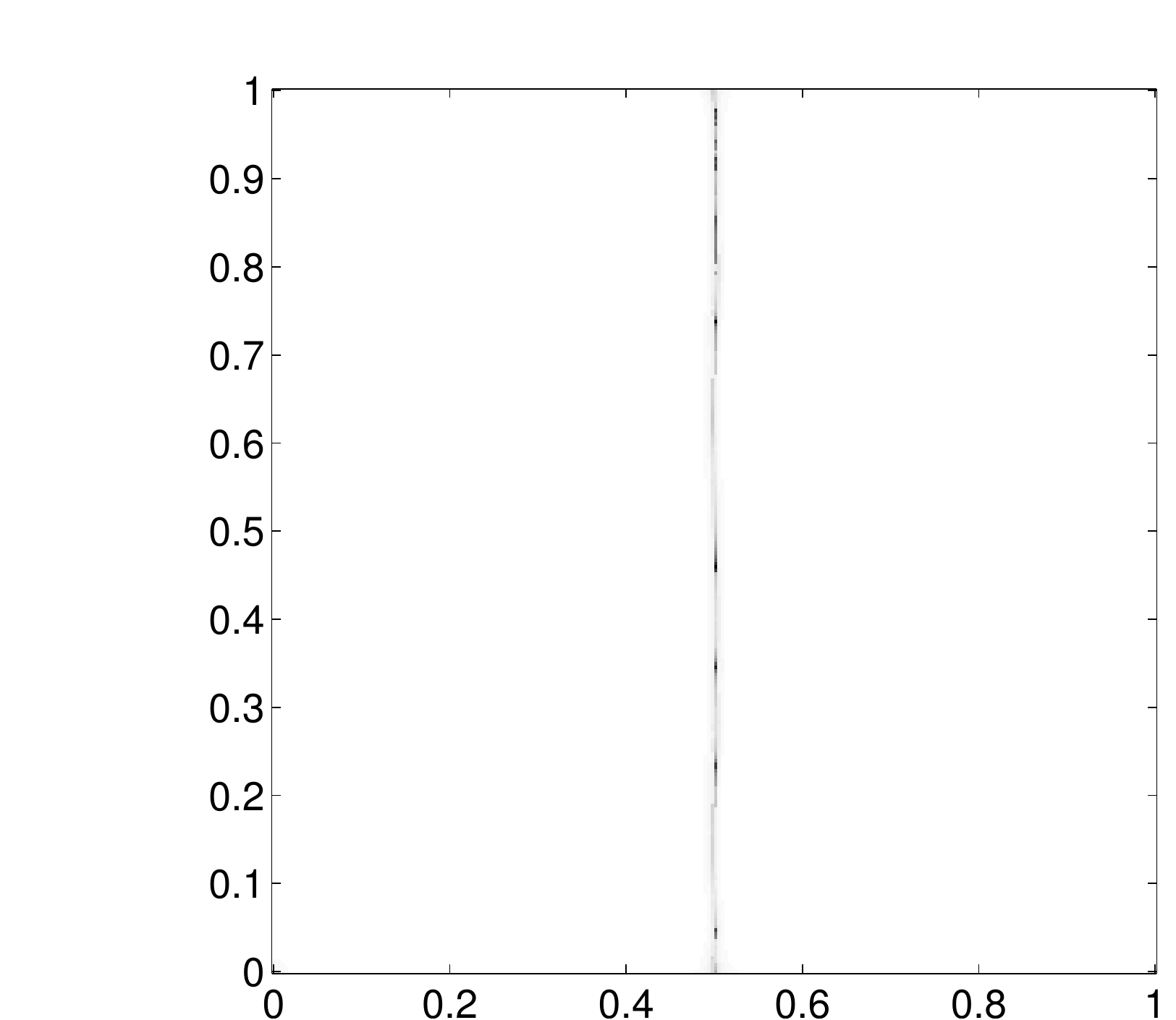}  \includegraphics[height=1.8in]{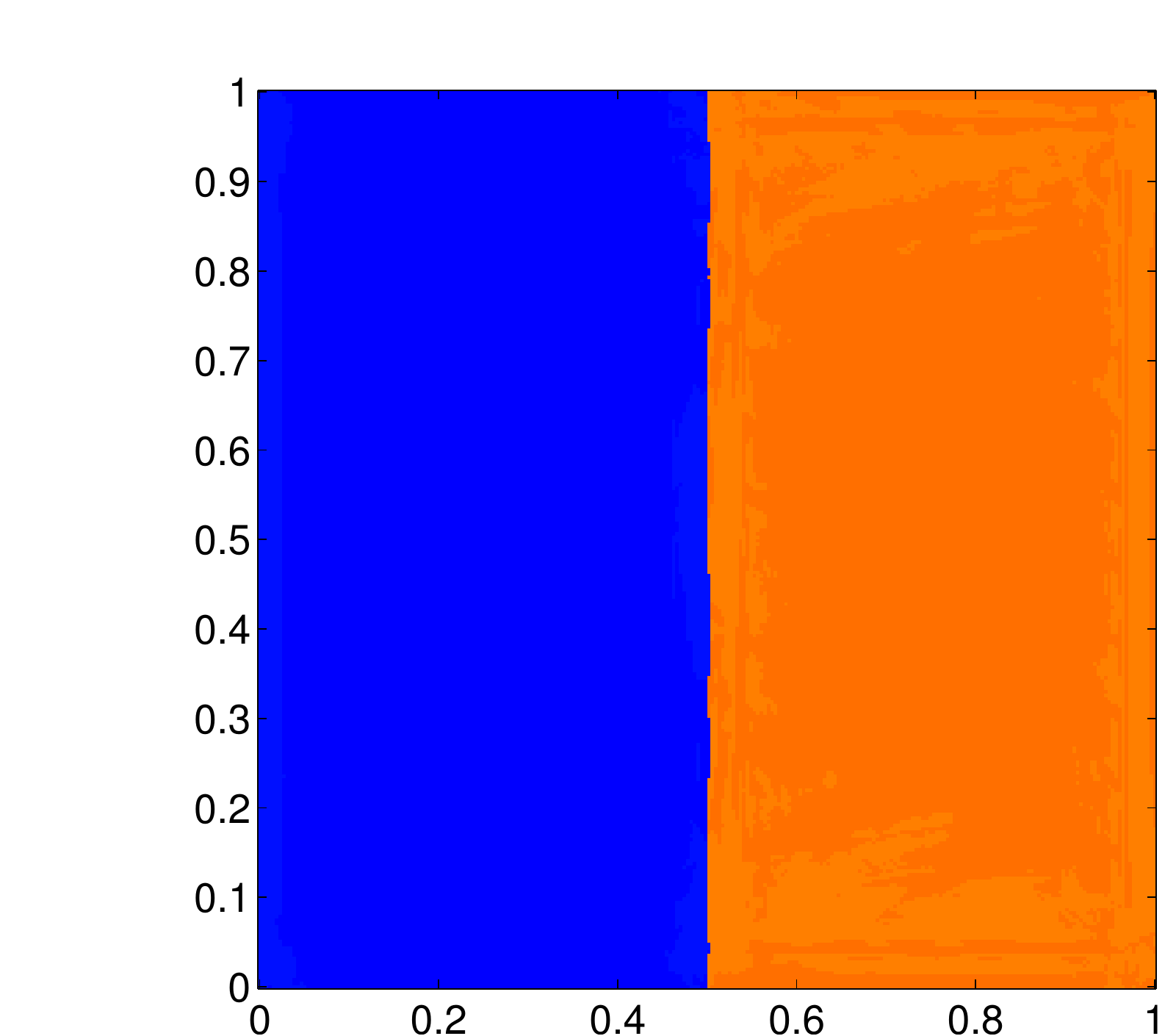}
  \end{center}
  \caption{Left: The total energy function $TE_1(b)$ in blue and $TE_2(b)$ in red for fixed $b_2=0.5$. Middle: The boundary indicator function $\BD(b)$. Right: The weighted average angle $\Angle(b)$ as an approximation of crystal rotations. The real crystal rotations are $15$ degrees on the left and $57.5$ degrees on the right.}
  \label{fig:Toy_fast_result}
\end{figure}

The synchrosqueezed energy distribution of each grain behaves like an energy bump essentially supported in $\Omega_k$ and quickly decays outside $\Omega_k$. Hence, two energy bumps have close energy in the transition area containing the grain boundary. Since $\TE_1(b)$ acts as the maximum of two energy distributions and $\TE_2(b)$ acts as the minimum (see Figure \ref{fig:Toy_fast_result} (left)), $\TE_1(b)-\TE_2(b)$ decays near the grain boundary and the boundary indicator function $\BD(b)$ is relatively large at the grain boundary. As Figure \ref{fig:Toy_fast_result} (middle) shows, the grain boundary can be identified from a grayscale image of $\BD(b)$.

By the discussion in Section \ref{sec:BumpDetection}, the local rotation function with respect to each local wave vector of the reference hexagonal configuration are approximately the same. Since, the synchrosqueezed energy distribution $T_f(a,\theta,b)$ is stacked together per $\frac{\pi}{3}$ in $\theta$, the weighted average $\Angle(b)$ approximates the local rotation functions in a weighted average sense. As shown in Figure \ref{fig:Toy_fast_result} (right), $\Angle(b)$ accurately reflects the crystal rotations in this example.

As discussed in Section \ref{sec:BumpDetection} and \ref{sec:SST}, before Algorithm $\ref{alg:undeformed}$, it takes $O(L^2\log L+N^{1-s}L_B^2\log L_B)$ time complexity to estimate the frequency band $[r_1,r_2]$ and to compute $T_f(a,\theta,b)$. Suppose that the discrete grid of $T_f(a,\theta,b)$ at each $b$ is of size $L_R \times L_A$, then the time complexity of Algorithm $\ref{alg:undeformed}$ is $O(L_A L_R L_B^2)$. Hence, the total complexity for the crystal analysis is $O(L^2\log L+N^{1-s}L_B^2\log L_B+L_A L_R L_B^2)$.

The stacking step in Algorithm~$\ref{alg:undeformed}$ is averaging the influence of each local wave vector. This gives a stable result of grain boundary and crystal rotation estimates even with severe noise as will be illustrated in Section \ref{sec:examples}. However, Algorithm~$\ref{alg:undeformed}$ might miss some local defects that would not influence all local wave vectors simultaneously. For example, in Figure \ref{fig:GB5GB8} (left), two of the underlying wave-like components have smoothly changing directions, while the third one changes its direction suddenly at a line segment, resulting in a small angle boundary. At some local dislocations, as Figure \ref{fig:GB5GB8} (right) shows, the dislocation might not cause irregularity to all wave-like components. Hence, to be more sensitive to local irregularity, it is reasonable to remove the stacking step in Algorithm~$\ref{alg:undeformed}$, if noise is relatively small. This motivates the following algorithm.

\begin{figure}[ht!]
  \begin{center}
    \includegraphics[height=1.8in]{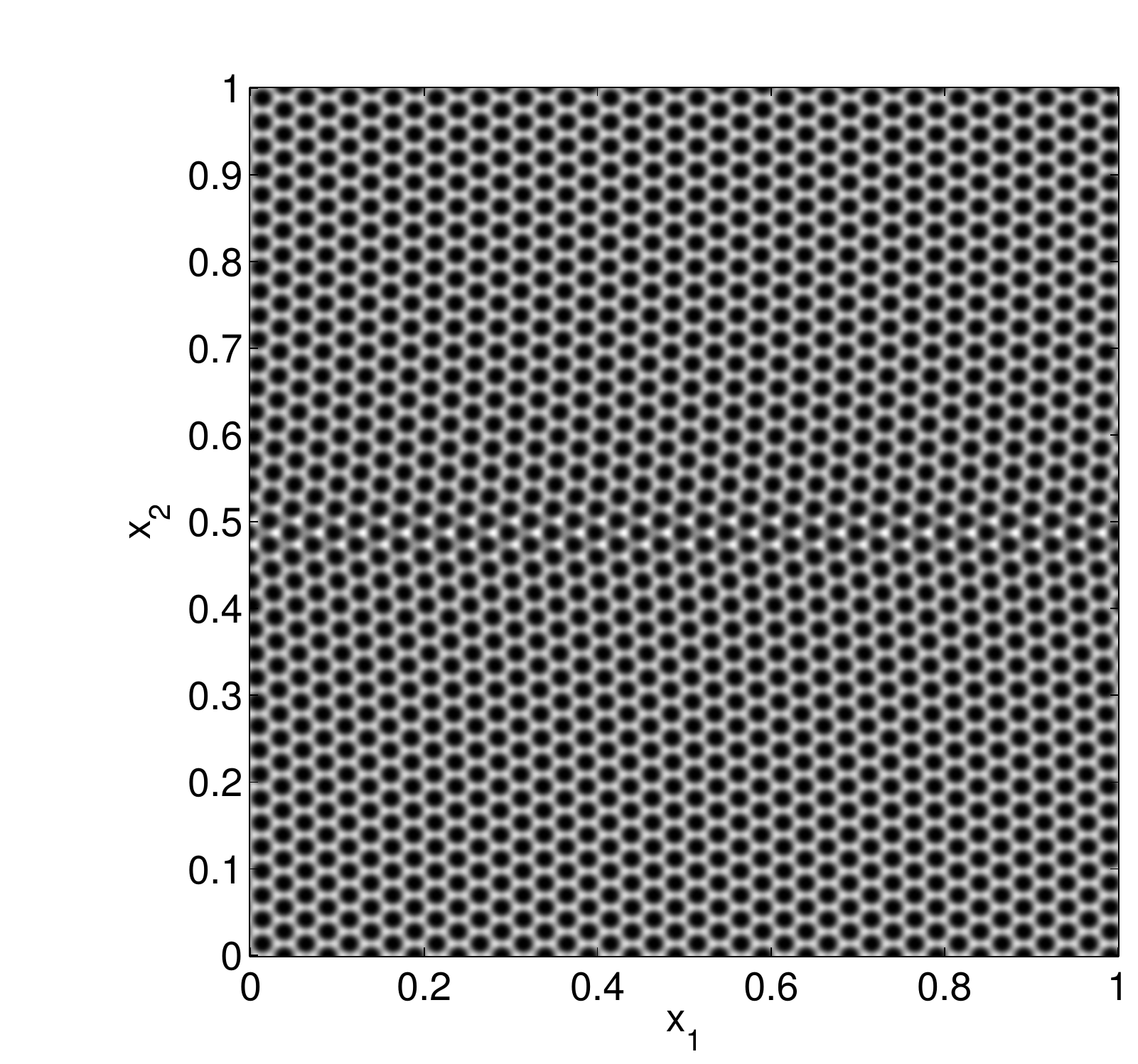} \qquad  \includegraphics[height=1.8in]{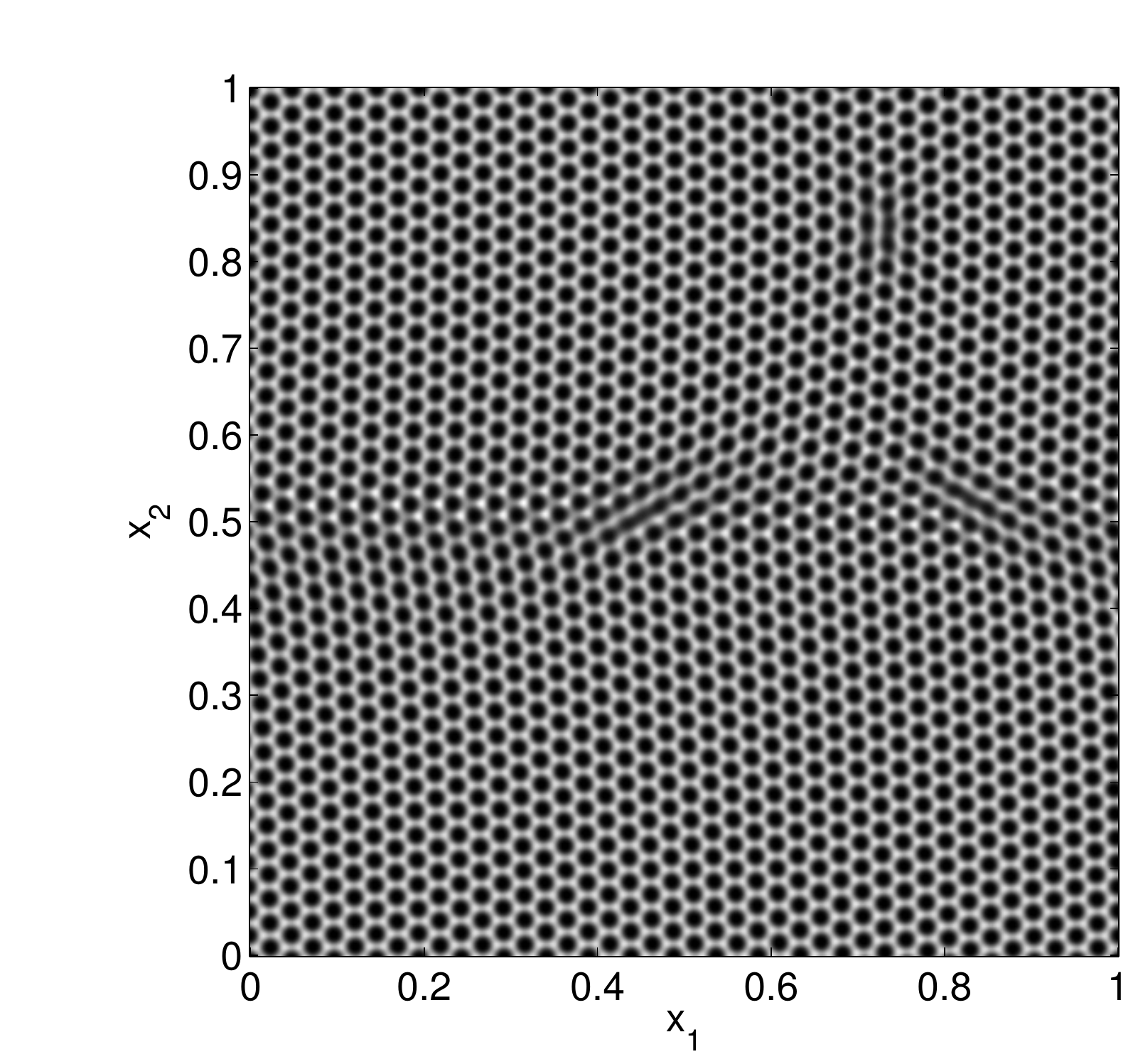}
  \end{center}
  \caption{Left: An example of a small angle boundary. Right: An example of some isolated dislocations. Courtesy of Benedikt Wirth. The sizes of these images are $512\times 512$ pixels.}
  \label{fig:GB5GB8}
\end{figure}

\begin{algomain}[Fast defect detection algorithm with enhanced sensitivity]\

\label{alg:deformed}

\begin{itemize}
\item[\textbf{Step 1:}] Define $\wt{T}_j(a,\theta,b) = T_f(a,\theta+j\pi/3,b)$ for $\theta\in[0,\frac{\pi}{3})$ and $j=0,1,2$. Apply Steps $2-8$ in Algorithm~\ref{alg:undeformed} to $\wt{T}_j(a,\theta,b)$ to compute $\TE_{1j}(b)$, $\TE_{2j}(b)$, $\Angle_j(b)$ and $\BD_j(b)$ for each $j$.
\item[\textbf{Step 2:}] For each $j=0,1,2$, compute the weight function
\[
W_j(b)=\frac{\TE_{1j}(b)+\TE_{2j}(b)}{\displaystyle \sum_j \left(\TE_{1j}(b)+\TE_{2j}(b)\right)}.
\]
\item[\textbf{Step 3:}] Compute the weighted average angle
\[
\Angle(b)=\sum_j W_j(b)\Angle_j(b).
\]
\item[\textbf{Step 4:}] Compute the  weighted boundary indicator function
\[
\BD(b)=\sum_j W_j(b)\BD_j(b).
\]
\end{itemize}
\end{algomain}

\begin{figure}[ht!]
  \begin{center}
    \begin{tabular}{lll}
    \includegraphics[height=1.3in]{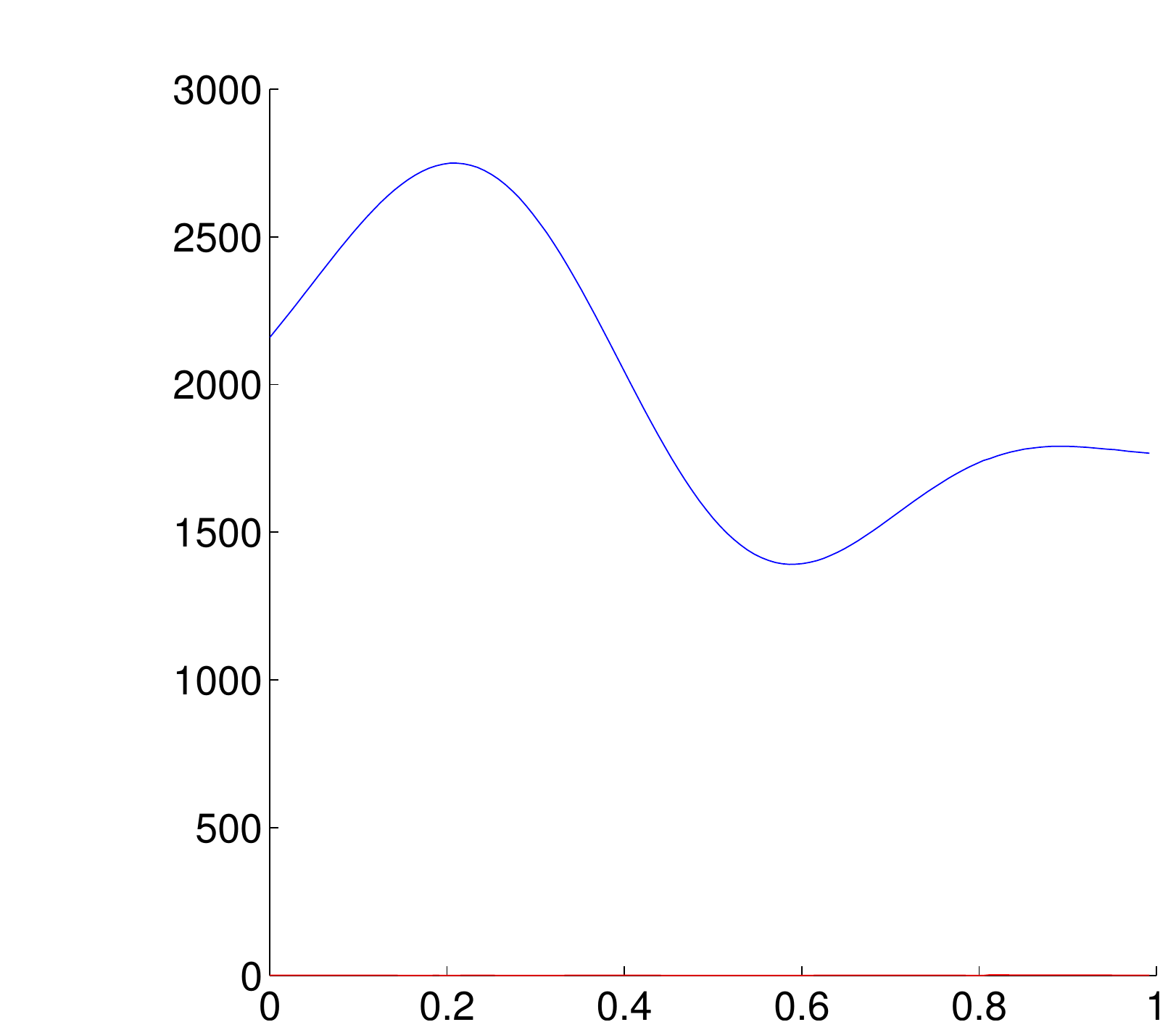}  & \includegraphics[height=1.3in]{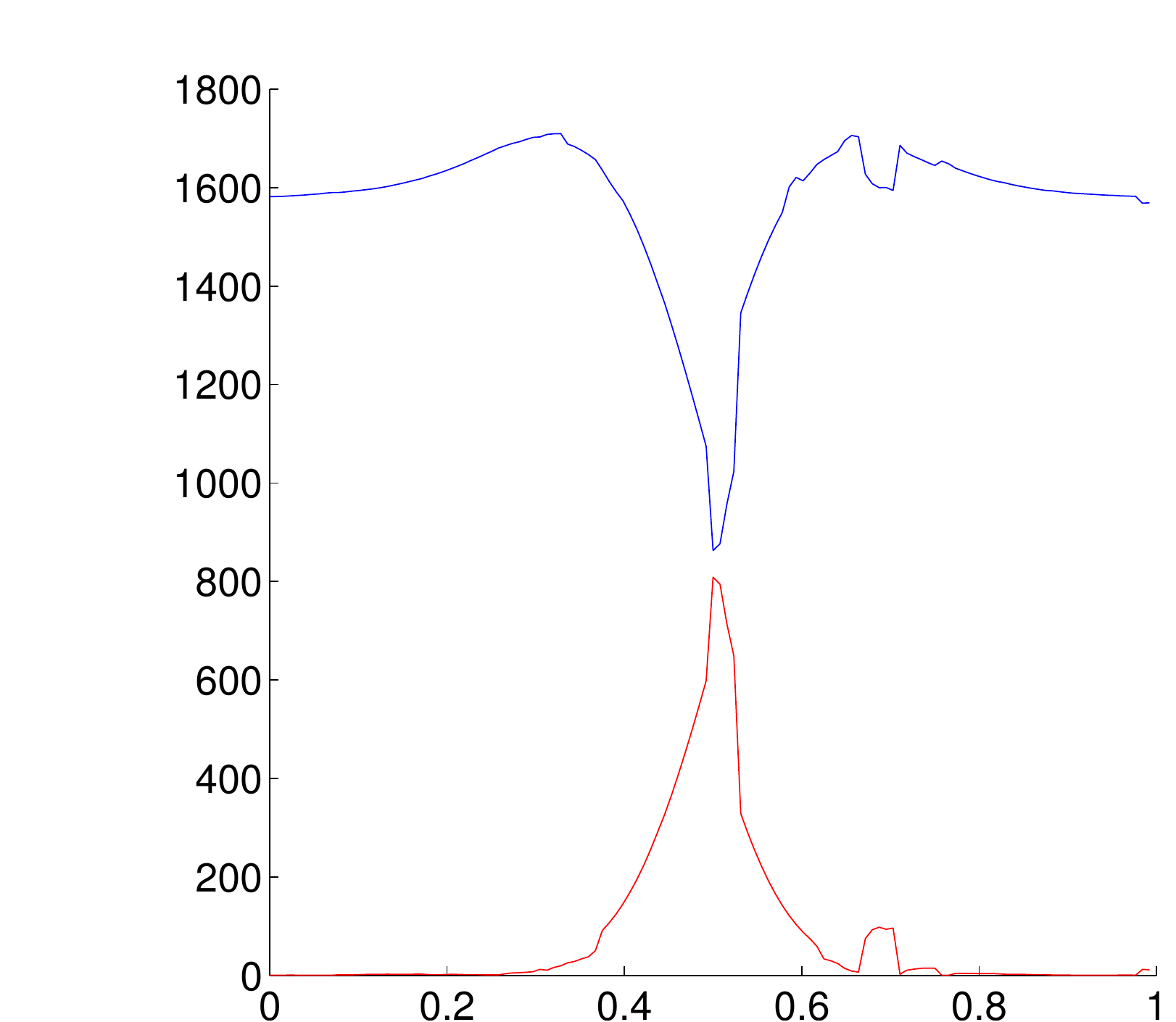} & \includegraphics[height=1.3in]{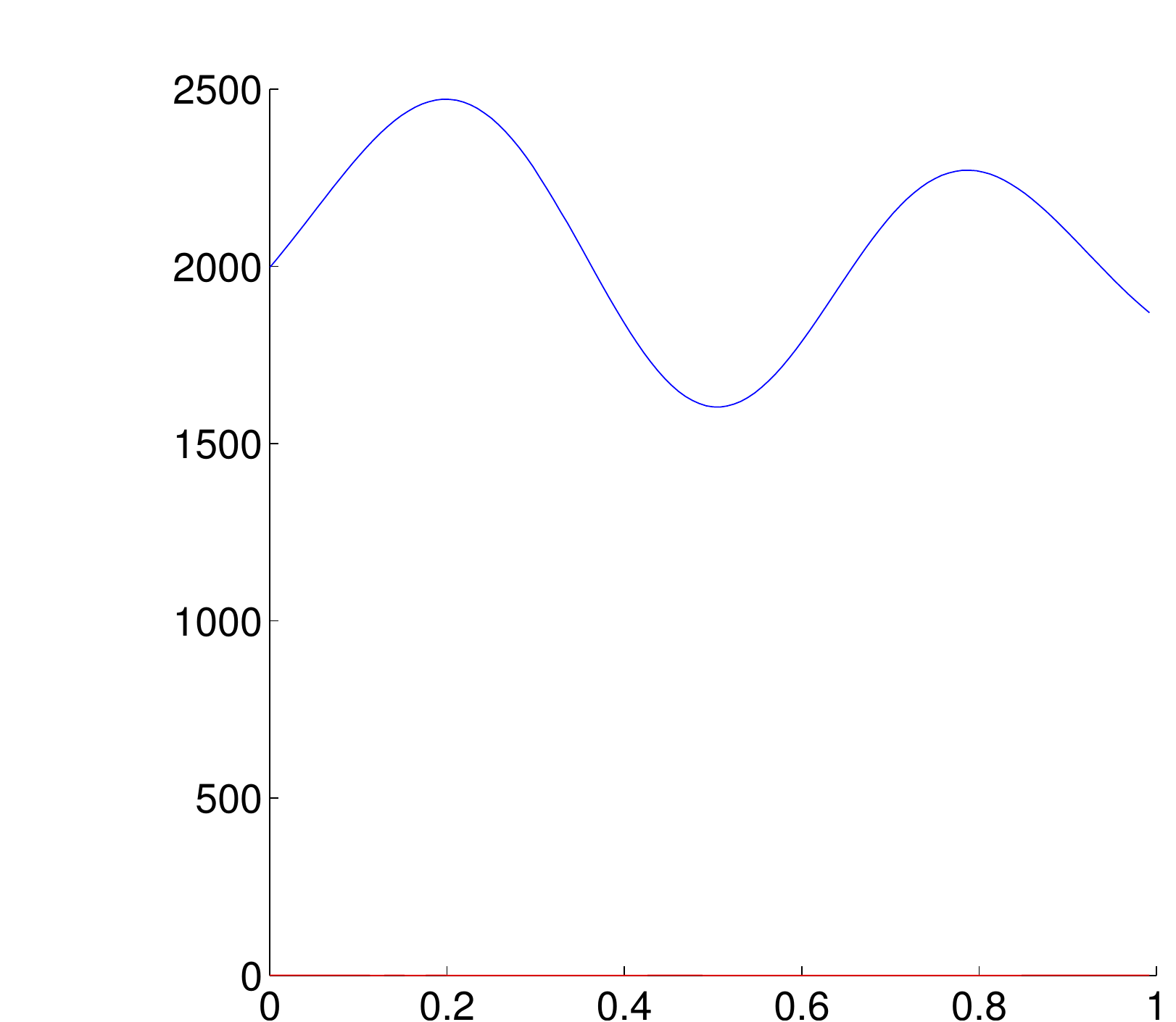} \\  \includegraphics[height=1.3in]{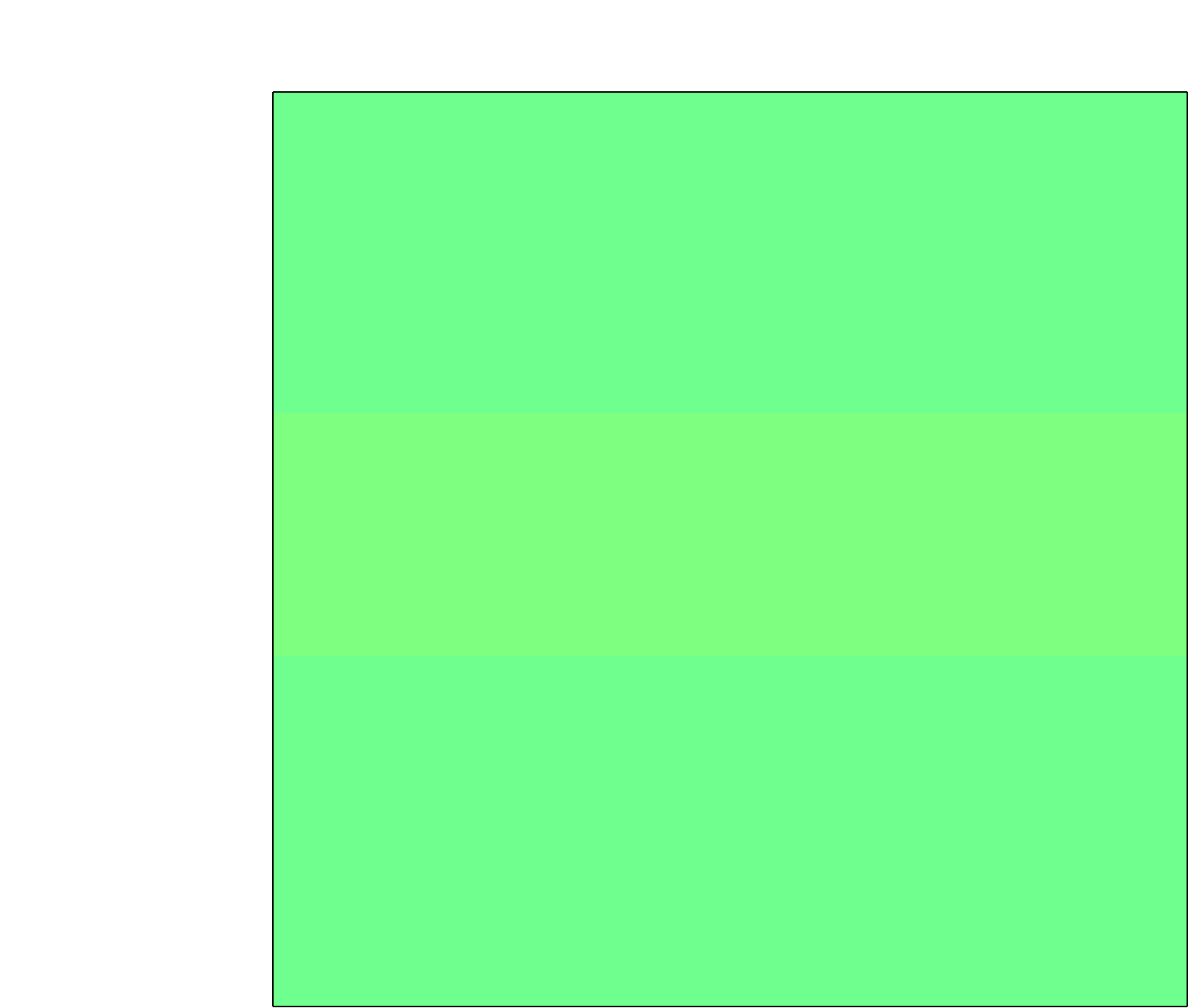} &\includegraphics[height=1.3in]{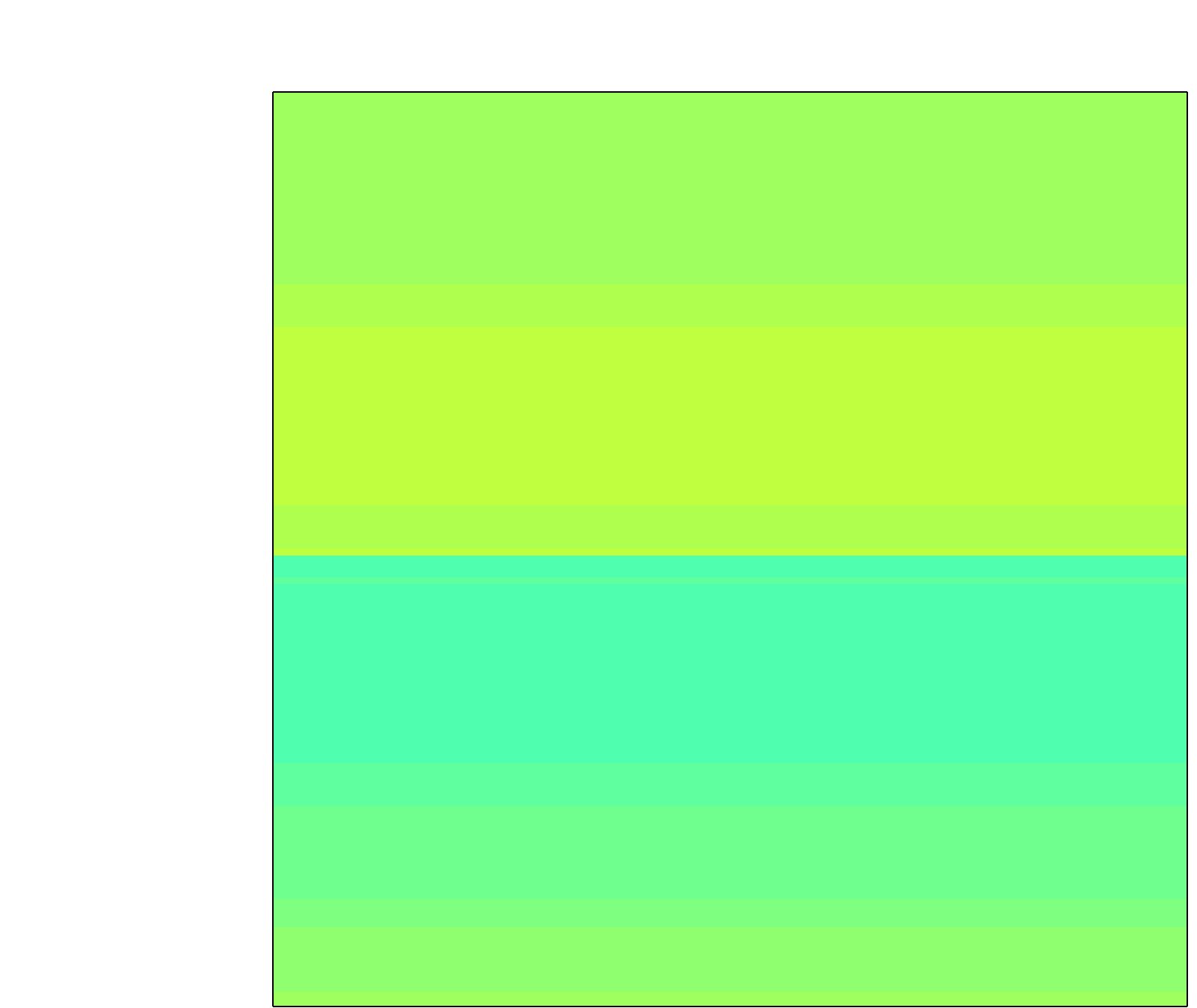}  &\includegraphics[height=1.3in]{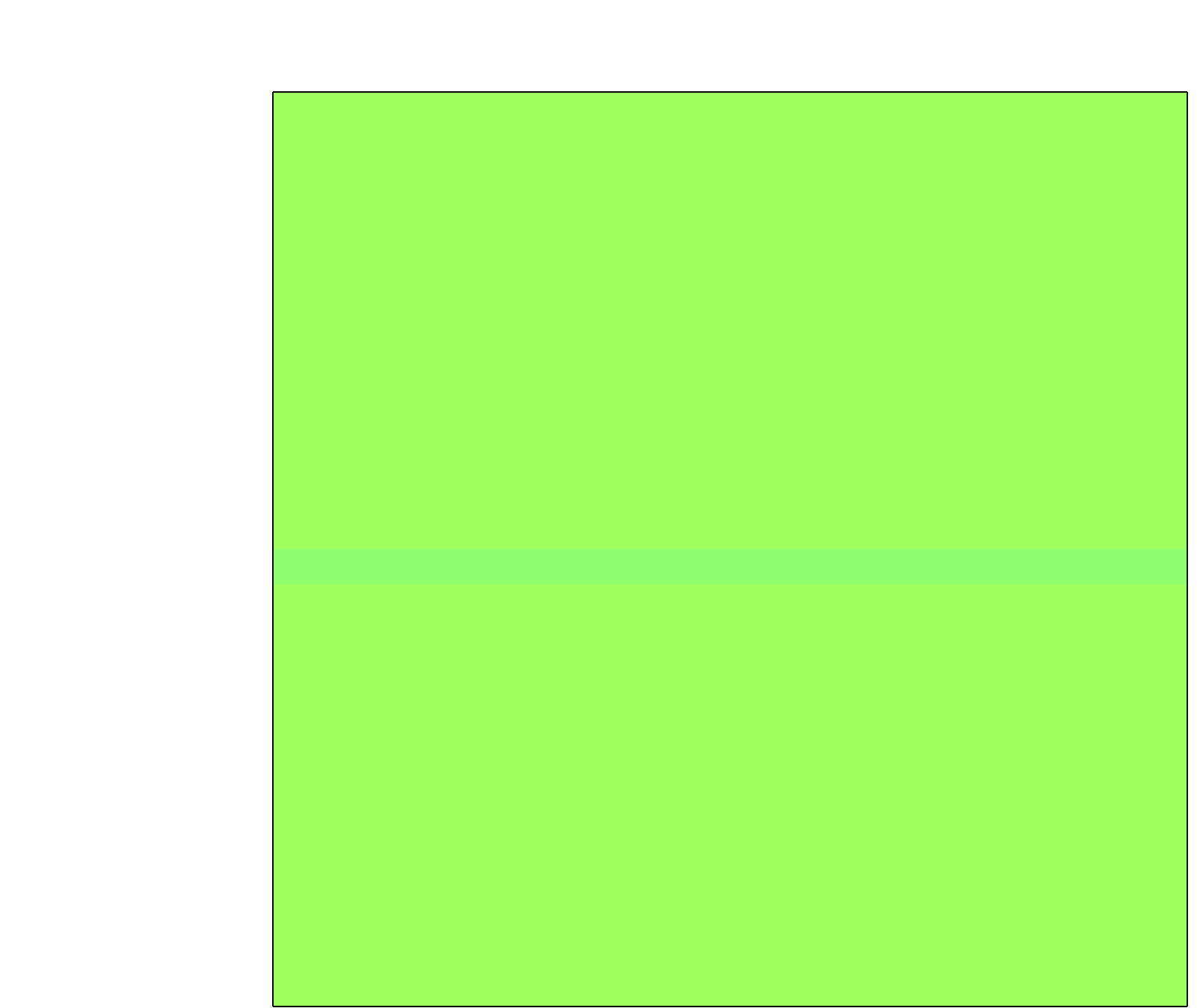} 
  \end{tabular}
\includegraphics[height=1.4in]{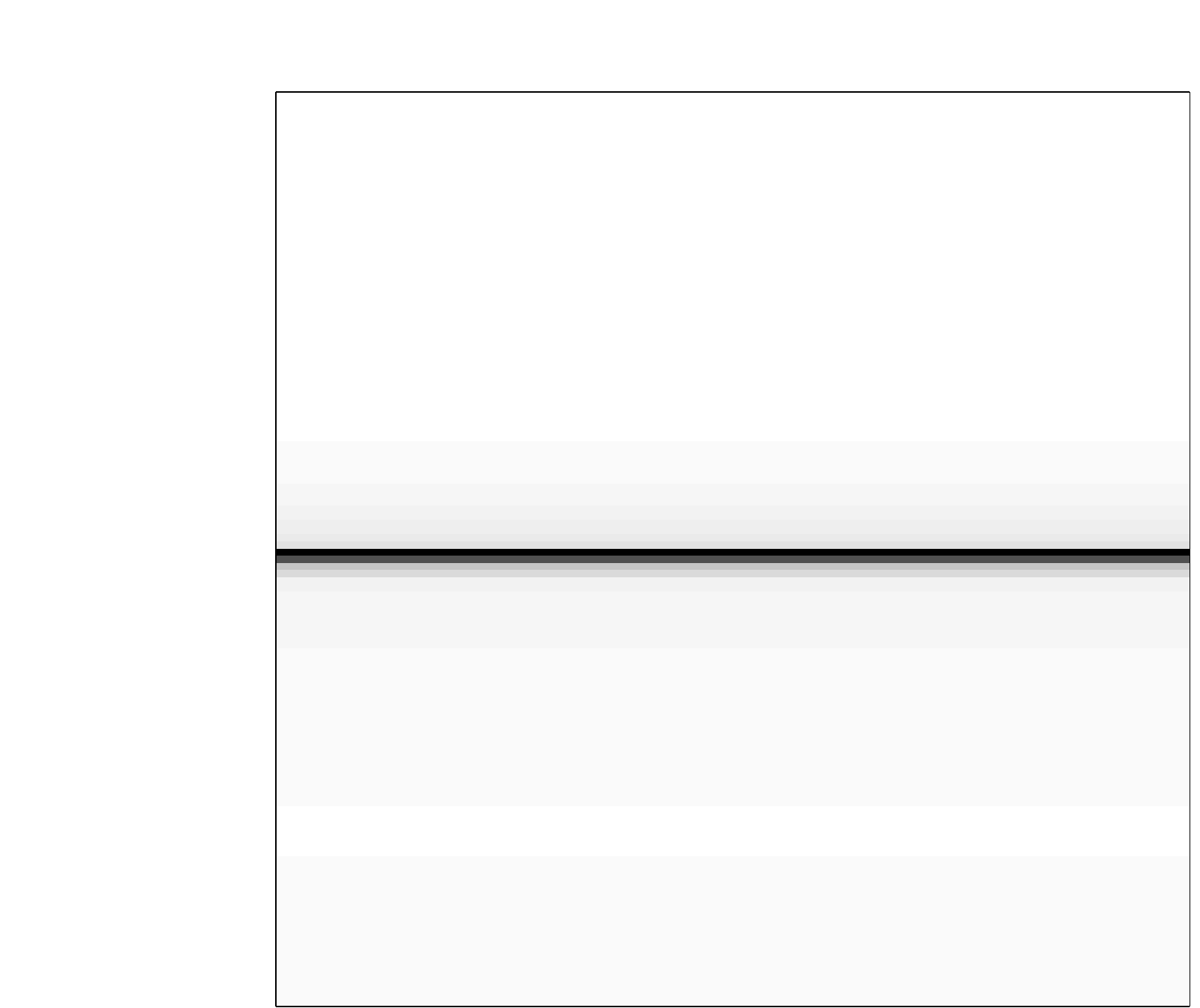} \qquad 
  \includegraphics[height=1.4in]{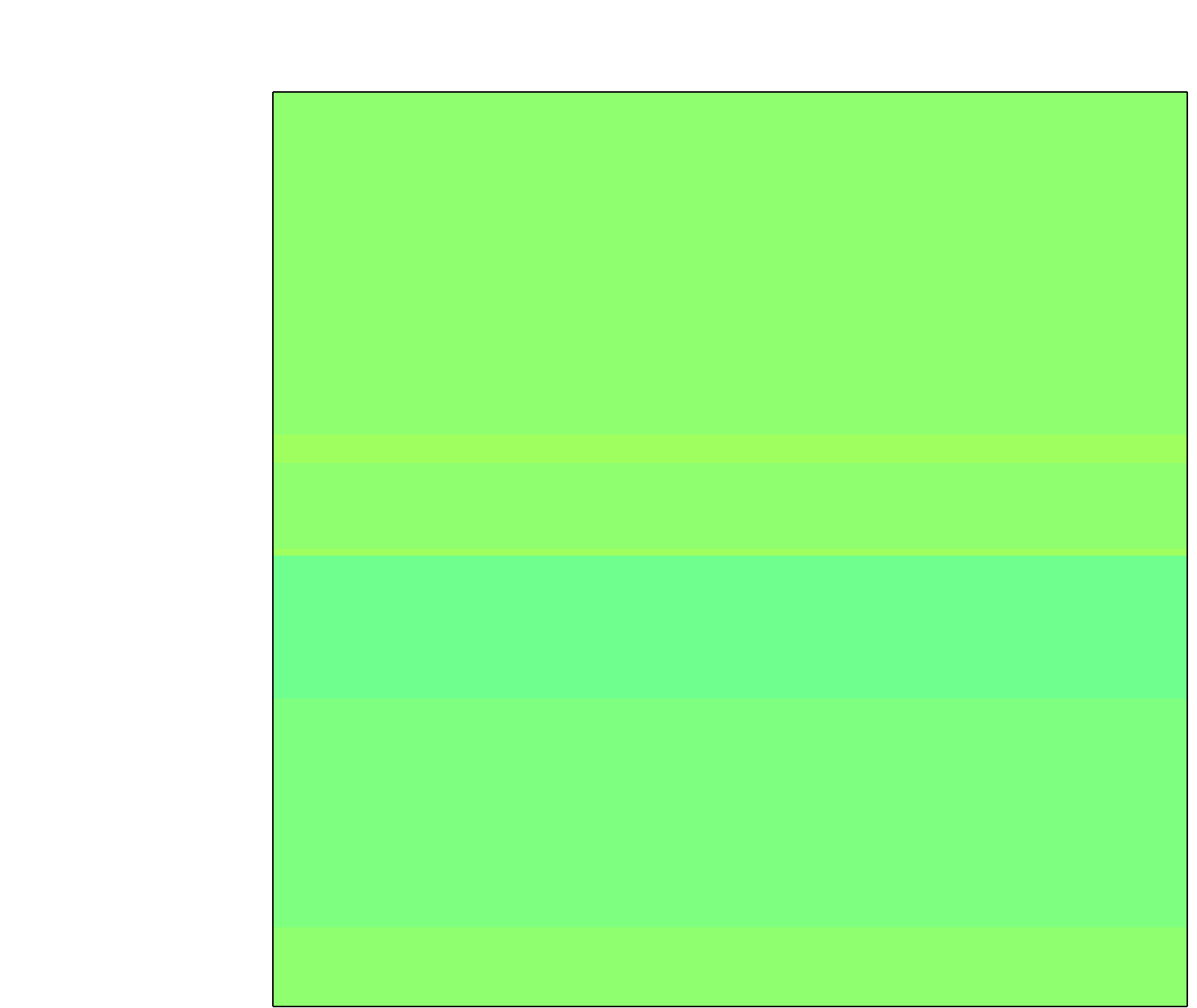}
  \end{center}
  \caption{Top: $\TE_{1j}(b)$ (in blue) and $\TE_{2j}(b)$ (in red) of Figure~\ref{fig:GB5GB8} (left) for fixed $b_1=0.5$ and $j=0,1,2$ from left to right, respectively. Center: The weighted average angle functions $\Angle_j(b)$ provided by each local wave vector for $j=0,1,2$, respectively.
Bottom left: the weighted boundary indicator function $\BD(b)$. Bottom right: 
the weighted average angle function $\Angle(b)$.}
  \label{fig:GB5result}
\end{figure}

\begin{figure}[ht!]
  \begin{center}
    \begin{tabular}{lll}
\includegraphics[height=1.3in]{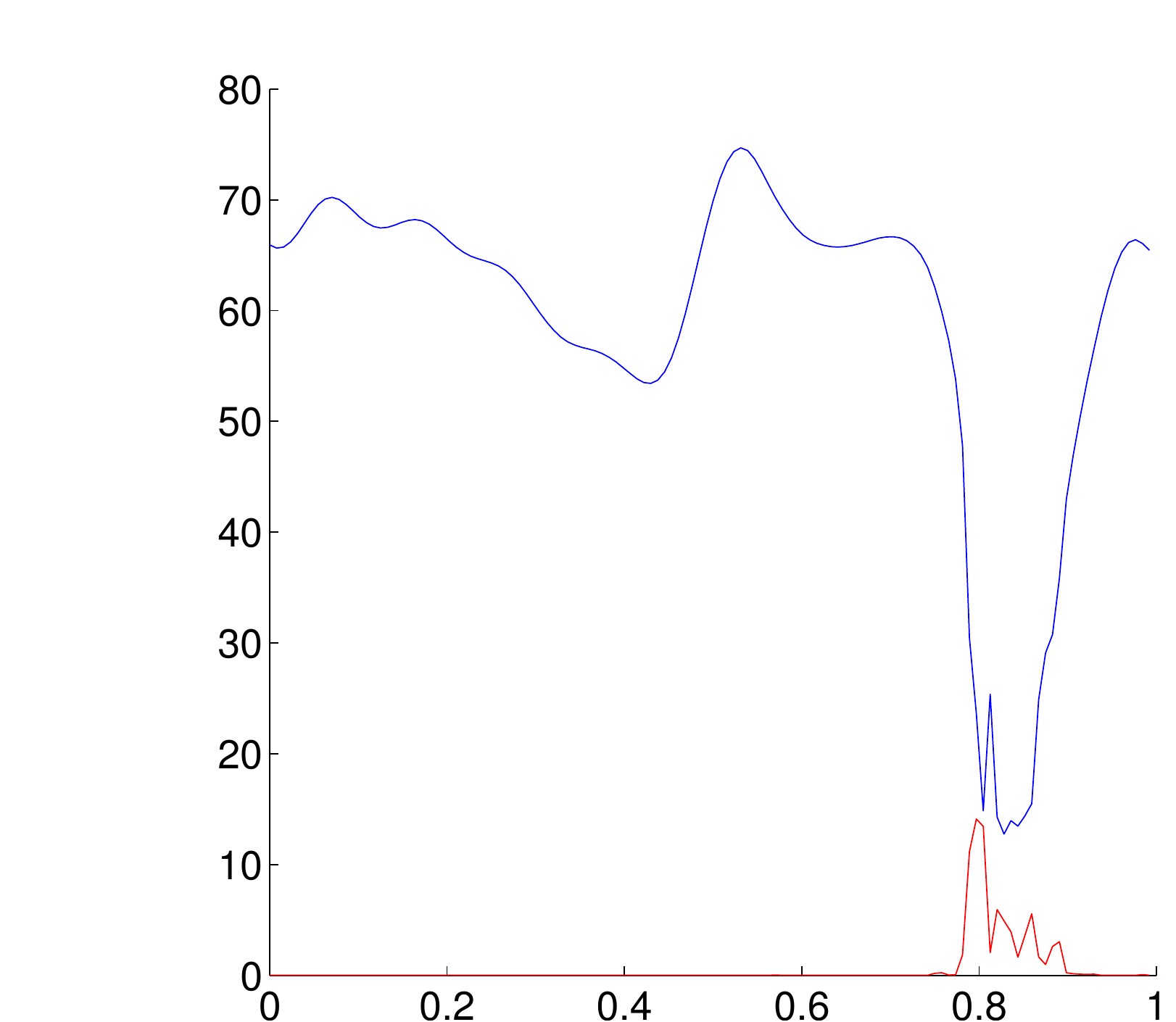} &  \includegraphics[height=1.3in]{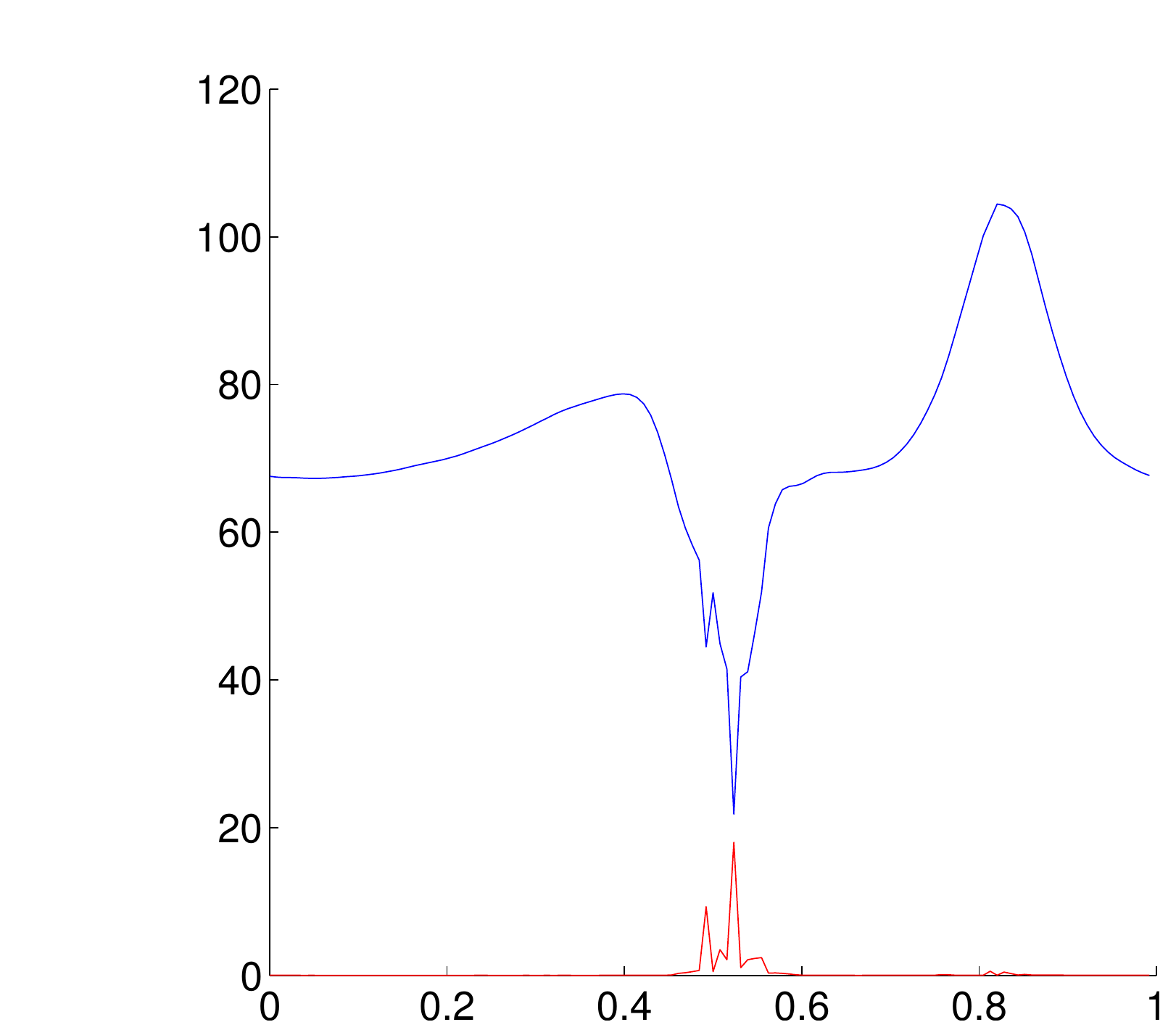}&  \includegraphics[height=1.3in]{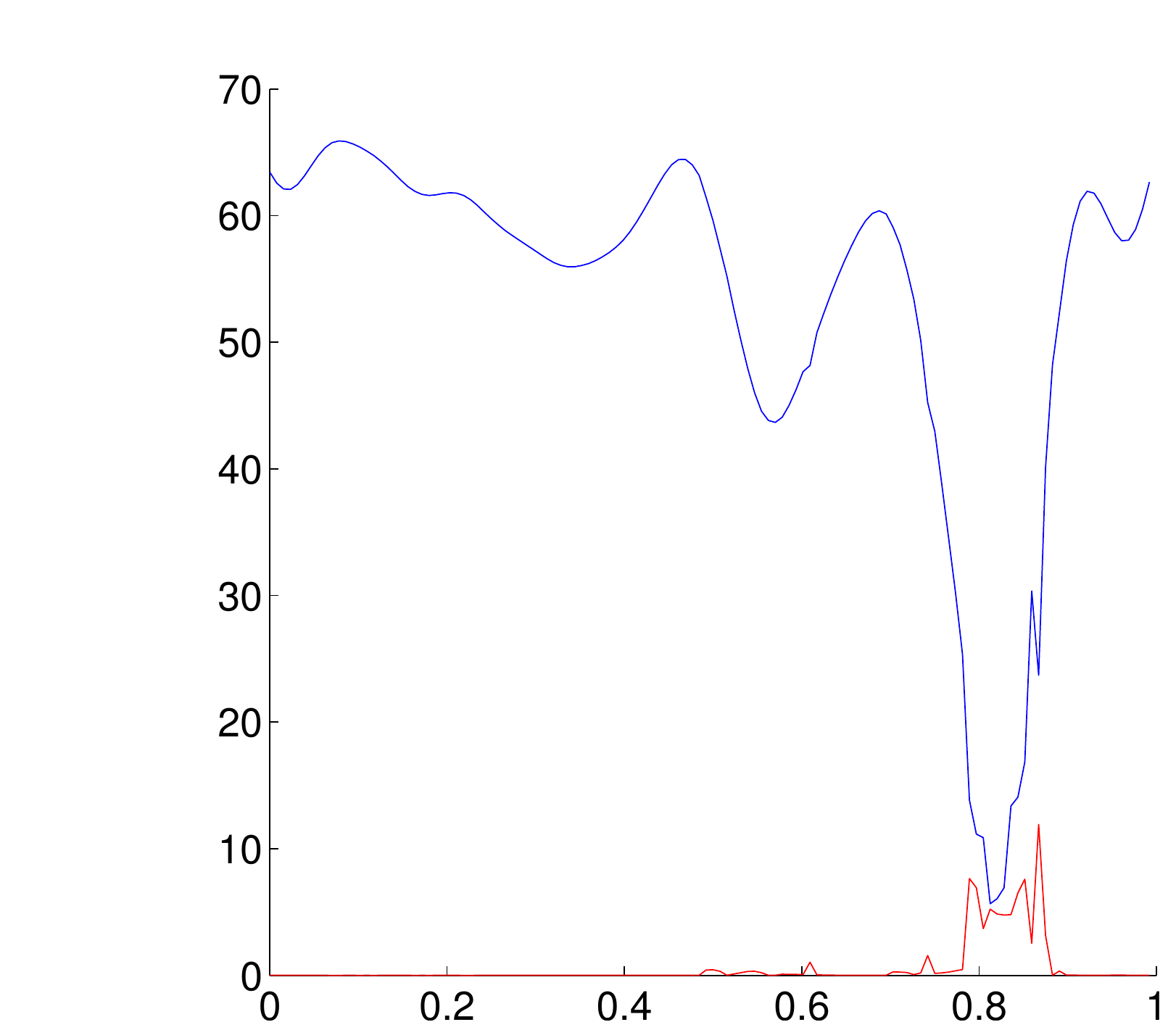} \\
  \includegraphics[height=1.3in]{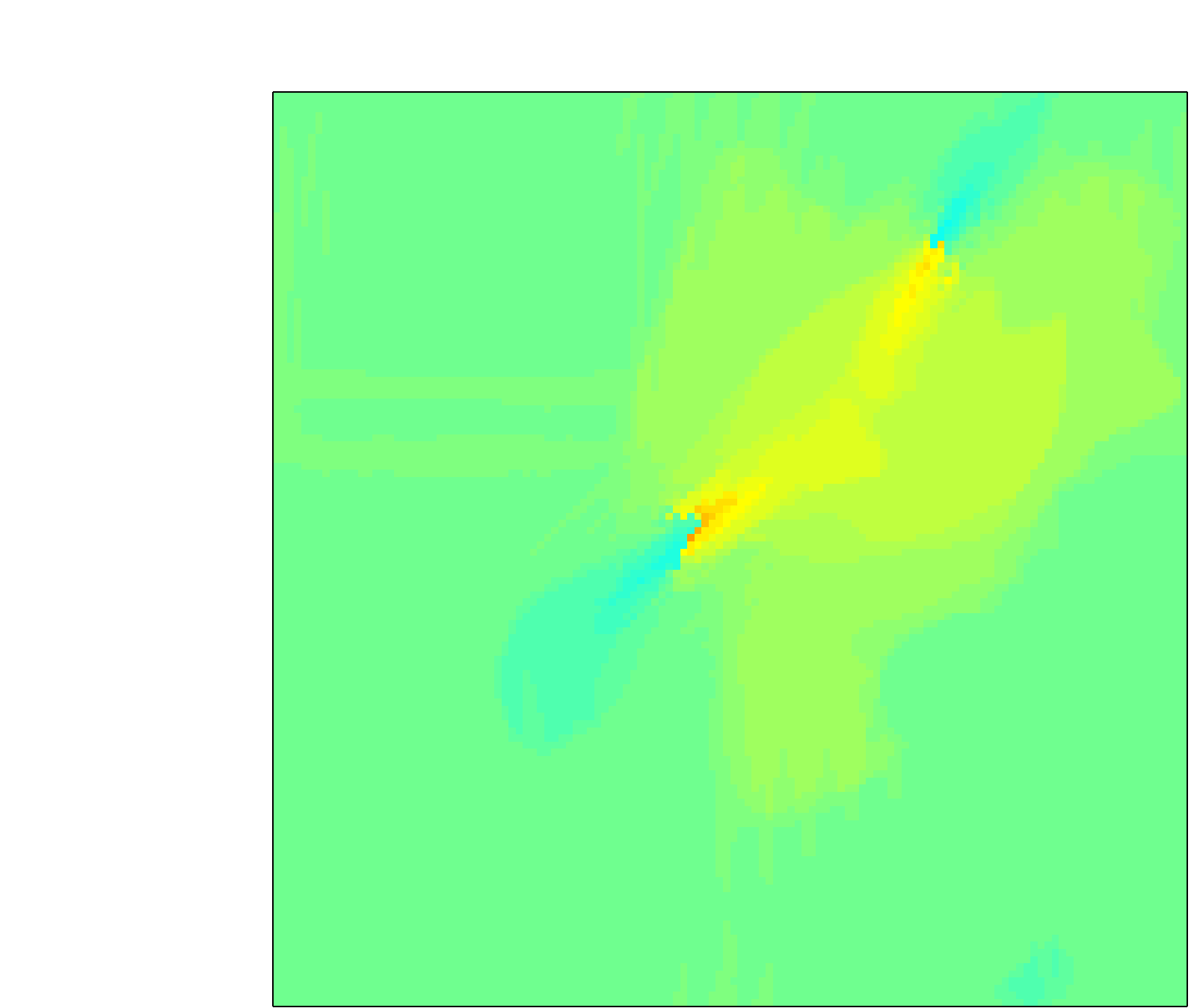} &  \includegraphics[height=1.3in]{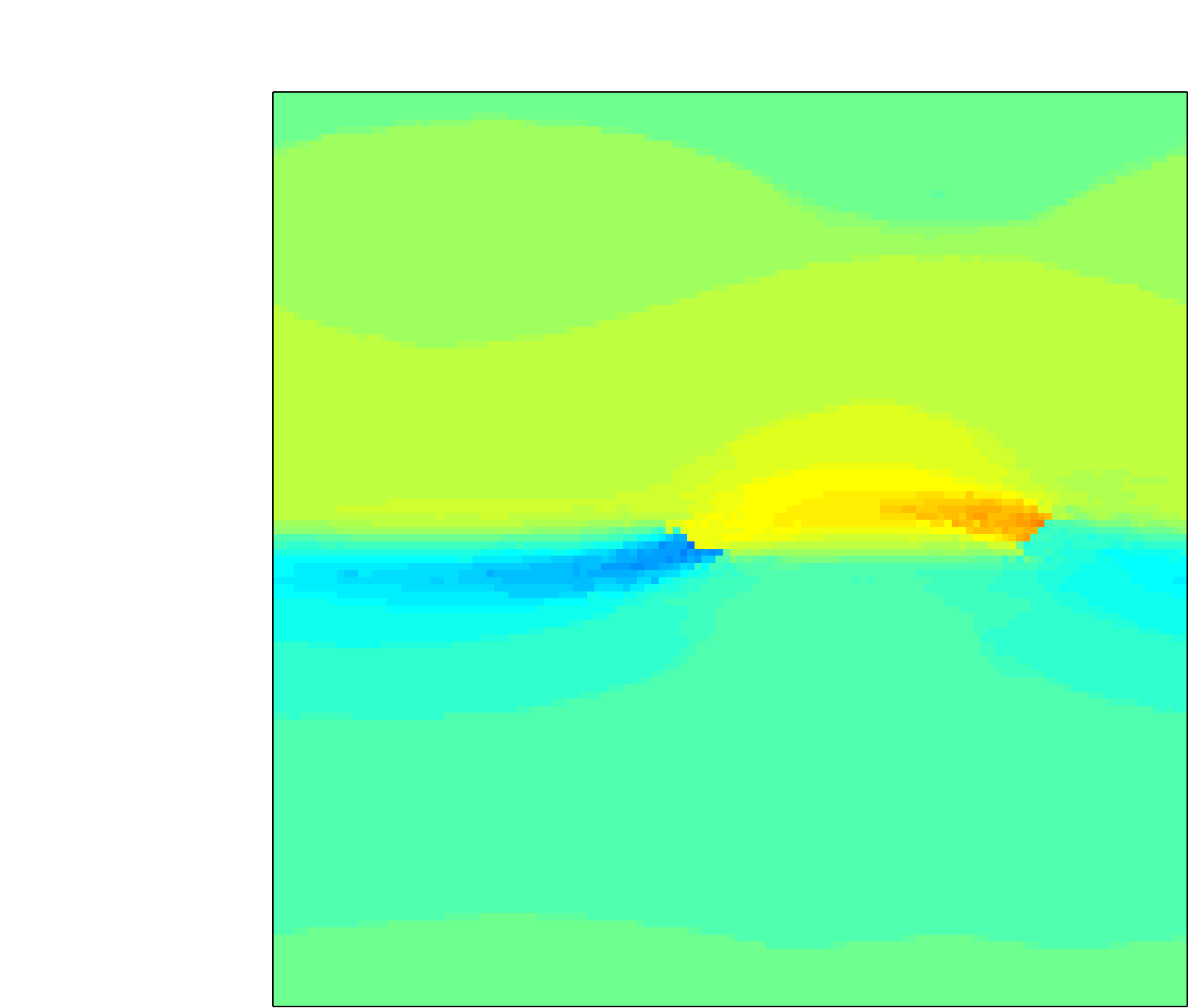}&  \includegraphics[height=1.3in]{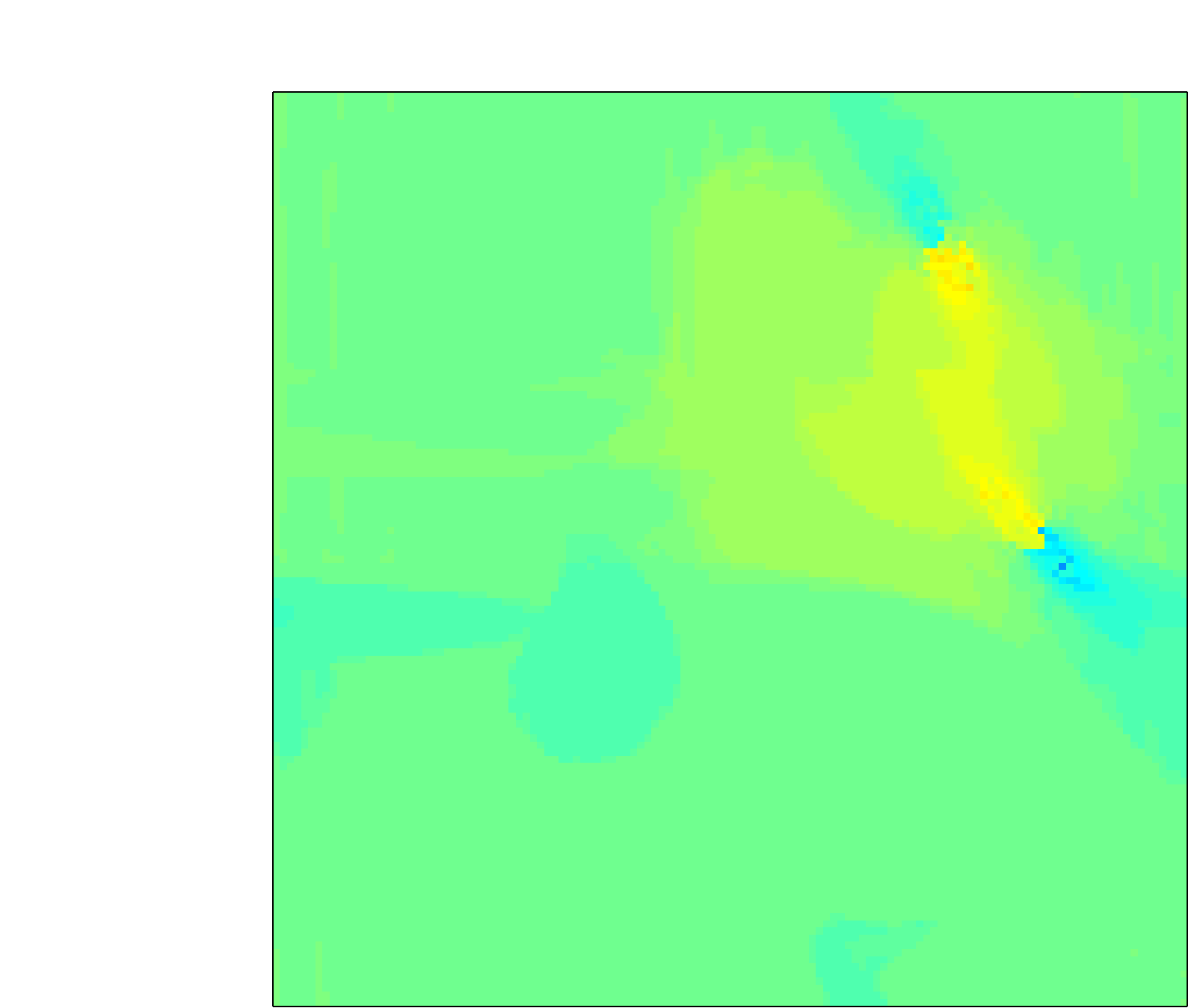} 
  \end{tabular}
\includegraphics[height=1.4in]{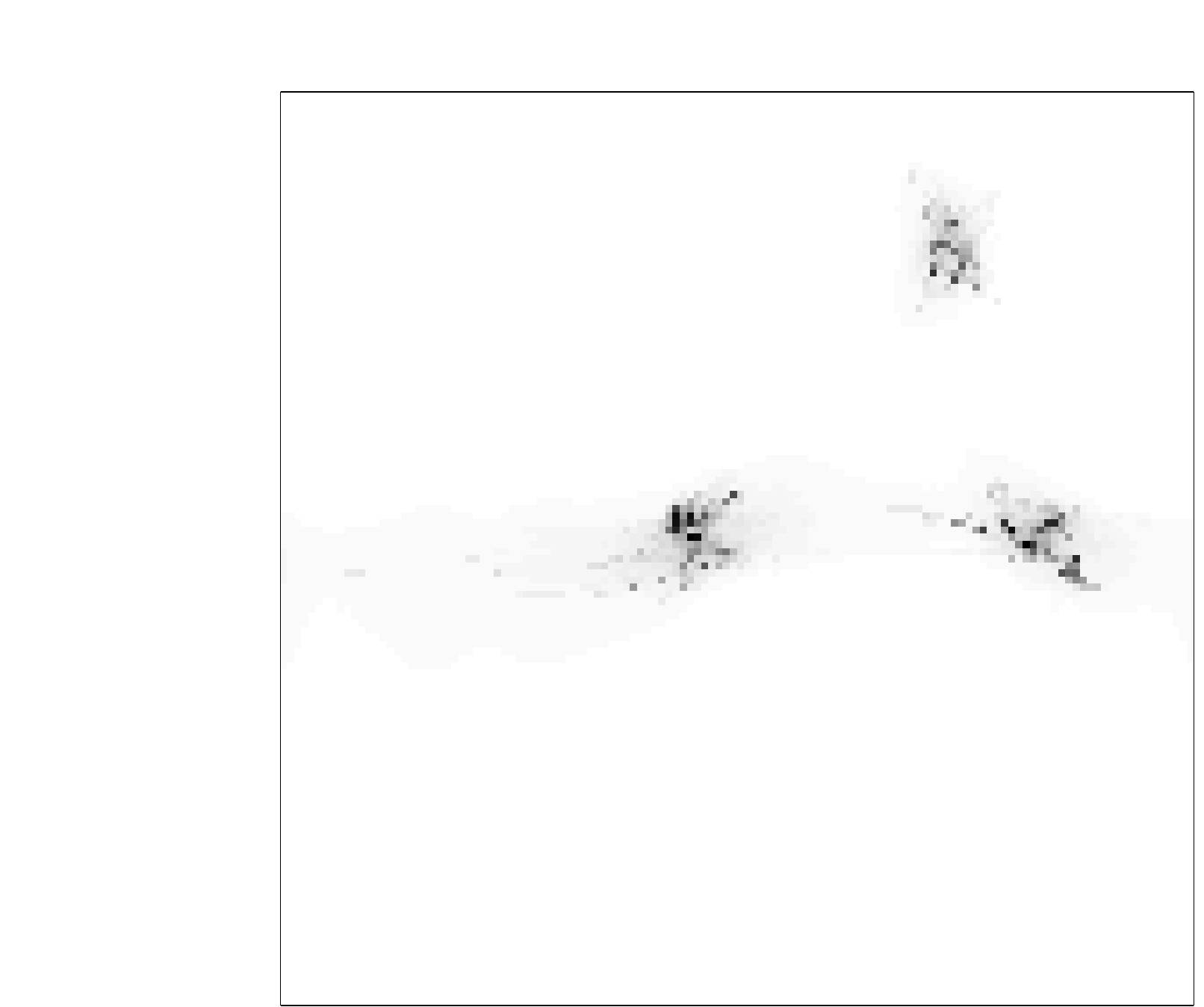} \qquad \includegraphics[height=1.4in]{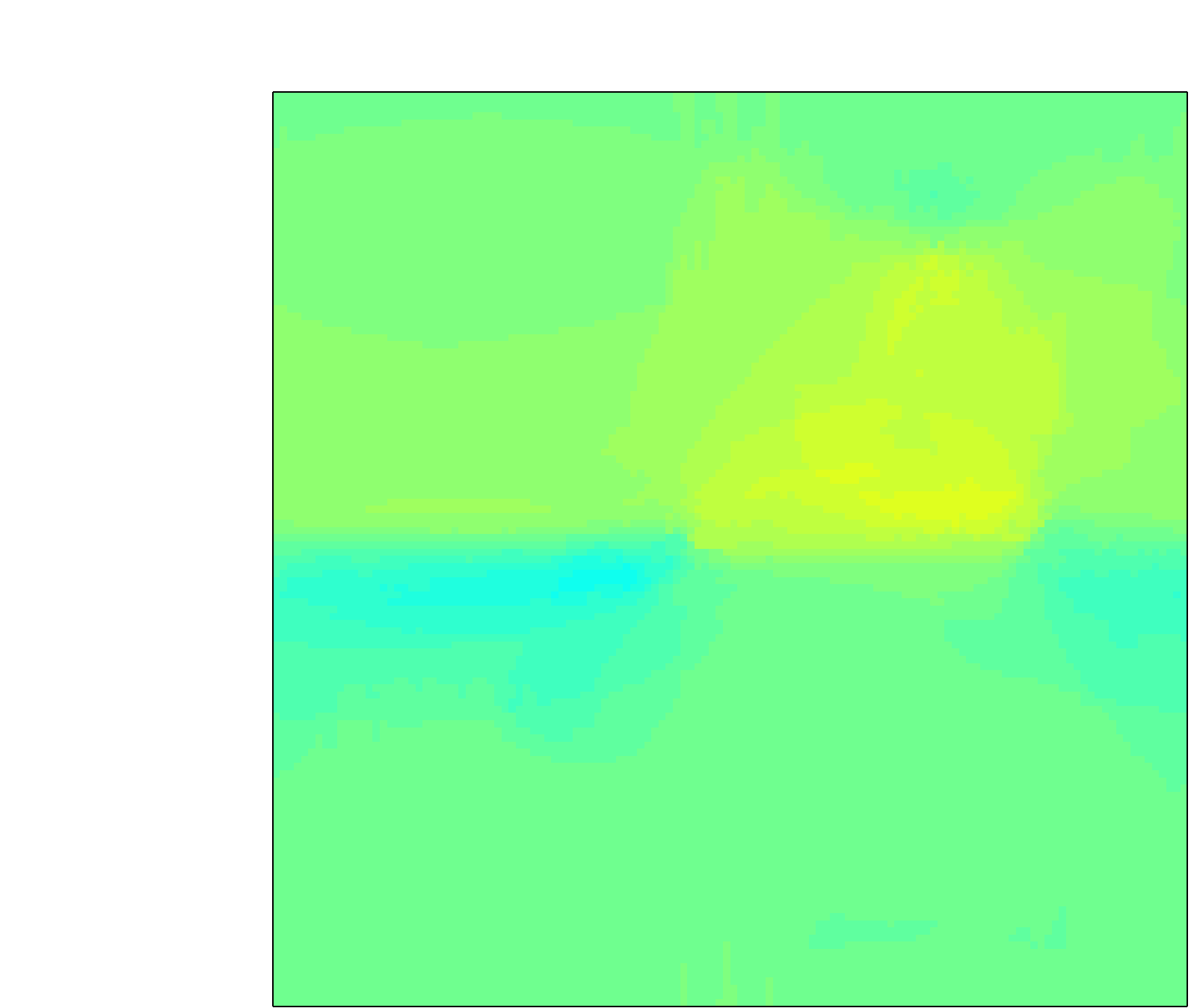}
  \end{center}
  \caption{Top: $\TE_{1j}(b)$ (in blue) and $\TE_{2j}(b)$ (in red) of Figure~\ref{fig:GB5GB8} (right) for  fixed $b_1=0.8$ and $j=0,1,2$, respectively. Center: The weighted average angle functions $\Angle_j(b)$ provided by each local wave vector for $j=0,1,2$, respectively.
Bottom left: the weighted boundary indicator function $\BD(b)$. Bottom right: the weighted average angle function $\Angle(b)$.}
  \label{fig:GB8result}
\end{figure}

In the example in Figure \ref{fig:GB5GB8} (left), only the second local wave vector exhibits irregularity at the small angle boundary, resulting in a sharp decrease of $\TE_{11}(b)$ and a sharp increase of $\TE_{21}(b)$ at the boundary. Hence, as shown in Figure \ref{fig:GB5result}, the weighted boundary indicator function $\BD(b)$ with $\BD_1(b)$ as a key integrand can clearly indicate the small angle boundary. As Figure \ref{fig:GB8result} shows, the top point dislocation in the example in Figure \ref{fig:GB5GB8} (right) interrupts the first and third underlying wave-like component, resulting in a sharp decrease in $\TE_{10}(b)$ and $\TE_{12}(b)$ and a sharp increase in $\TE_{20}(b)$ and $\TE_{22}(b)$. Therefore, the weighted boundary indicator function can reveal this point dislocation. The results in Figure \ref{fig:GB5result} and \ref{fig:GB8result} indicates that the information of some local defects is hidden behind some particular local wave vectors. By using the synchrosqueezed energy distribution of each local wave vector individually, more information of local defects can be discovered. 

\subsection{Recovery of local deformation gradient}

In addition to grain boundaries and crystal rotations, a reliable extraction of the elastic deformation of a crystal image is also essential for an efficient material characterization. Instead of estimating the elastic deformation $\phi(x)$ directly, we would emphasize how to recover the local deformation gradient 
\[
\grad\phi(x)=\left( \begin{array}{cc}
\partial_{x_1}\phi_1(x) & \partial_{x_2}\phi_1(x)\\
\partial_{x_1}\phi_2(x) & \partial_{x_2}\phi_2(x) 
\end{array}\right),\]
and how to read more information from $\grad\phi(x)$, e.g., directions of Burgers vectors. 

The estimation of the local deformation gradient $\grad\phi(x)$ relies on the complete estimate of at least two local wave vectors $\grad\left(Nn^\TT F\phi(x)\right)$. Let us continue with hexagonal crystal images as an example. In this case, there are six local wave vectors of interest  as discussed in Section \ref{sec:BumpDetection}. By symmetry, it is enough to consider those in the upper half plane of the Fourier domain. They are 
\[
v_j(x)=\left(\grad\phi(x)\right)^\TT \left(N\cos(\frac{j\pi}{3}),N\sin(\frac{j\pi}{3})\right)^\TT,\quad j=0,1,2.
\]

In Algorithm~$\ref{alg:deformed}$, the argument of each $v_j(x)$ has been estimated by the weighted average angle $\Angle_j(b)$. Similarly, one more step for applying the bump detection algorithm radially can provide a fast estimate of the length of $v_j(x)$. Suppose $\wt{v}_j(x)$ is an estimate of $v_j(x)$ obtained by the fast algorithms above, then we have an over-determined linear system at each $x$
\[
\wt{v}_j(x)\approx\left(\grad\phi(x)\right)^\TT \left(N\cos(\frac{j\pi}{3}),N\sin(\frac{j\pi}{3})\right)^\TT,\quad j=0,1,2.
\]
Notice that the reciprocal number $N$ can be estimated by $\arg\max E(r)$, where $E(r)$ is the radially average Fourier power spectrum of the given crystal image. Hence, a least square method is sufficient to provide a good estimate $\grad\wt{\phi}(x)$ of the local deformation gradient $\grad\phi(x)$. 

As an example of reading information from the local deformation gradient $\grad\phi(x)$, a local volume distortion estimate of $\det\left(\grad\phi(x)\right)-1$ is computed by
\[
\Vol(x) = \det\left(\frac{\grad\wt{\phi}(x)}{\int_{\Omega}\det\left(\grad\wt{\phi}(x)\right)\ud x/\left|\Omega\right|}\right)-1,
\]
where $\Omega$ is the domain of the crystal image and $\left|\Omega\right|$ denotes its area. The normalization reduces the influence of the estimate error of $N$. In the presence of a lattice distortion of a dislocation, a Burgers vector is introduce to represent this distortion. On one side of a Burgers vector, the space between atoms is slightly compressed, while the space on the other side is expanded. Hence, the local volume distortion would be positive on one side of the Burgers vector and negative on the other side. Suppose $\Vol(x_1)$ and $\Vol(x_2)$ are the maximum and the minimum of $\Vol(x)$ in a local region near the dislocation. Then $\Vol(x_1)-\Vol(x_2)$ reflects the length of the Burgers vector and the vector $x_1-x_2$ is orthogonal to the Burgers vector. 

Figure \ref{fig:GB8vol} shows the local distortion volume estimate $\Vol(x)$ of the example in Figure~\ref{fig:GB5GB8} (right). $\Vol(x)$ reaches its local maximum and local minimum near each point dislocation and the direction of the corresponding Burgers vector can be directly read off from the color coding of $\Vol(x)$. As a remark, $\Vol(x)$ is not tightly supported around local dislocations in this example. This is because the synchrosqueezed transforms cannot be too localized in space and the estimate of $\grad\phi(x)$ is polluted by the information at the other points nearby. We leave further regularization of these estimates to future works.

\begin{figure}[ht!]
  \begin{center}
    \begin{tabular}{c}
      \includegraphics[height=1.6in]{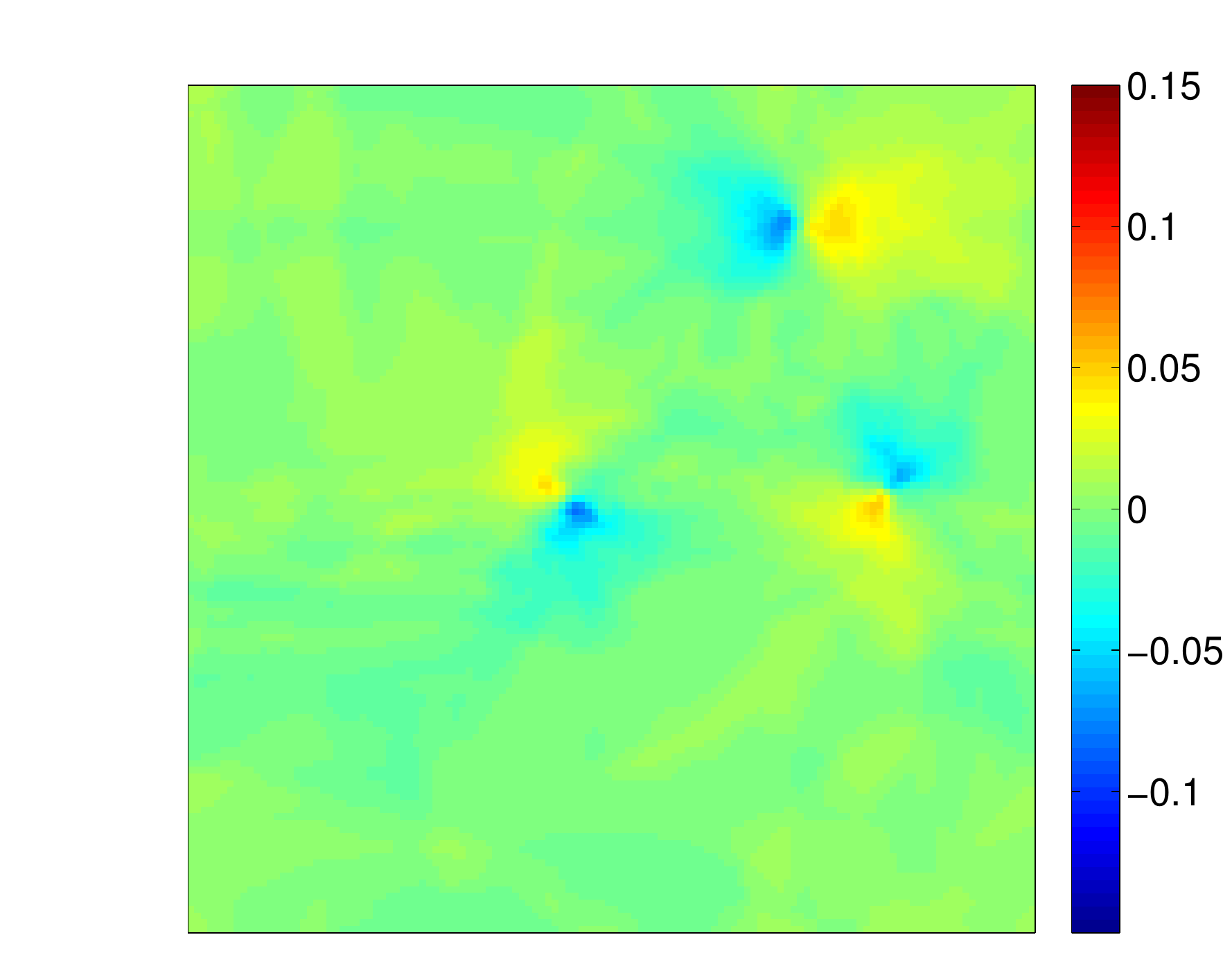} 
    \end{tabular}
  \end{center}
  \caption{The local distortion volume $\Vol(x)$ of the example in Figure~\ref{fig:GB5GB8} (right).  To make the image more readable, the color limits of these two images are set to be $0.15$ and $-0.15$ as the maximum and the minimum, respectively.}
  \label{fig:GB8vol}
\end{figure}

\section{Examples and discussions}
\label{sec:examples}

This section contains a series of experiments of synthetic and real images to illustrate the performance of Algorithm $\ref{alg:undeformed}$ and $\ref{alg:deformed}$. In the first part of this section, we focus on the application of Algorithm $\ref{alg:undeformed}$ to detect grain boundaries and to estimate crystal rotations. The robustness of this method will be emphasized and supported by some noisy examples and examples with blurry grain boundaries. In the second part, Algorithm $\ref{alg:deformed}$ is applied to two more examples with different types of point dislocations. Algorithm $\ref{alg:deformed}$ is more sensitive to local defects and is able to discover Burgers vectors of these dislocations by estimating the local volume distortion $\Vol(b)$. Throughout all examples, the threshold value $\epsilon$ for the synchrosqueezed transforms is $10^{-4}$ and other primary parameters such as $s$, $t$ and $d$ are application dependent. Generally speaking, large parameters $s$, $t$ and $d$ result in more sensitive synchrosqueezed transforms to local defects, while smaller parameters provide more stable analysis results.

Recall that both Algorithm $\ref{alg:undeformed}$ and $\ref{alg:deformed}$ have a time complexity $O(L^2\log L+N^{1-s}L_B^2\log L_B+L_A L_R L_B^2)$. Since the time complexity depends on the reciprocal number $N$ of a given crystal image and $N$ is not known a priori, we would not emphasize the runtime taken to analyze images of different sizes. All the numerical results except for the one in Figure \ref{fig:GB1_elastic} and \ref{fig:GB1vol} take within $10$ seconds using a Matlab code in a MacBook Pro with a 2.7GHz quad-core Intel Core i7 CPU. The result in Figure \ref{fig:GB1_elastic} and \ref{fig:GB1vol} takes about $50$ seconds and the most expensive part is for the bump detection. The time expense will be reduced significantly by straightforward parallelization, though this is not our focus in this paper. The codes of these algorihtms together with some numerical examples are open source and available as \texttt{SynLab} at
\url{https://github.com/HaizhaoYang/SynLab}.

\subsection{Examples for Algorithm~$\ref{alg:undeformed}$}

Our first example is a phase field crystal (PFC)
image in Figure~\ref{fig:GB1_fast} (left). It contains several grains with low
and high angle grain boundaries and some point dislocations. As shown in Figure \ref{fig:GB1_fast} (middle), the weighted
average angle $\Angle(b)$ is changing gradually in the interior of a grain and jumps at a large angle boundary.
Figure \ref{fig:GB1_fast} (right) shows the boundary indicator function $\BD(b)$. Large angle grain boundaries appear in a form of line segments. Adjacent point dislocations are connected and identified as grain boundaries, while well-separated point dislocations are identifed by light grey regions.  As pointed out in Section
\ref{sec:alg}, some isolated point dislocations may be missed due to the stacking step, while Algorithm~$\ref{alg:deformed}$ is better suited for detecting local defects, as will be shown in Figure \ref{fig:GB1_elastic} and Figure \ref{fig:GB3_elastic}.

\begin{figure}[ht!]

  \begin{center}
    \hspace{-4em}
      \includegraphics[height=1.6in]{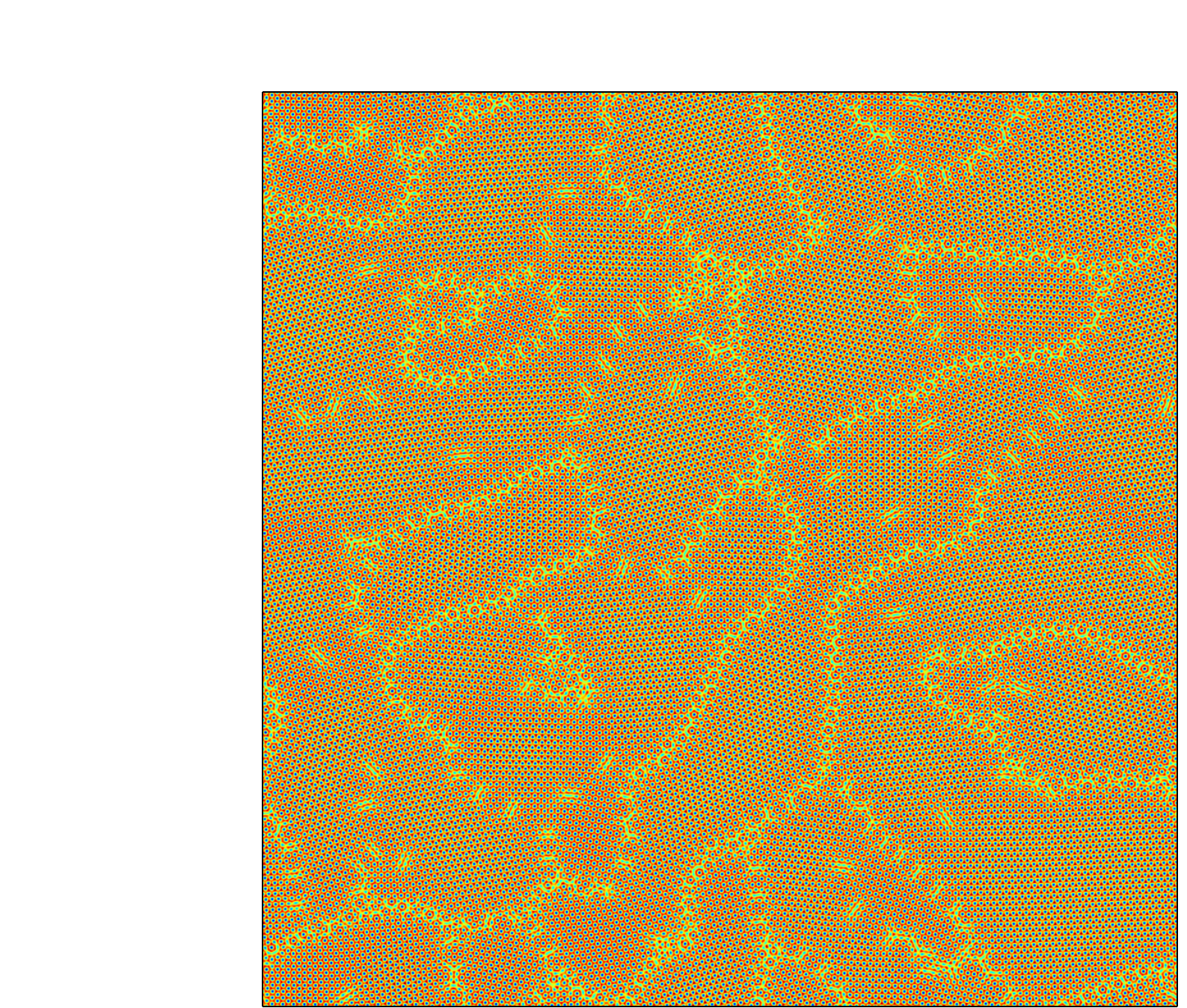}  \includegraphics[height=1.6in]{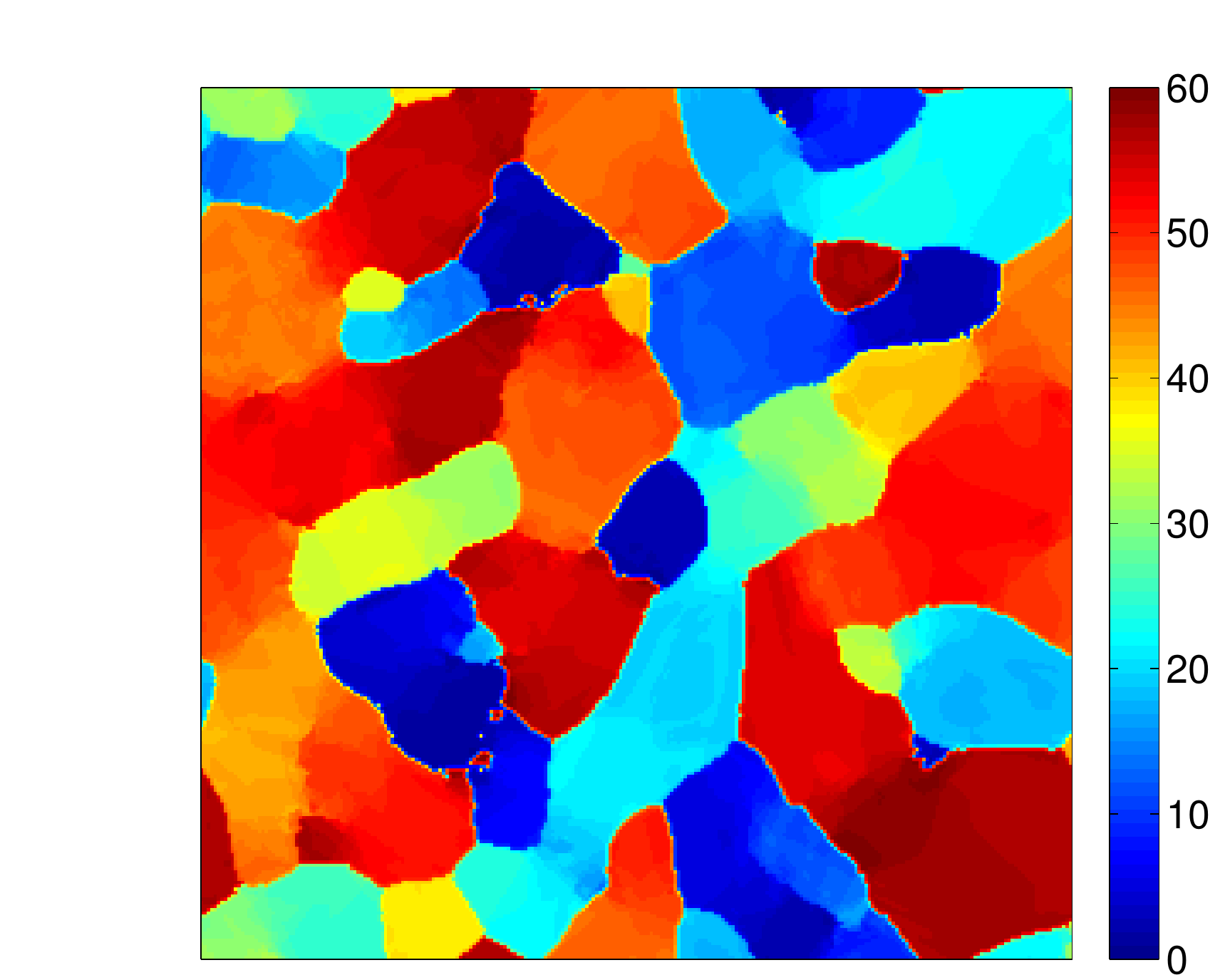}  \includegraphics[height=1.6in]{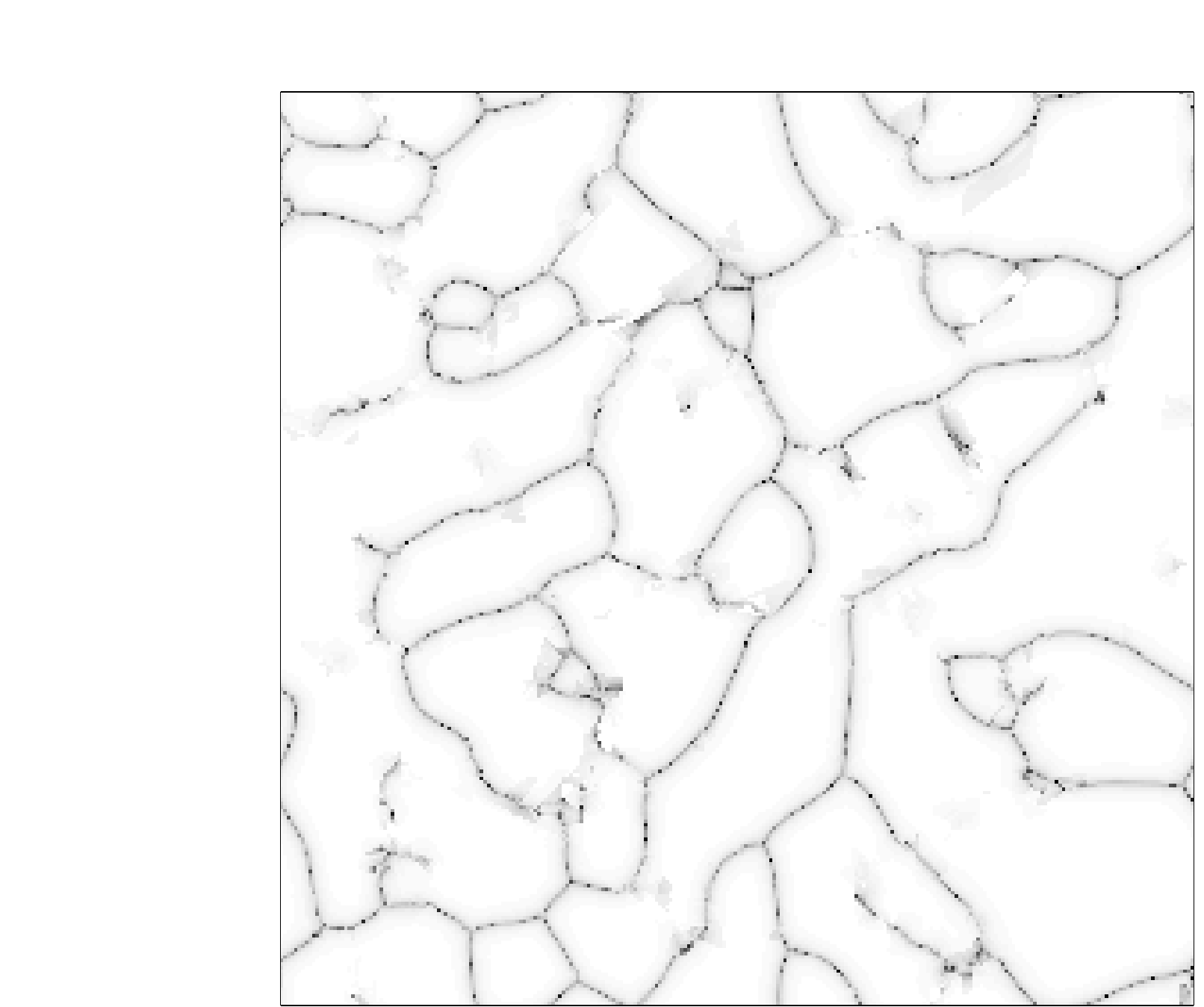}\\
    \hspace{-4em}
      \includegraphics[height=1.6in]{GB1_image_zoomed.pdf}  \includegraphics[height=1.6in]{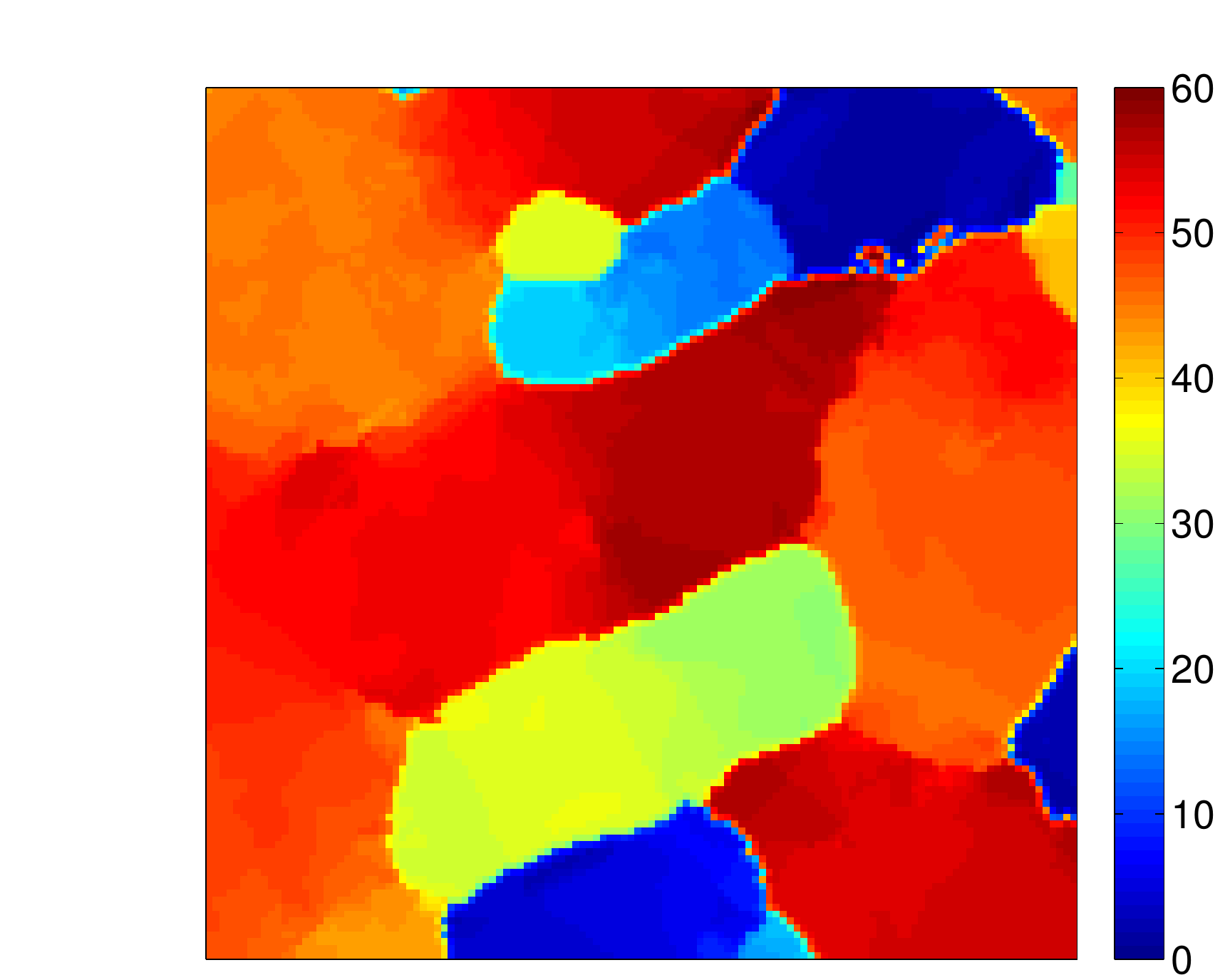}  \includegraphics[height=1.6in]{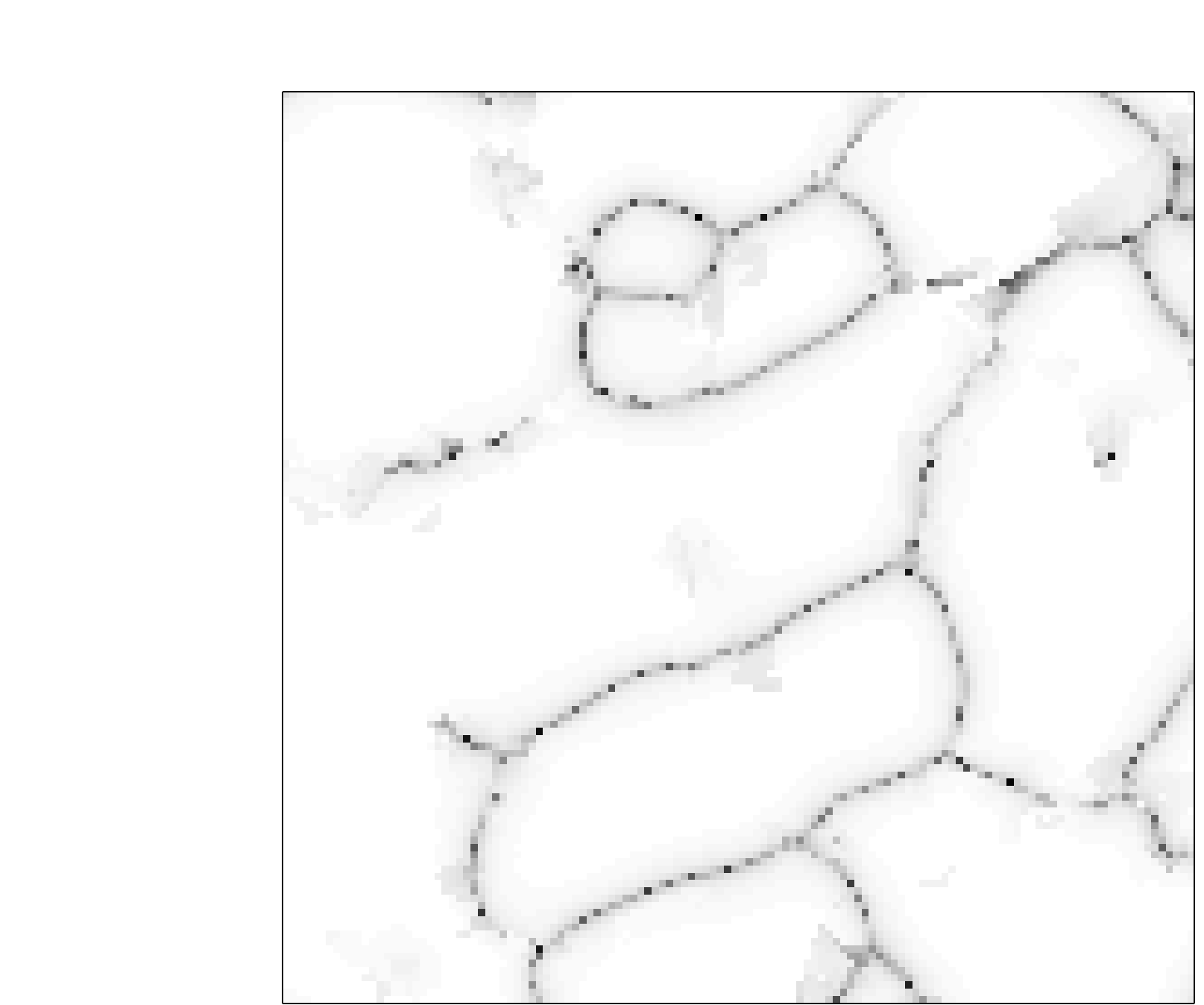}
  \end{center}
  \caption{Analysis results of a large phase field crystal image (of size $1024\times 1024$ pixels) provided by Algorithm $\ref{alg:undeformed}$. The top row is the original image and the bottom row gives the zoomed-in view. Left: A phase field crystal (PFC) image and its zoomed-in image. Courtesy of Benedikt Wirth \cite{ElseyWirth:MMS}. Middle: The weighted average angle $\Angle(b)$ and its zoomed-in result. Right: The boundary indicator function $\BD(b)$ and its zoomed-in result. Primary implementation parameters: $s=t=0.5$, $d=1$.}
  \label{fig:GB1_fast}
\end{figure}

We next consider a real data example of a twin boundary in a TEM-image in GaN (see Figure \ref{fig:GB2_fast} (left)). Algorithm $\ref{alg:undeformed}$ identifies two grains as shown in Figure \ref{fig:GB2_fast} (middle) and estimates their rotation angles. The grain boundary is approximated by a smooth curve in the image of the boundary indicator function $\BD(b)$ in Figure \ref{fig:GB2_fast} (right).

\begin{figure}[ht!]
  \begin{center}
    \hspace{-4em}
      \includegraphics[height=1.6in]{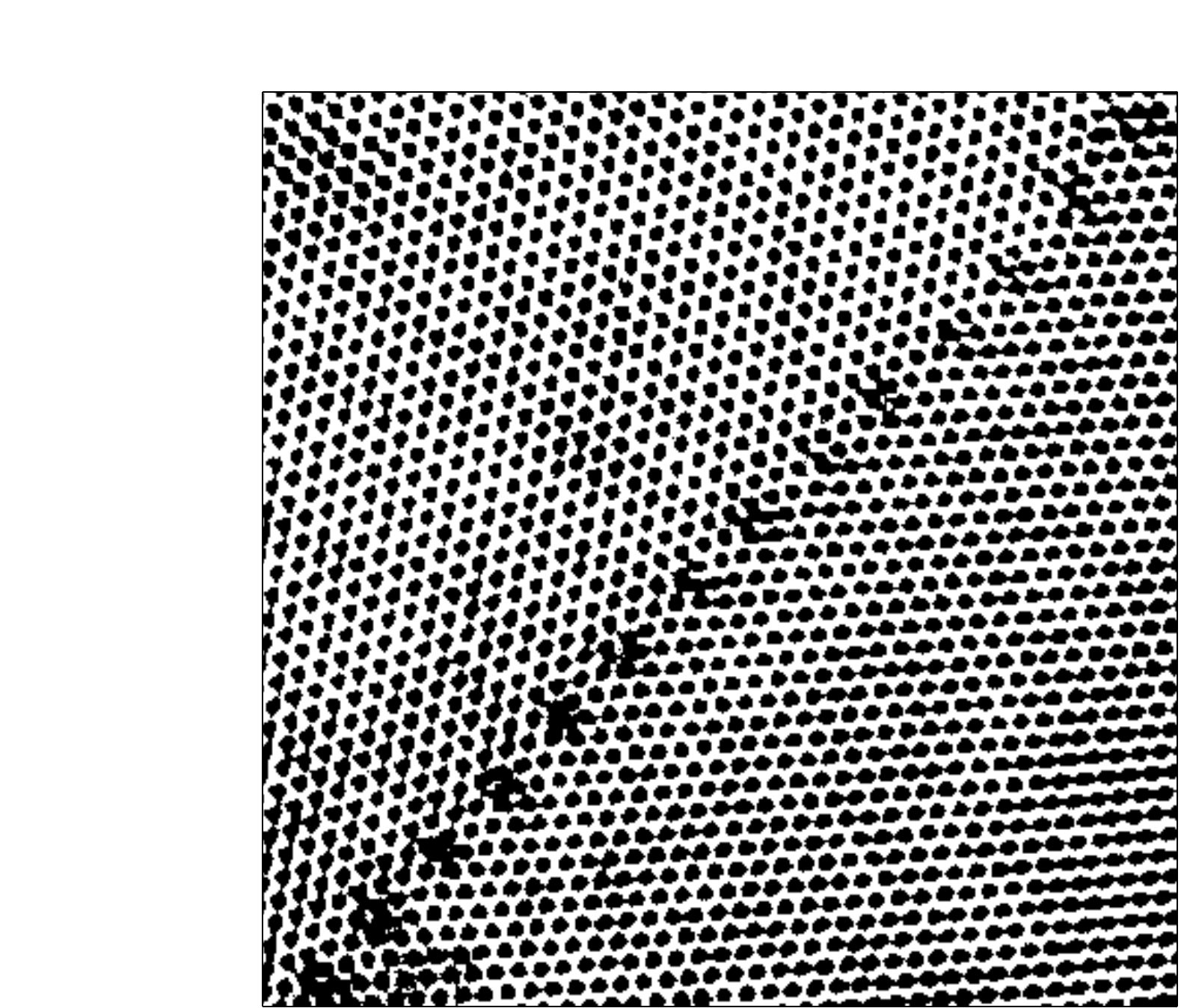}  \includegraphics[height=1.6in]{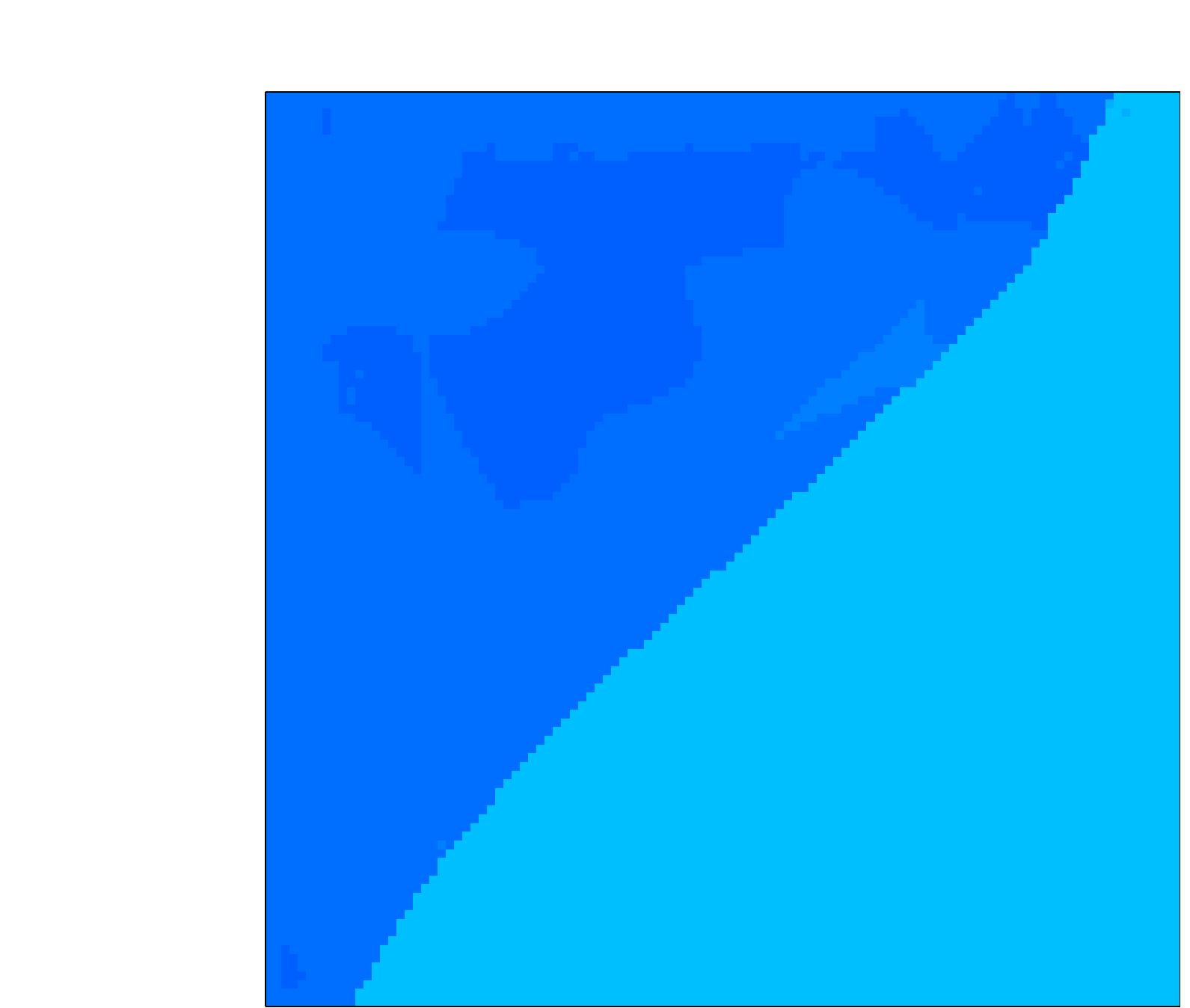}  \includegraphics[height=1.6in]{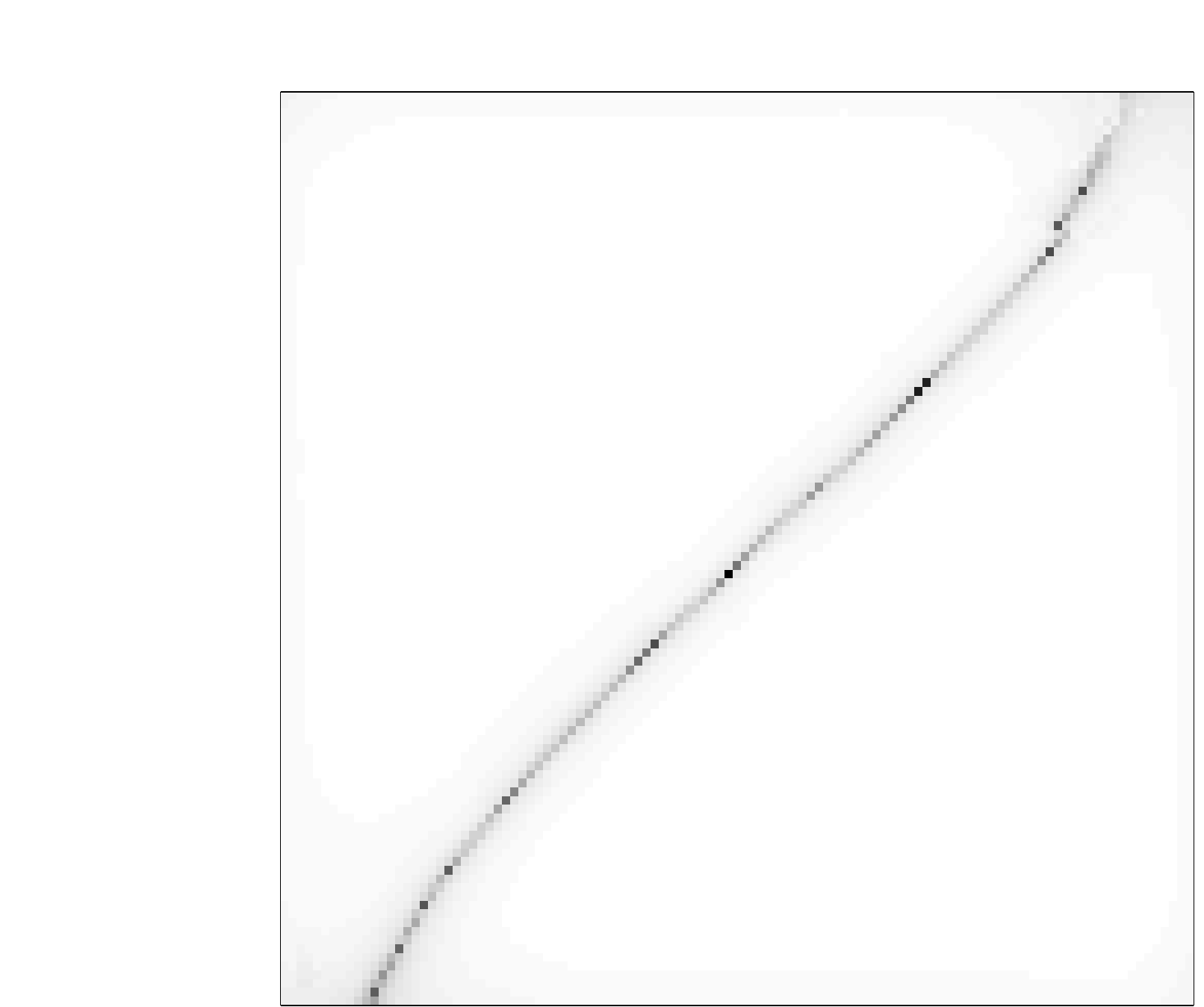}
  \end{center}
  \caption{Left: A TEM-image of size $420\times 444$ pixels in GaN. Courtesy of David M. Tricker (Department of Material Science and Metallurgy, University of Cambridge). Middle: The weighted average angle $\Angle(b)$ provided by Algorithm $\ref{alg:undeformed}$. Right: The boundary indicator function $\BD(b)$ provided by Algorithm $\ref{alg:undeformed}$. Primary implementation parameters: $s=t=0.5$, $d=0.5$.}
  \label{fig:GB2_fast}
\end{figure}

Figure \ref{fig:GB6_fast} shows an example of blurry boundaries in a photograph of a bubble raft. In this case, the transition between two grains is not sharp due to large distortion, in particular the local crystal structure is heavily disturbed near the boundary. Nevertheless, Algorithm $\ref{alg:undeformed}$ is capable of identifying three grains with sharp grain boundaries matching the distortion area as shown in Figure \ref{fig:GB6_fast} (middle) and \ref{fig:GB6_fast} (right).

\begin{figure}[ht!]
  \begin{centering}
    \hspace{-2em}
      \includegraphics[height=1.2in]{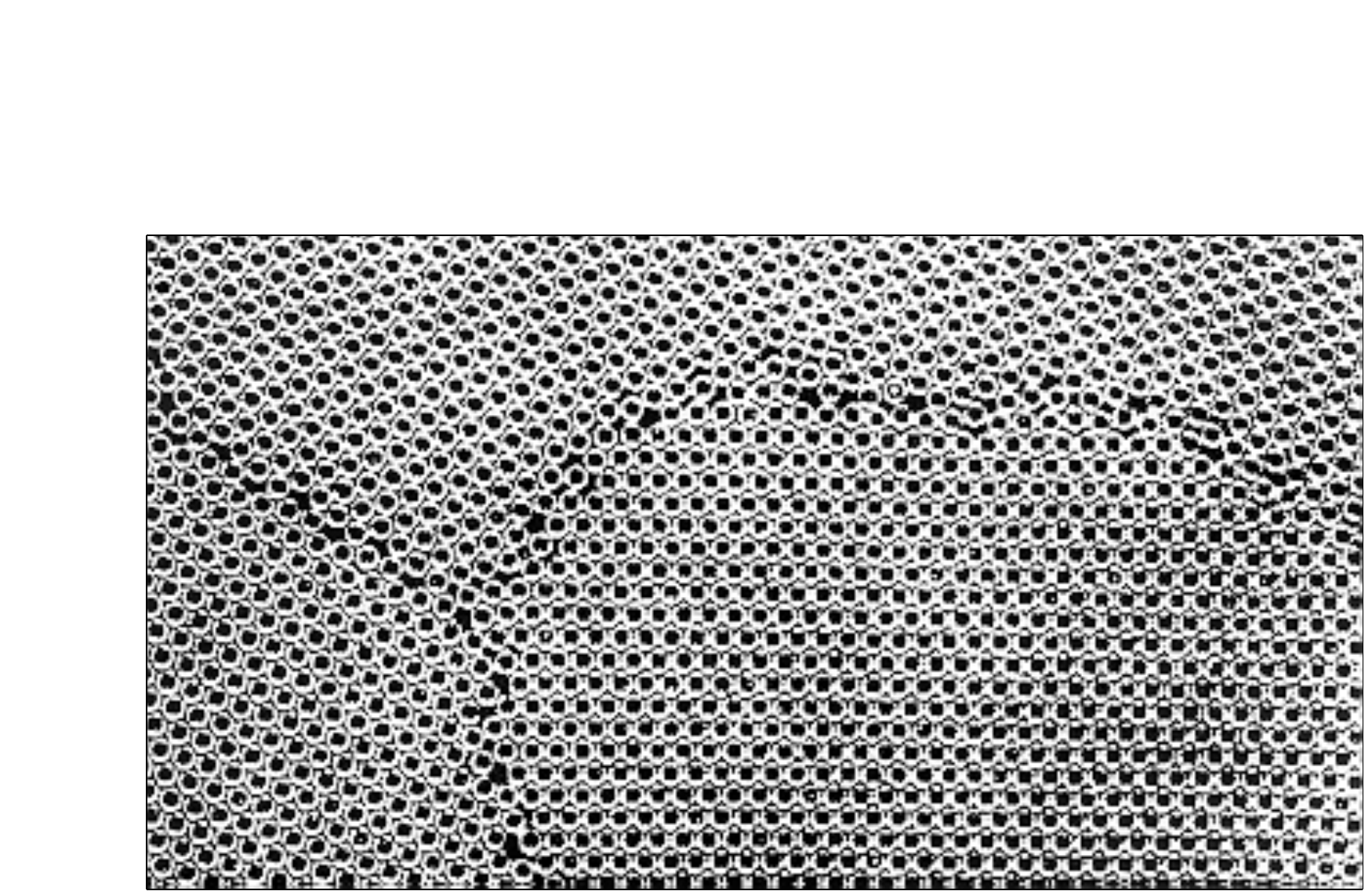}  \includegraphics[height=1.3in]{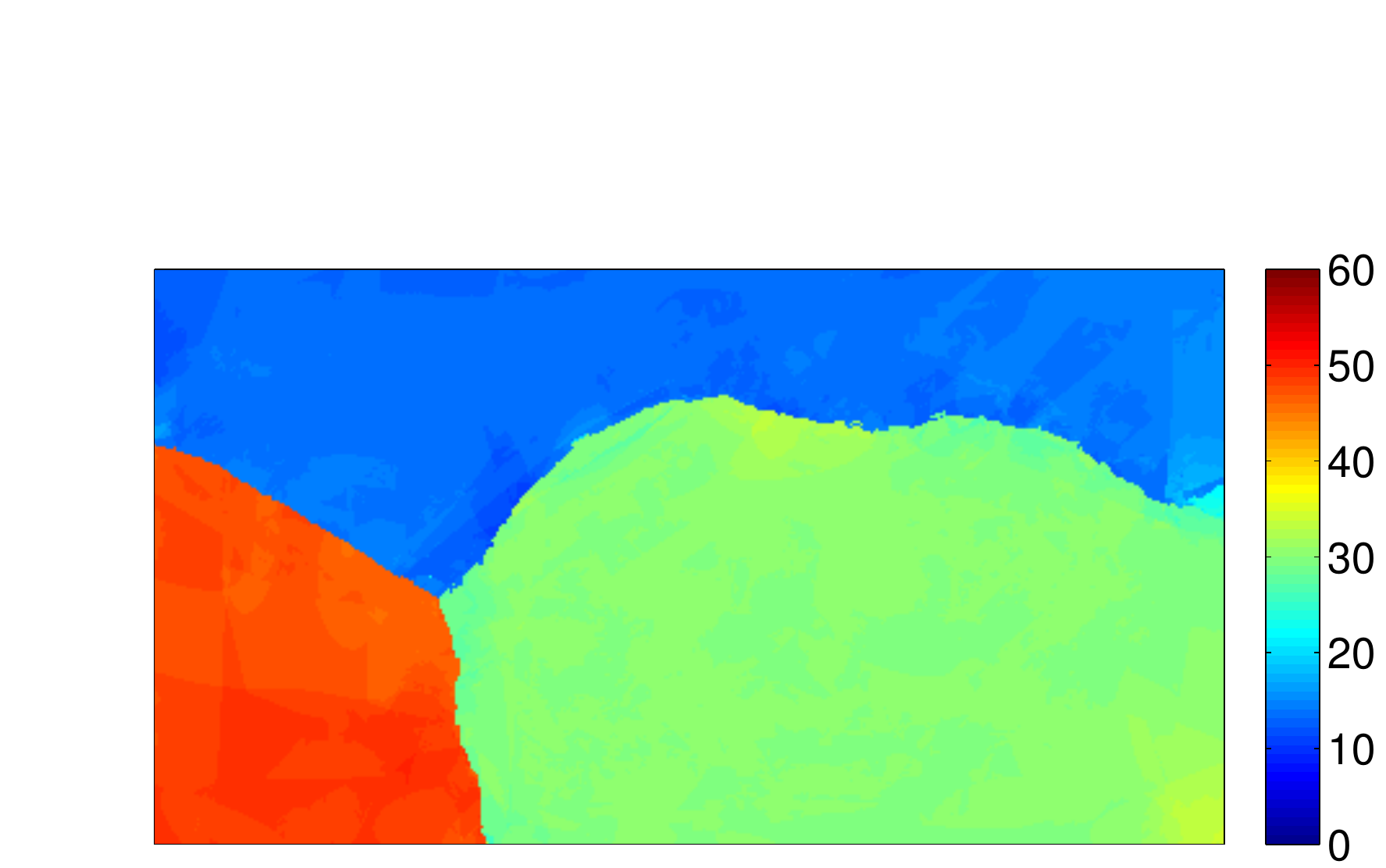}  \includegraphics[height=1.2in]{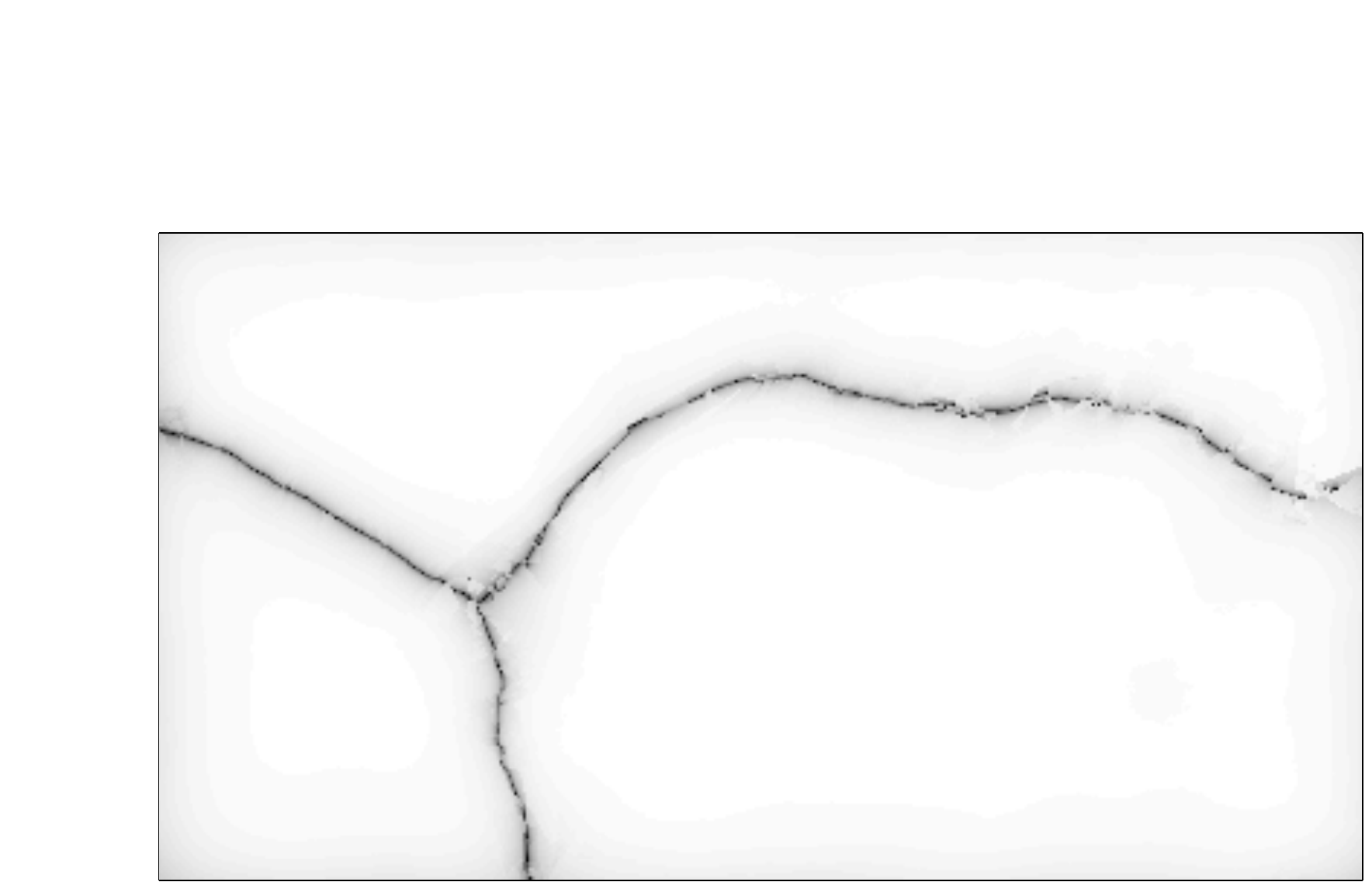}
  \end{centering}
  \caption{Left: A photograph of a bubble raft with large disorders and blurry boundaries. Courtesy to Barrie S. H. Royce in Princeton University. The image size is $223\times 415$. Middle and right: The weighted average angle $\Angle(b)$ and the boundary indicator function provided by Algorithm $\ref{alg:undeformed}$. Primary implementation parameters: $s=t=0.5$, $d=1$.}
  \label{fig:GB6_fast}
\end{figure}

To show the ability of our algorithm to analyze crystal images with strong illumination artifacts in the imaging like reflections and shadows, we analyze the example shown in Figure \ref{fig:GB3_fast} (left), where the reflections and shadows exhibit strong spatial variations. Algorithm $\ref{alg:undeformed}$ identifies two grains as shown in Figure \ref{fig:GB3_fast} (middle). The weighted average angle $\Angle(b)$ is slowly varying, indicating a smooth deformation over the domain. A large angle grain boundary consisting of an array of close point dislocations is identified by $\BD(b)$ in Figure \ref{fig:GB3_fast} (right). This example also shows that  Algorithm $\ref{alg:undeformed}$ may miss isolated defects, which motivates the design of Algorithm $\ref{alg:deformed}$. We will revisit this example later with Algorithm $\ref{alg:deformed}$ which will discover all the missing point defects. 

\begin{figure}[ht!]
  \begin{center}
    \hspace{-4em}
    \includegraphics[height=1.6in]{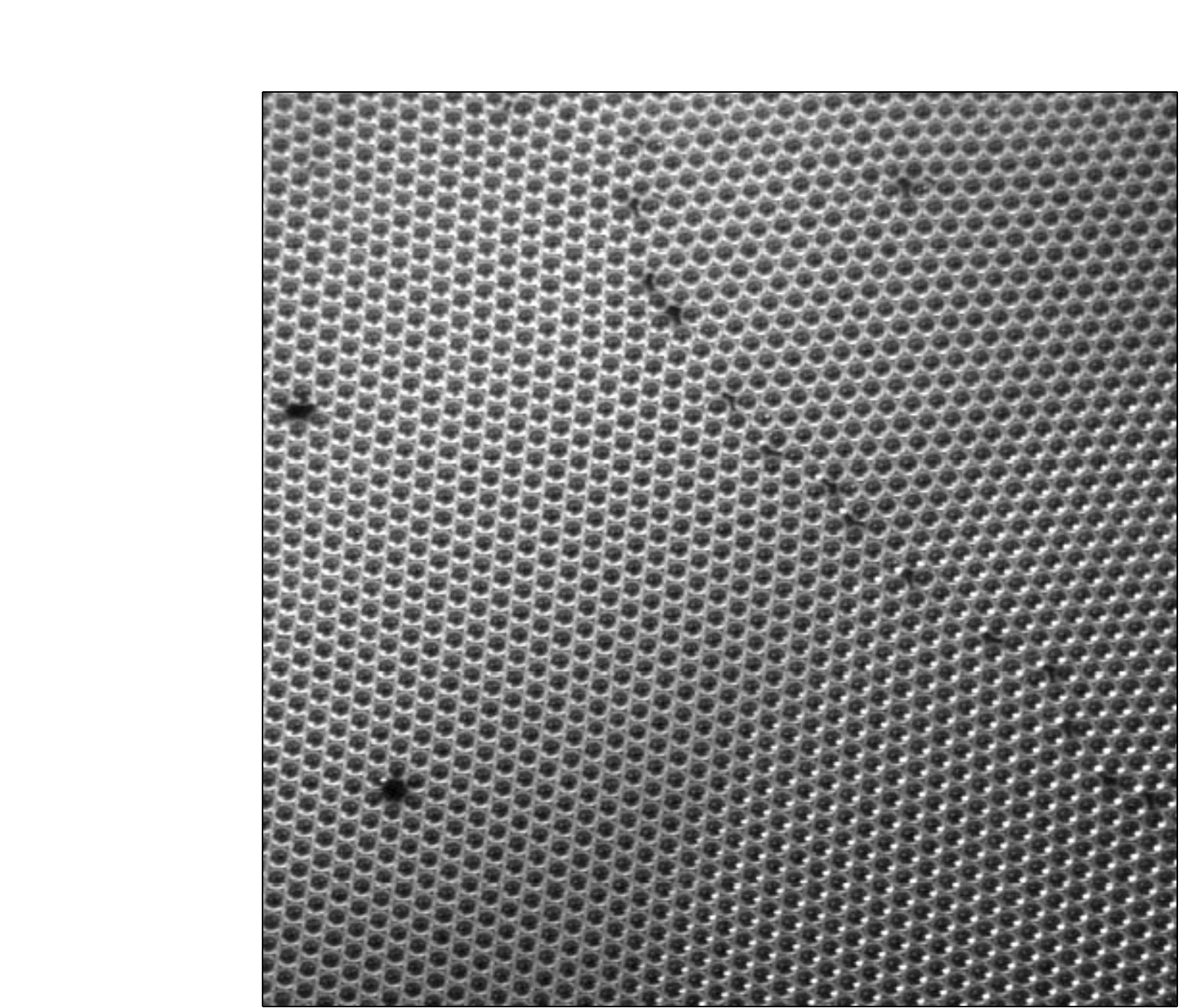}  \includegraphics[height=1.6in]{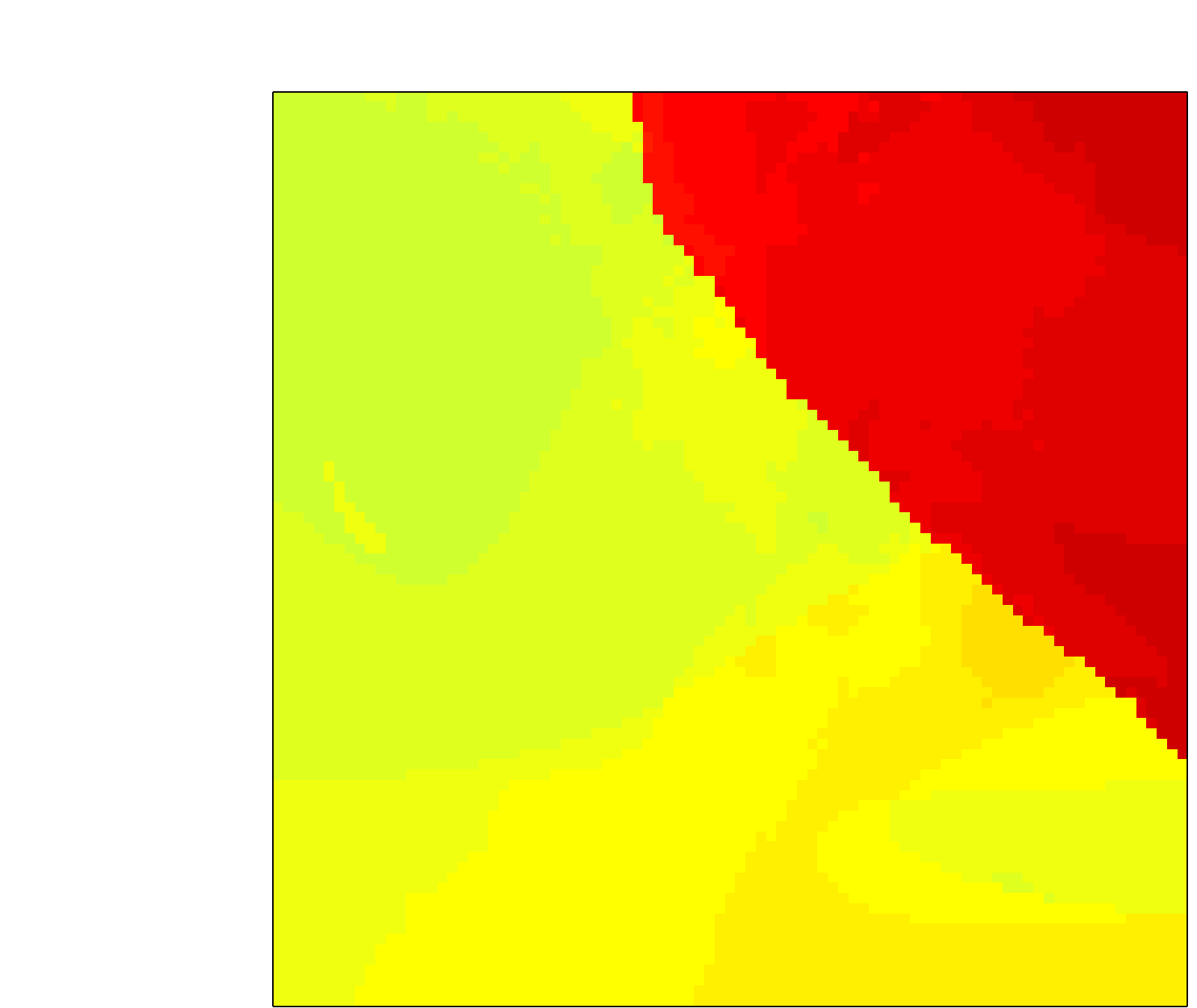} \includegraphics[height=1.6in]{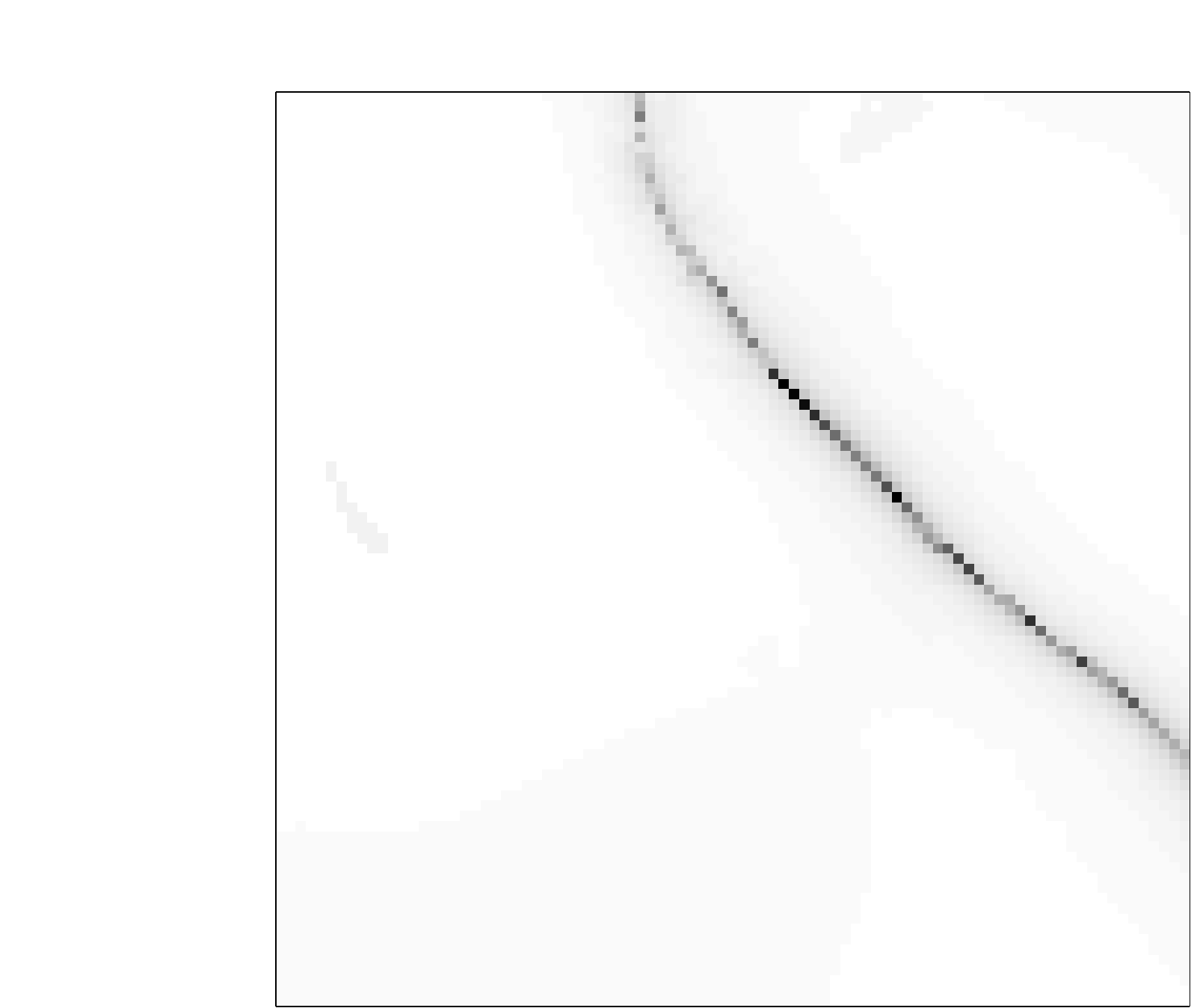}
  \end{center}
  \caption{Left: A photograph of a bubble raft with strong reflections and point dislocations (courtesy to Don Stone, copyright 2009 Board of Regents of the University of Wisconsin System). Its size is $476\times 476$ pixels. Middle and right: The weighted average angle $\Angle(b)$ and the boundary indicator function provided by Algorithm $\ref{alg:undeformed}$. Primary implementation parameters: $s=t=0.5$, $d=0.5$.}
  \label{fig:GB3_fast}
\end{figure}

Algorithm $\ref{alg:undeformed}$ also works robustly in noisy examples. A noisy example of a real atomic resolution image of a Sigma $99$ tilt grain boundary in aluminum is shown in Figure \ref{fig:GB4_fast} (a). The domain consists of  three grains with fuzzy transition regions. This is a rather challenging example, as individual atoms are hardly identifiable in the middle-left part of the image. Grain boundaries by manual inspection are indicated in red in Figure \ref{fig:GB4_fast} (g). The small region indicated by blue color contains a grain boundary that is hardly seen. 
The weighted average angle $\Angle(b)$ provided by Algorithm
$\ref{alg:undeformed}$ reveals three grains with slight deformation of
lattices in Figure \ref{fig:GB4_fast} (b). Grain boundaries can be also clearly seen from the boundary
indicator function $\BD(b)$ as shown in Figure \ref{fig:GB4_fast} (c). They can be identified by a simple thresholding using $\BD(b)$. Figure \ref{fig:GB4_fast} (h) embeds the identified boundaries into the original crystal image. The result by manual inspection in Figure \ref{fig:GB4_fast} (g) agrees with the one in Figure \ref{fig:GB4_fast} (h). 

To demonstrate the robustness of our method, we normalize the image in Figure \ref{fig:GB4_fast} (a) such that the maximum intensity is $1$ and add Gaussian white noise with a standard deviation equal to $1$. The polluted example is shown in Figure \ref{fig:GB4_fast} (d). The atom structure is barely seen even away from the grain boundaries. The weighted average angle $\Angle(b)$, the boundary indicator function $\BD(b)$, and the estimated grain boundaries are shown in Figure \ref{fig:GB4_fast} (e), (f), and (i). As shown by these results, Algorithm $\ref{alg:undeformed}$ is able to give similar results matching with the one by manual inspection even if noise is heavy.

\begin{figure}[ht!]
  \begin{center}
   \begin{tabular}{ccc}
    \includegraphics[height=1.5in]{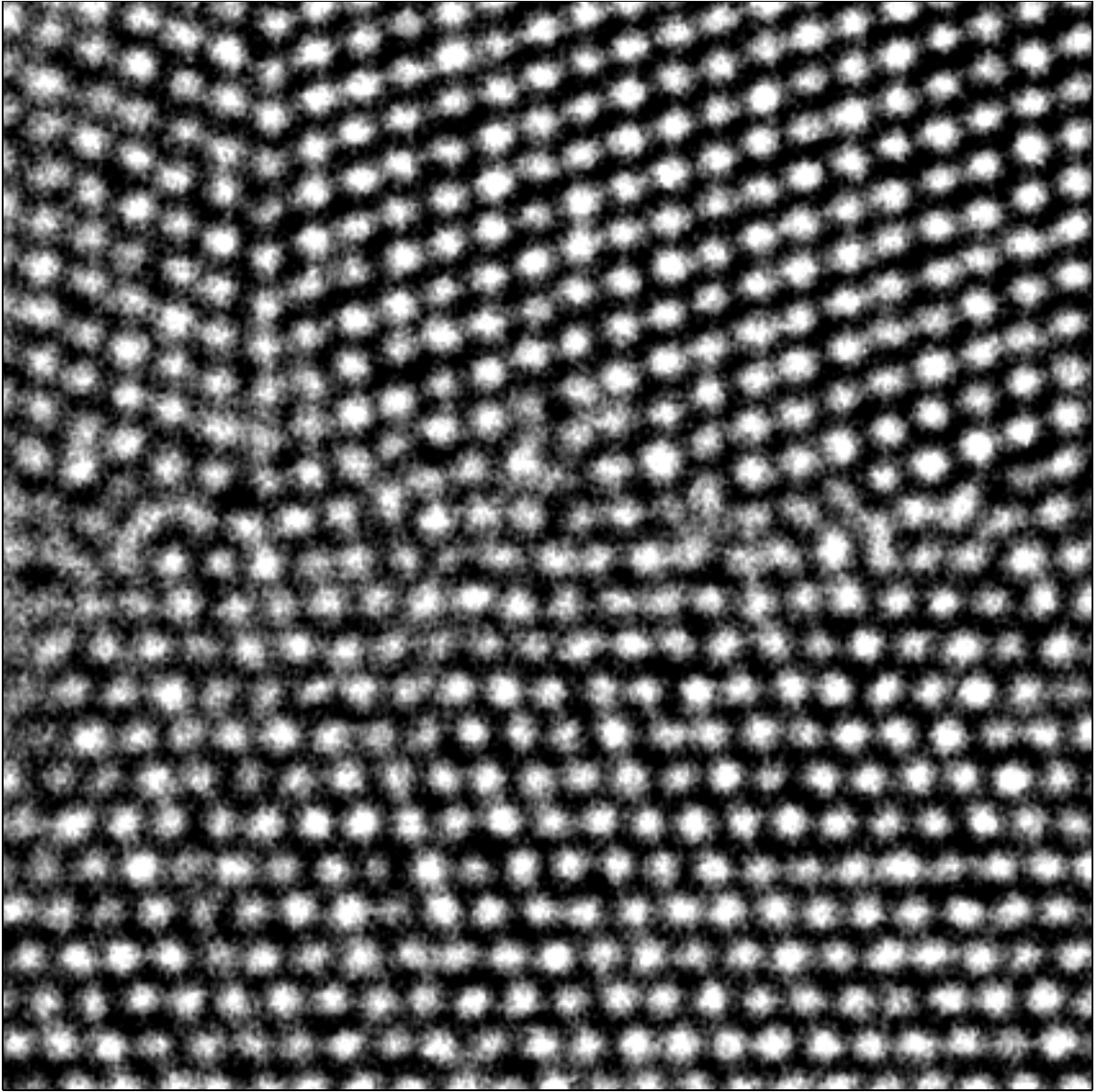}  & \includegraphics[height=1.5in]{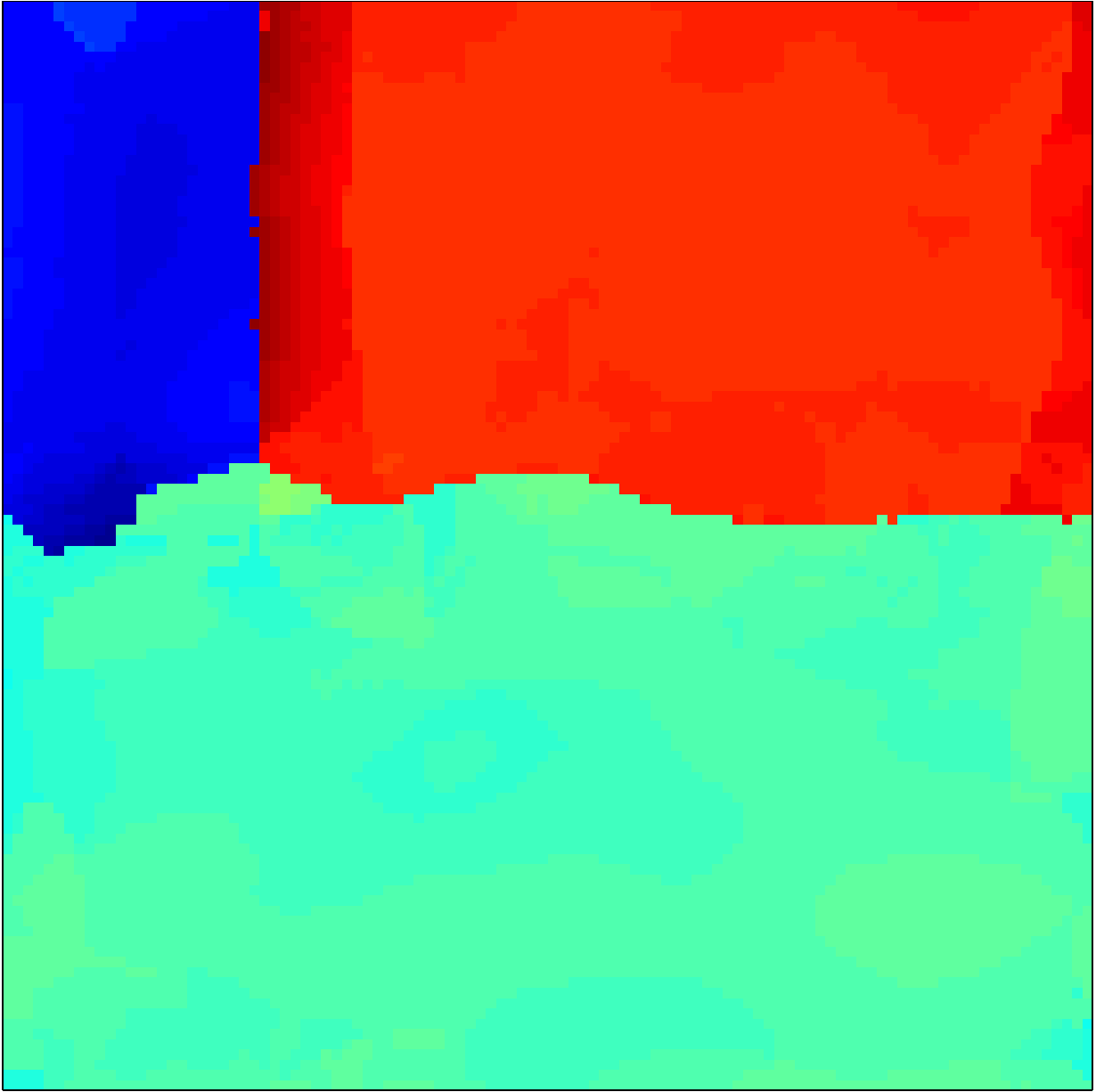} & \includegraphics[height=1.5in]{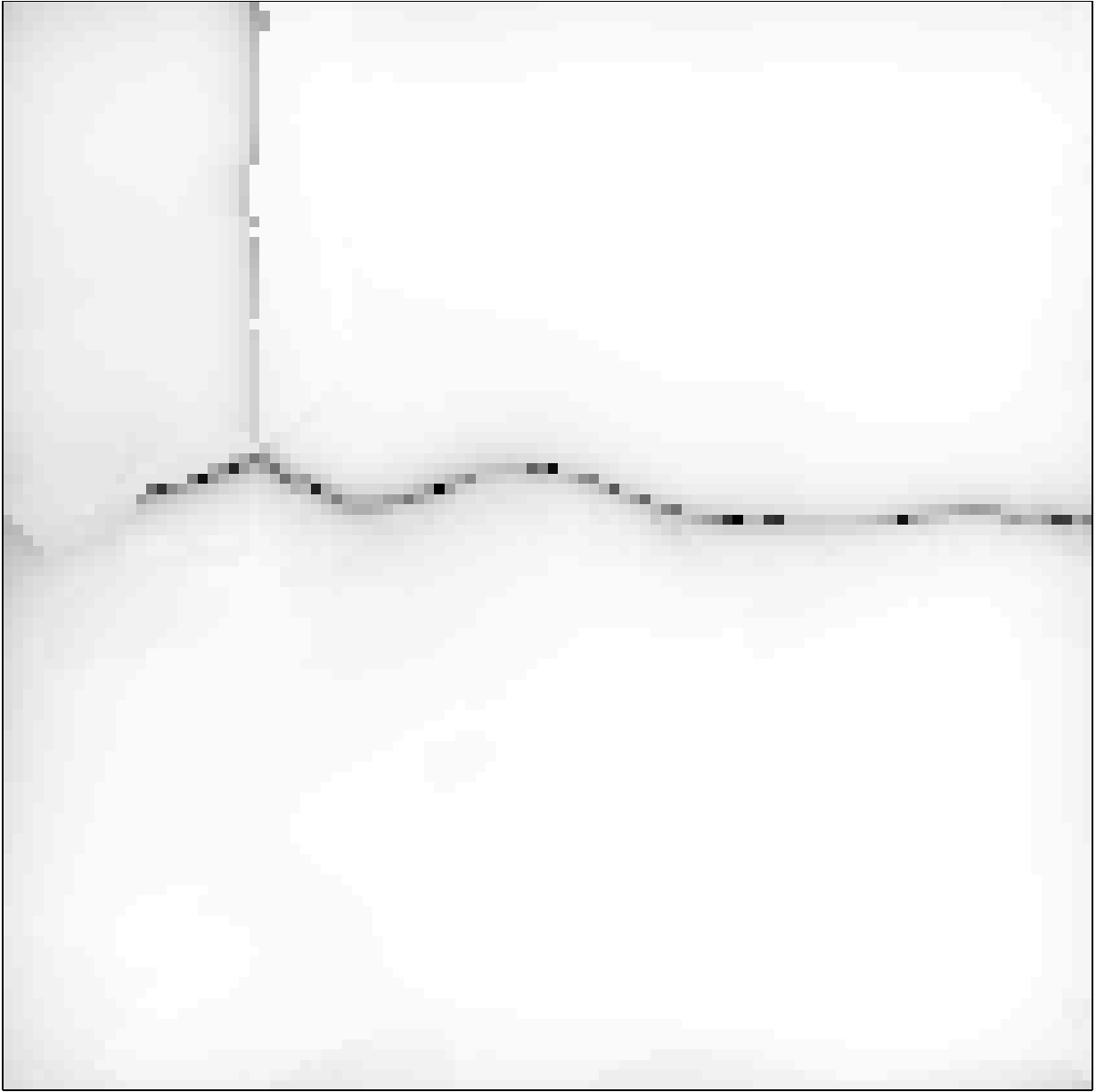} \\
    (a) & (b) & (c)\\
    \includegraphics[height=1.5in]{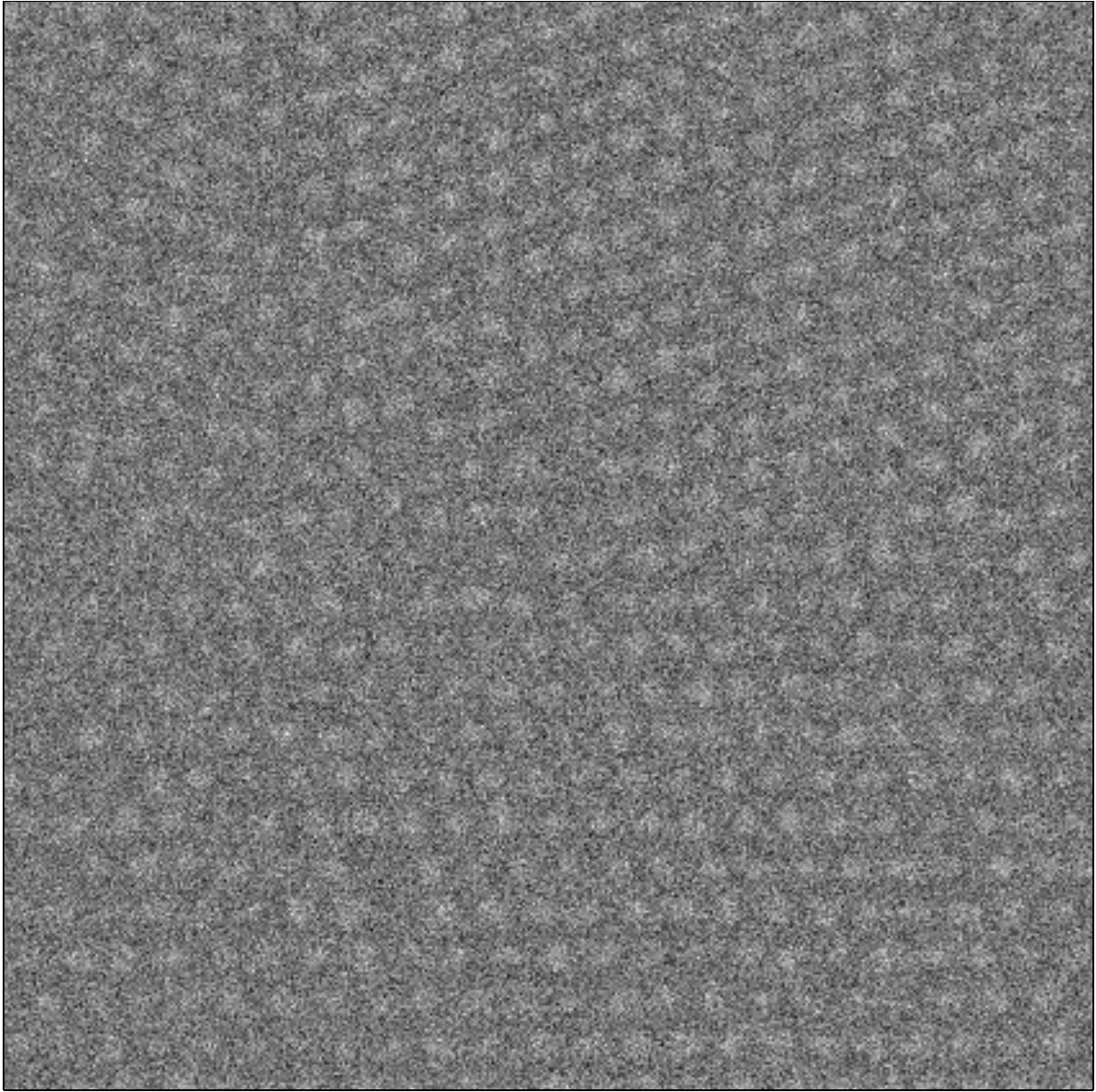}    & \includegraphics[height=1.5in]{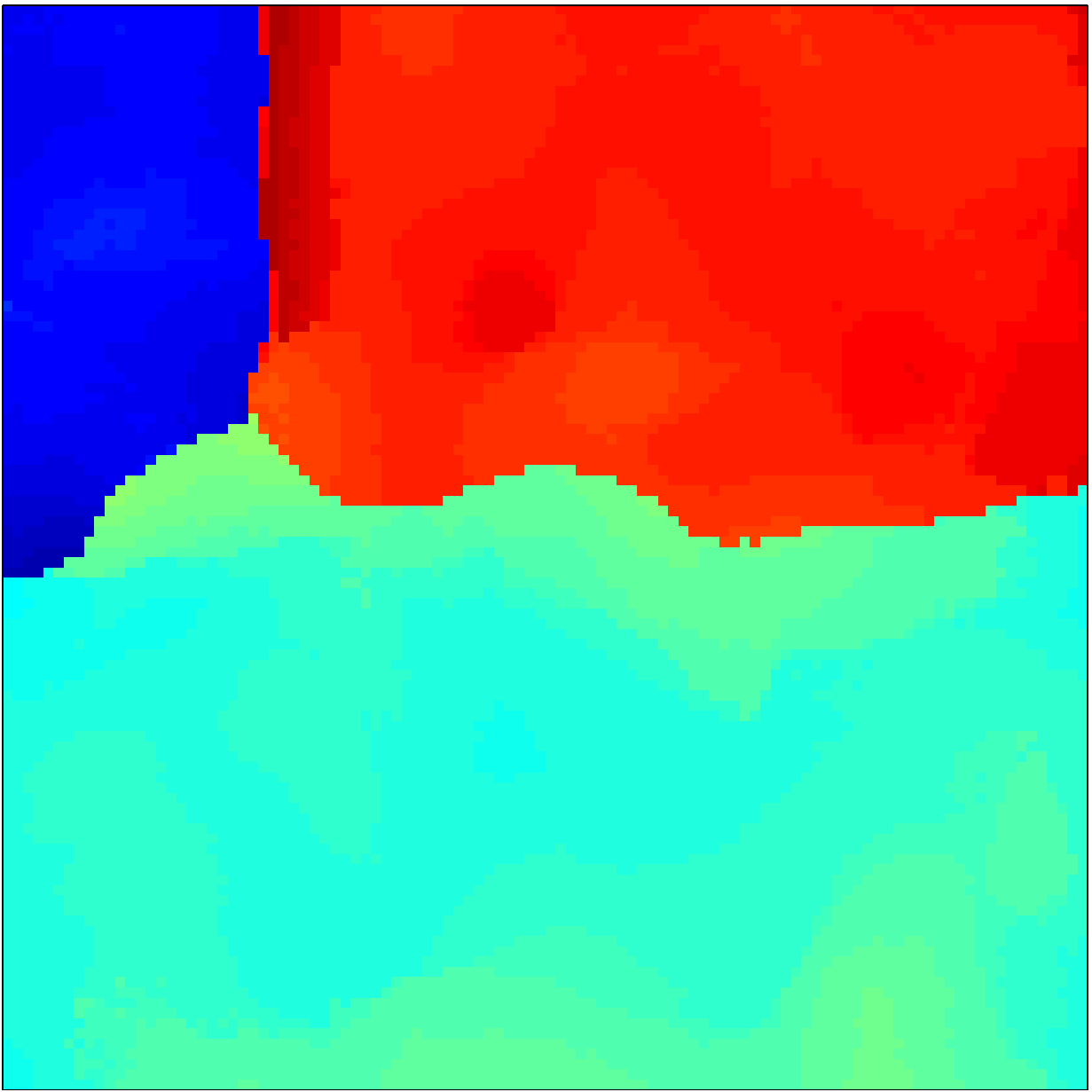}  &\includegraphics[height=1.5in]{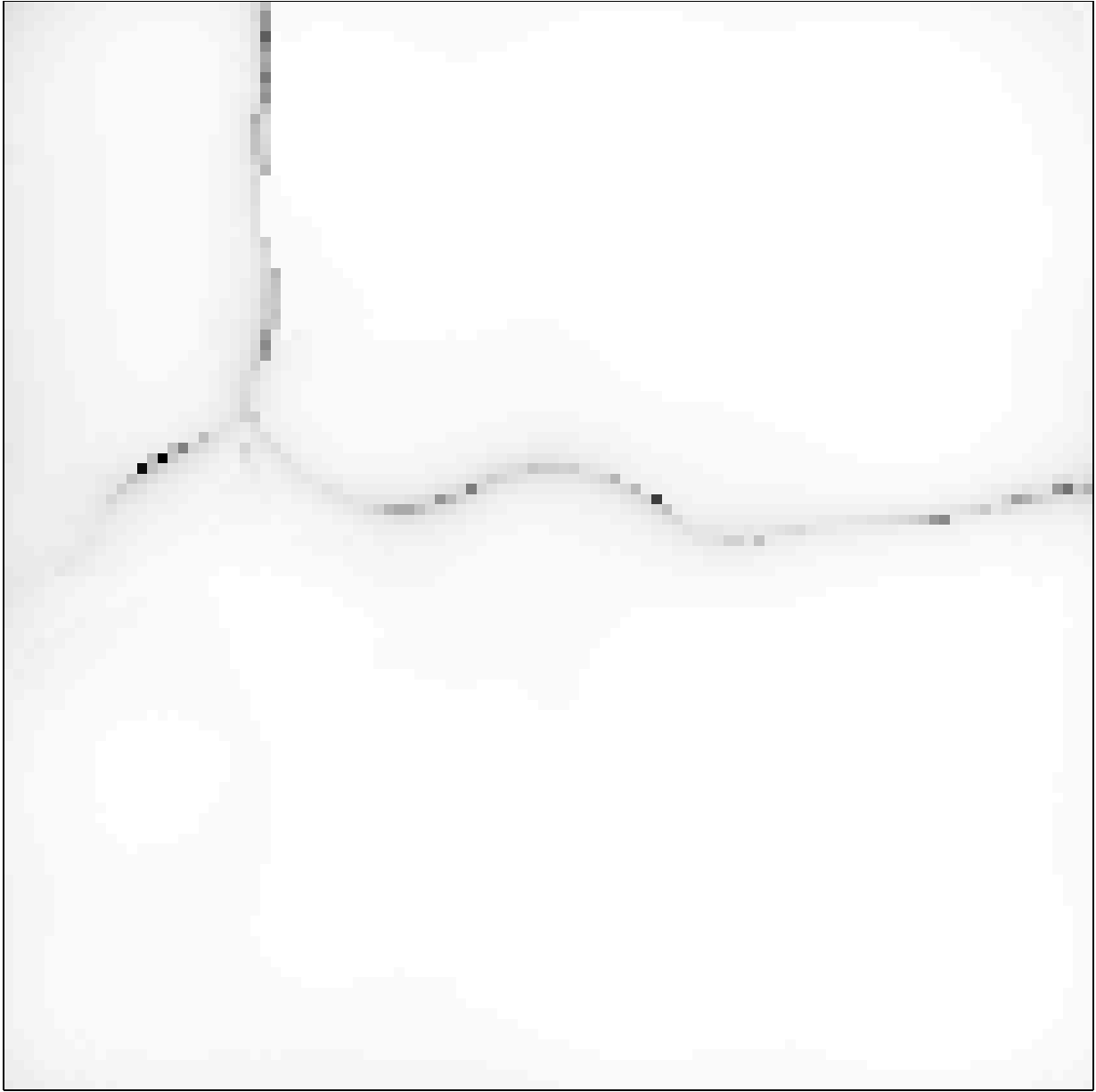}\\
    (d) & (e) & (f) \\
    \includegraphics[height=1.5in]{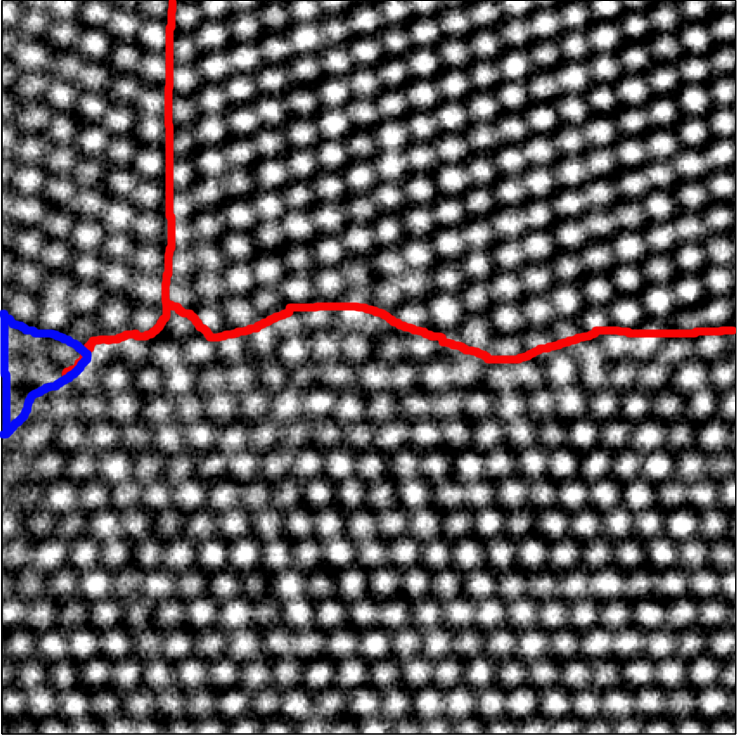}& \includegraphics[height=1.5in]{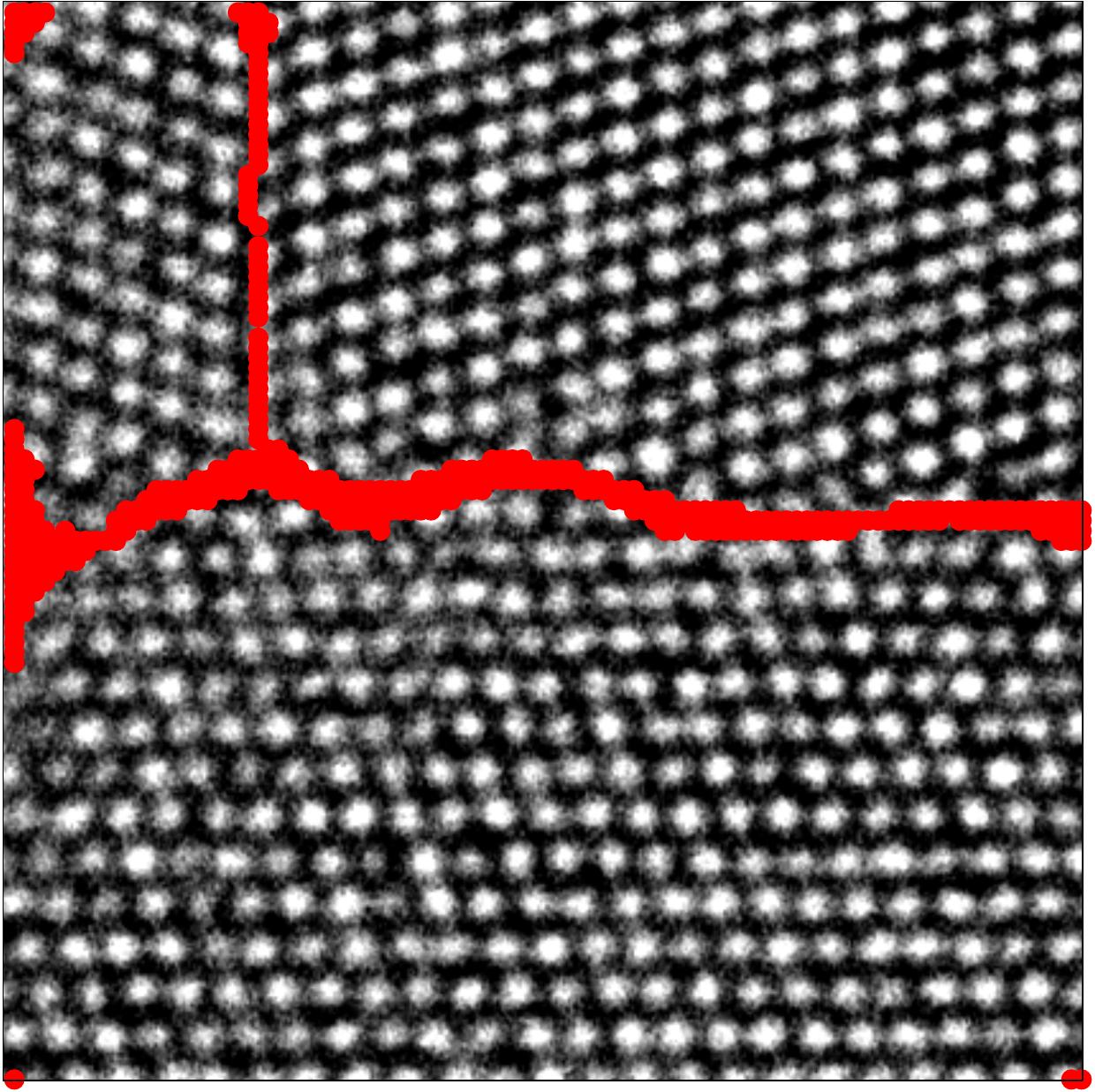}  &\includegraphics[height=1.5in]{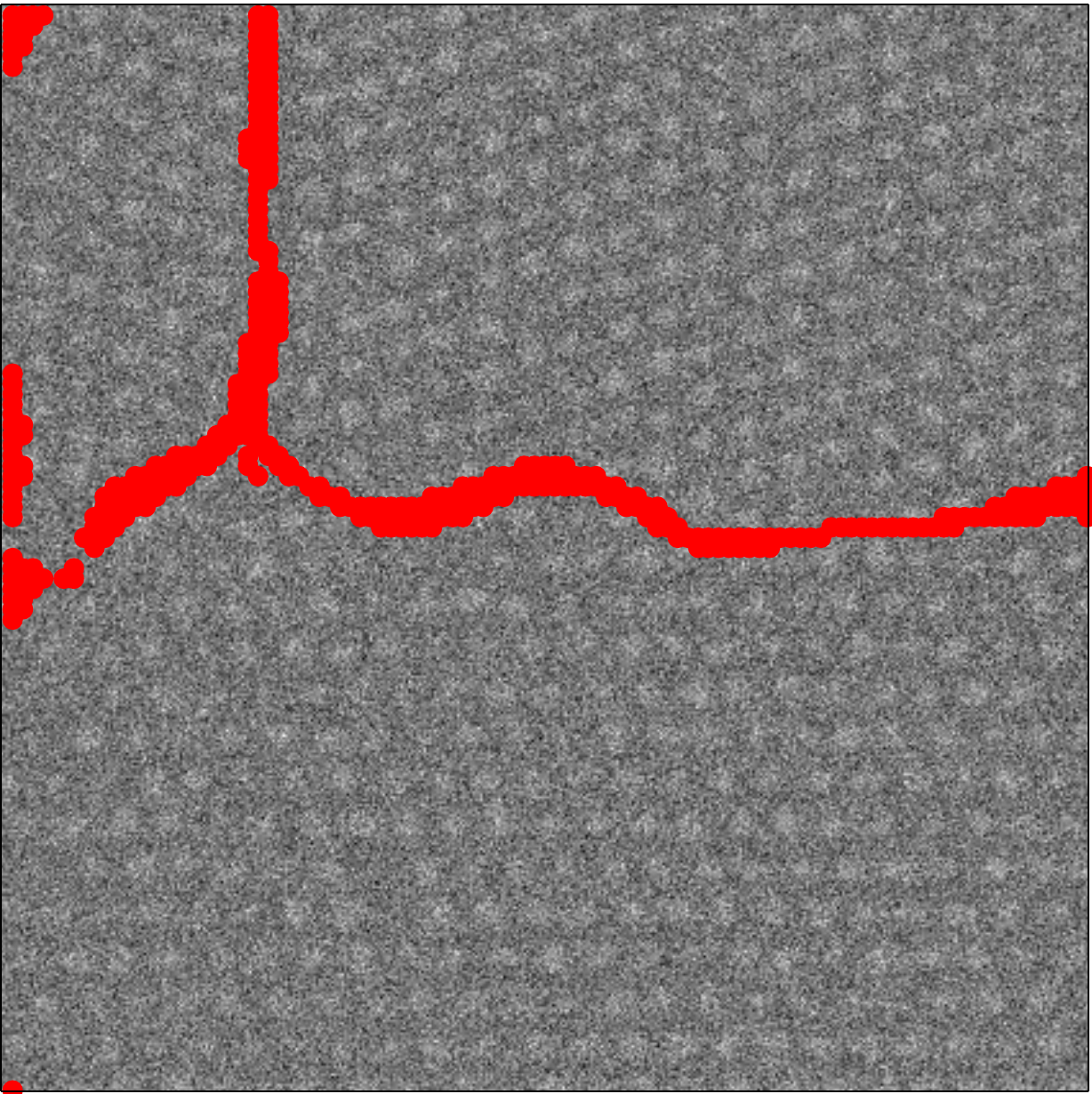} \\
    (g) & (h) & (i)
    \end{tabular}
  \end{center}
  \caption{(a) and (d): An atomic resolution image of $424\times 424$ pixels of Sigma $99$ tilt grain boundary in Al and its noisy version. Courtesy of National Center for Electron Microscopy in Lawrence Berkeley National Laboratory. (b) and (e): The weighted average angle $\Angle(b)$ of the images in (a) and (d), respectively. (c) and (f): The boundary indicator functions provided by Algorithm $\ref{alg:undeformed}$ for the images in (a) and (d), respectively. (g): The original crystal image associated with identified defect regions (in red) by manual inspection. The region indicated in blue contains a grain boundary that is hardly seen.  (h) and (i): Identified defect regions (in red) by thresholding the boundary indicator function of (a) and (d), respectively. Primary implementation parameters: $t=0.75$, $s=0.65$, $d=1$.}
  \label{fig:GB4_fast}
\end{figure}

\subsection{Examples for Algorithm $\ref{alg:deformed}$}

As pointed out previously, Algorithm $\ref{alg:deformed}$ is more sensitive to local defects than Algorithm $\ref{alg:undeformed}$. Therefore, it can better discover local defects hidden in the underlying wave-like components of crystal images. To validate this, let us revisit the phase field crystal image example shown in Figure~\ref{fig:GB1_fast}. 
The analysis results are presented in Figure \ref{fig:GB1_elastic} and \ref{fig:GB1vol}. First, as shown in Figure \ref{fig:GB1_elastic} (middle), the weighted average angle $\Angle(b)$ provided by Algorithm $\ref{alg:deformed}$ is consistent with the one by Algorithm $\ref{alg:undeformed}$. Second, the boundary indicator function $\BD(x)$ by Algorithm $\ref{alg:deformed}$ reveals all local defects including grain boundaries and isolated dislocations clearly. As a trade off, Algorithm $\ref{alg:deformed}$ visualizes grain boundaries as narrow fuzzy bands instead of sharp curves. Finally, as shown in Figure \ref{fig:GB1vol}, isolated point dislocations appear as dipoles in the visualization of the local volume distortion $\Vol(b)$. As a consequence, the Burgers vectors of corresponding isolated point dislocations can be inferred by the directions of these dipoles.

To demonstrate the robustness of Algorithm $\ref{alg:deformed}$, we normalize the intensity of the crystal image in Figure \ref{fig:GB1_elastic} and add Gaussian white noise with distributions $0.5\mathcal{N}(0,1)$ and $1.4\mathcal{N}(0,1)$. These noisy examples are shown in Figure \ref{fig:ns} (a) and (d), respectively. 
In these heavily polluted cases, even when no crystal structure is clearly visible by human eyes, Algorithm $\ref{alg:deformed}$ is still able to reveal grain boundaries and isolated defects with a reasonable accuracy (see Figure \ref{fig:ns} (b) and (e)). The distortion volume in Figure \ref{fig:ns} (c) and (f) still roughly reflects the strain stress encoded by color.

\begin{figure}[ht!]
  \begin{center}
\hspace{-1em}
      \includegraphics[height=1.55in]{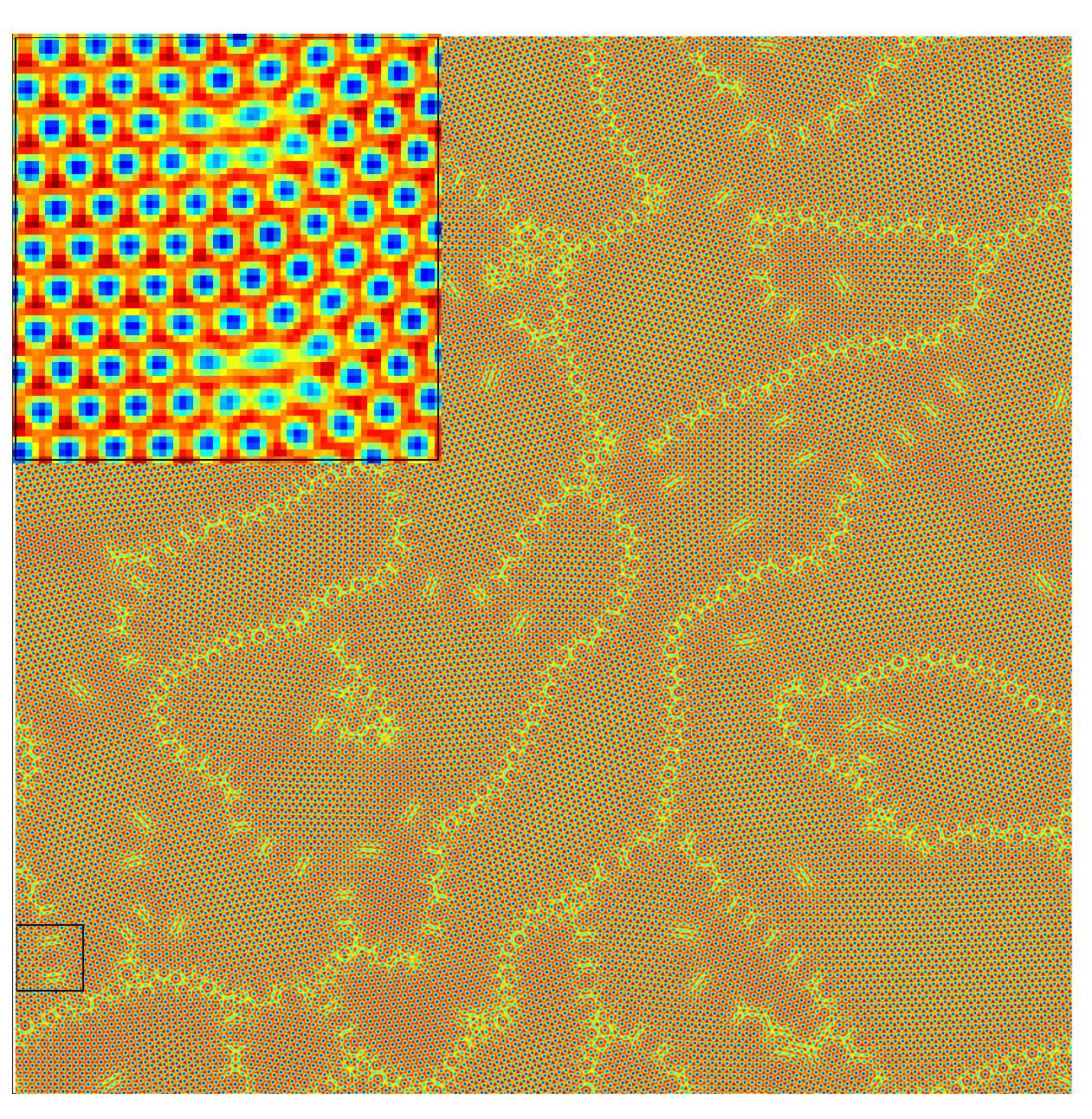}  \includegraphics[height=1.6in]{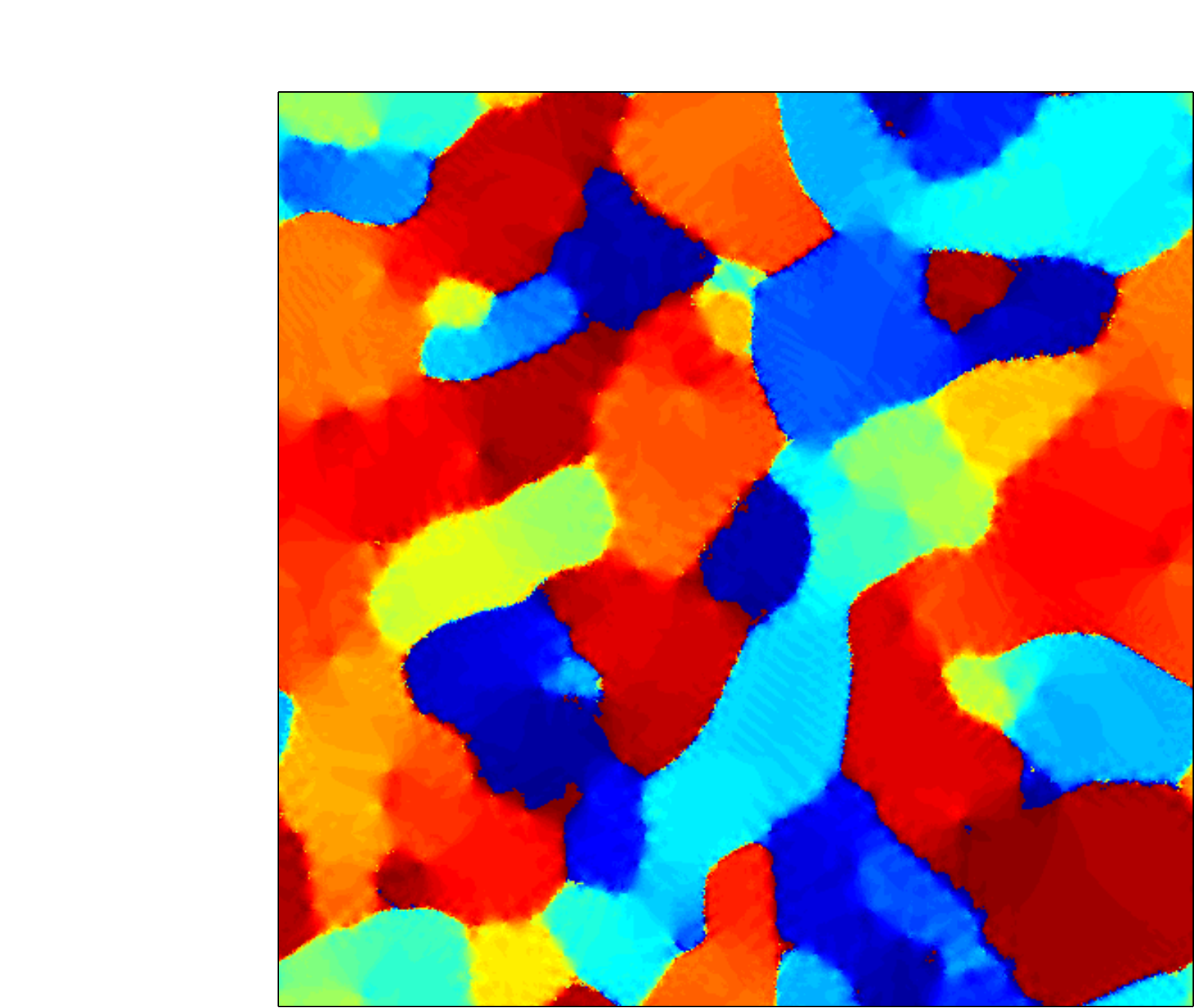} \hspace{1cm}  \includegraphics[height=1.55in]{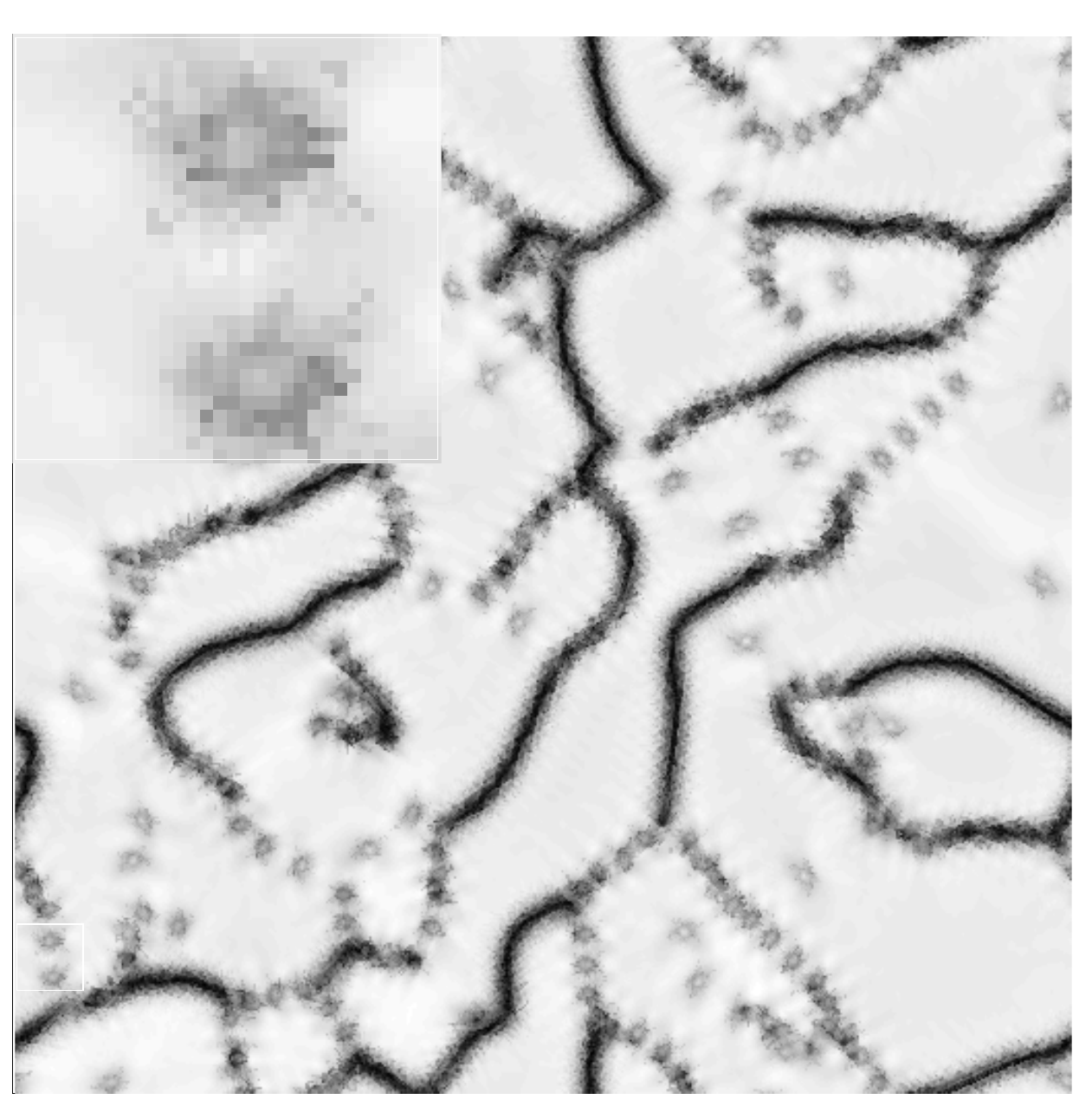}\\
\hspace{-4em}
      \includegraphics[height=1.6in]{GB1_image_zoomed.pdf}  \includegraphics[height=1.6in]{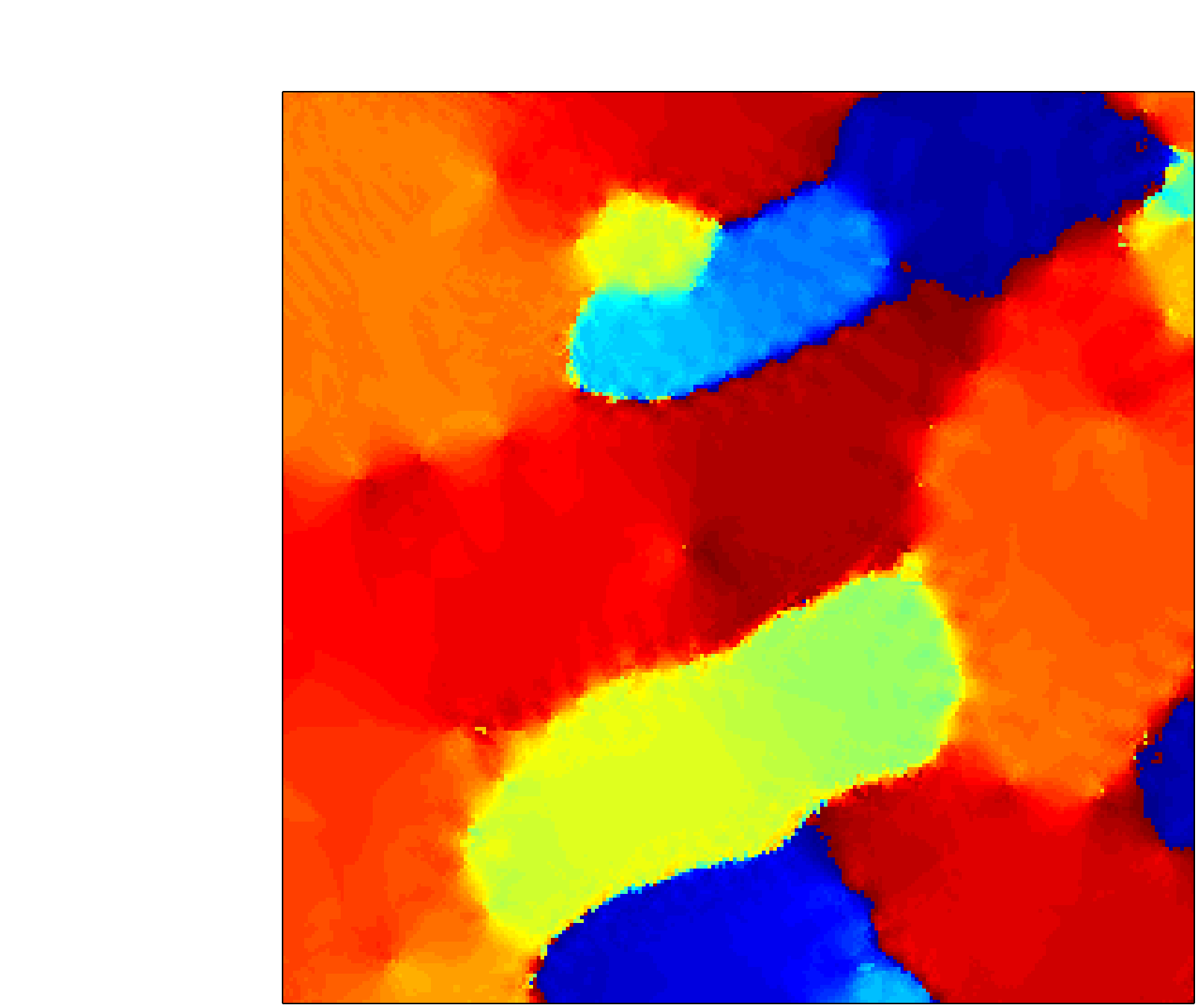}  \includegraphics[height=1.6in]{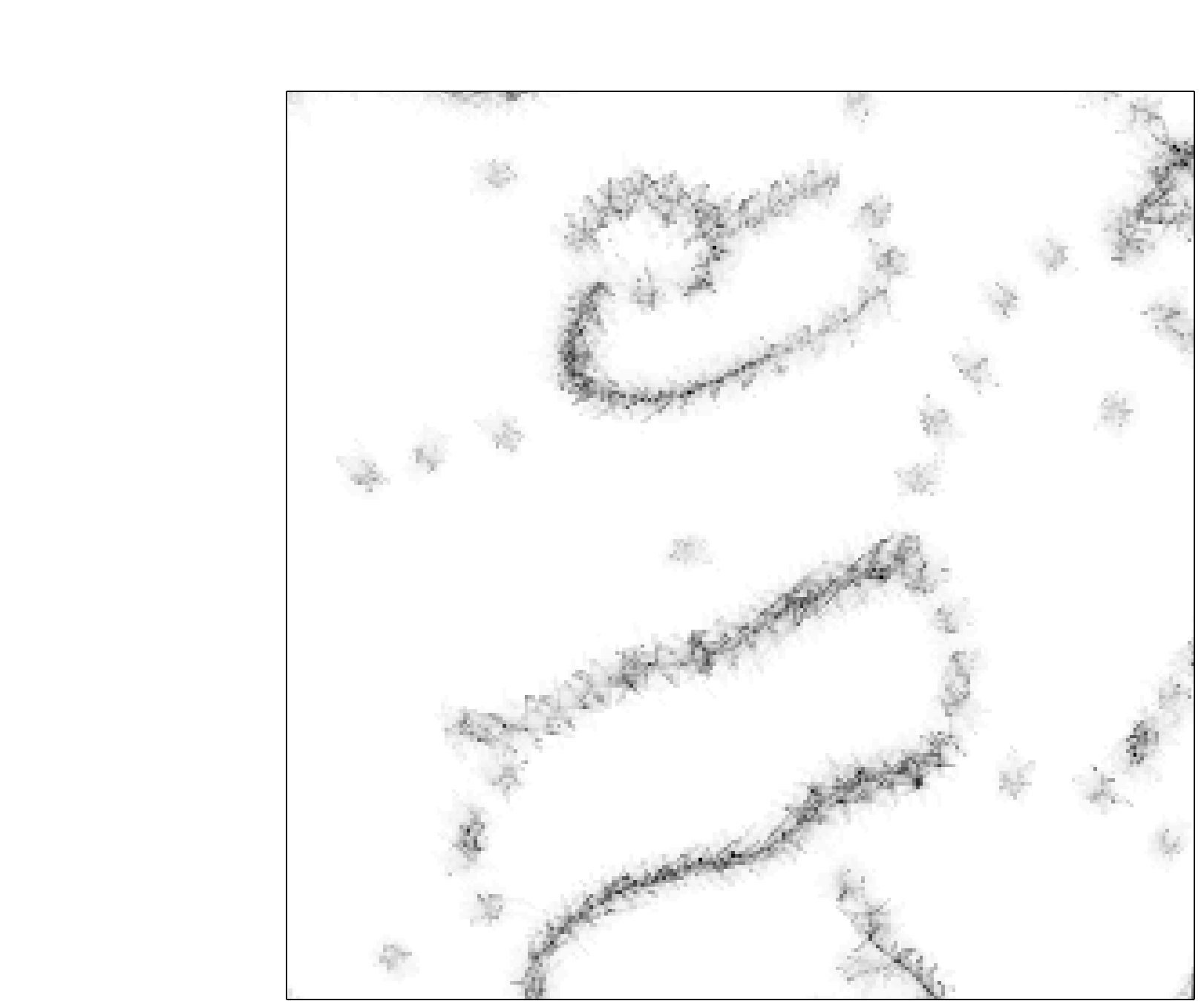}
  \end{center}
  \caption{Analysis results of a phase field crystal image  provided by Algorithm $\ref{alg:deformed}$. Compare with Figure~\ref{fig:GB1_fast}. Left: A phase field crystal (PFC) image and its zoomed-in images. Middle: The weighted average angle $\Angle(b)$ and its zoomed-in results. Right: The weighted boundary indicator function $\BD(b)$ and its zoomed-in results. The small rectangles in the bottom-left corner of these images indicate the zoomed-in images in top-left corner. Primary implementation parameters: $s=t=0.8$, $d=1$.}
  \label{fig:GB1_elastic}
\end{figure}

\begin{figure}[ht!]
  \begin{center}
    \begin{tabular}{cc}
      \includegraphics[height=1.8in]{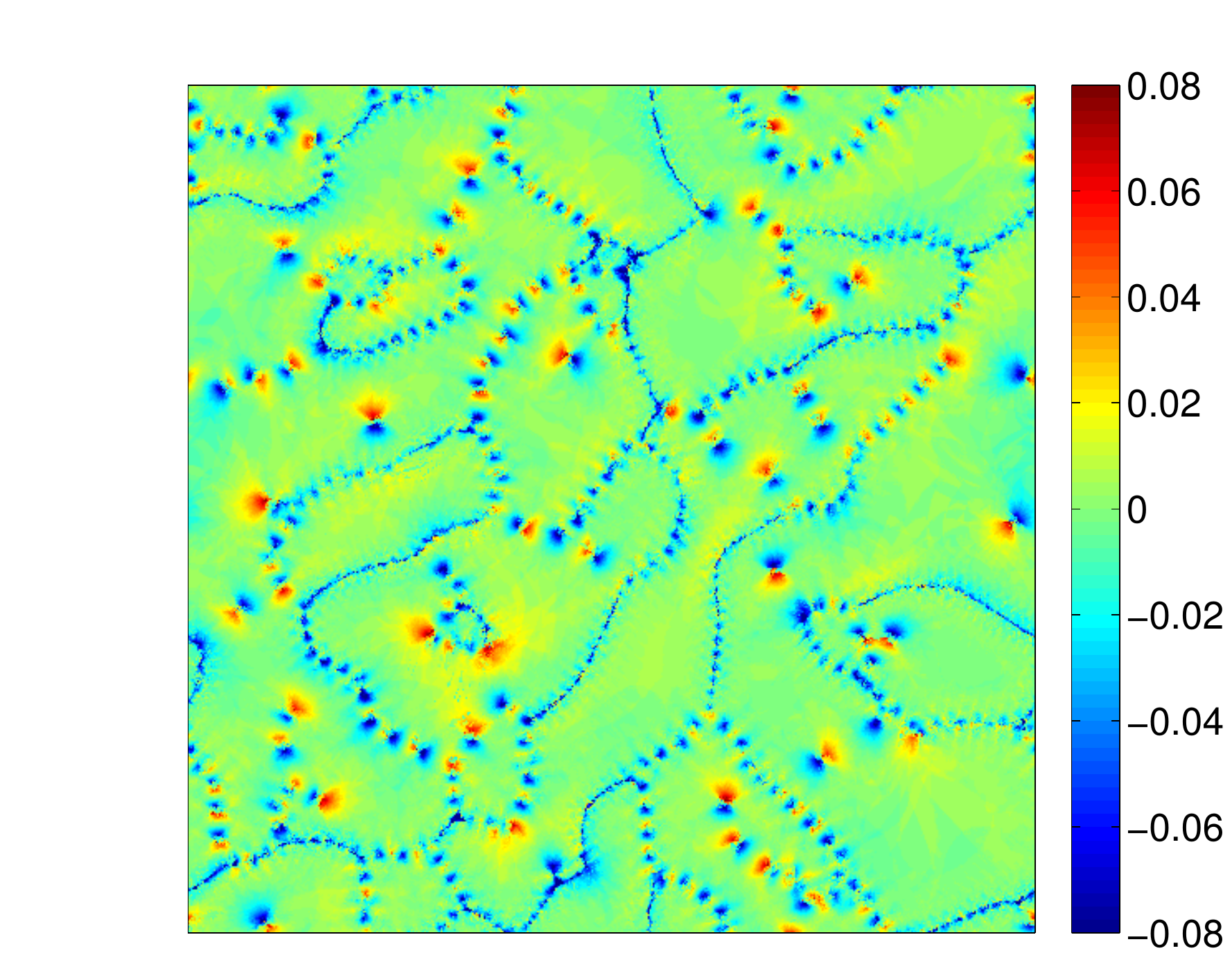} &  \includegraphics[height=1.8in]{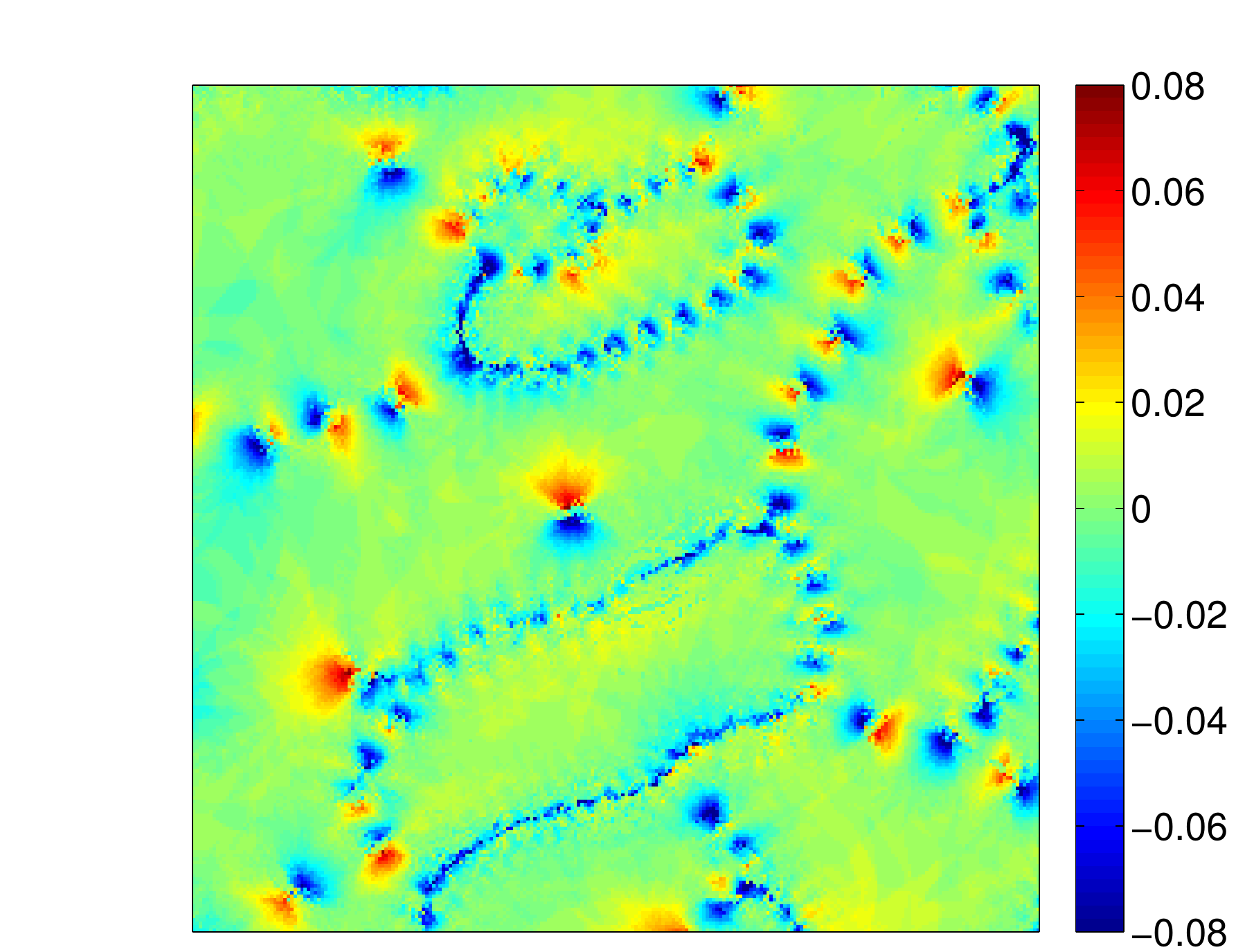}
    \end{tabular}
  \end{center}
  \caption{Left: The local volume distortion estimate $\Vol(b)$ of the phase field crystal image example in Figure~\ref{fig:GB1_elastic} and its zoomed-in result. The color limits of these two images are set to be $0.08$ and $-0.08$.} 
  \label{fig:GB1vol}
\end{figure}

\begin{figure}[ht!]
  \begin{center}
    \begin{tabular}{ccc}
   \includegraphics[height=1.6in]{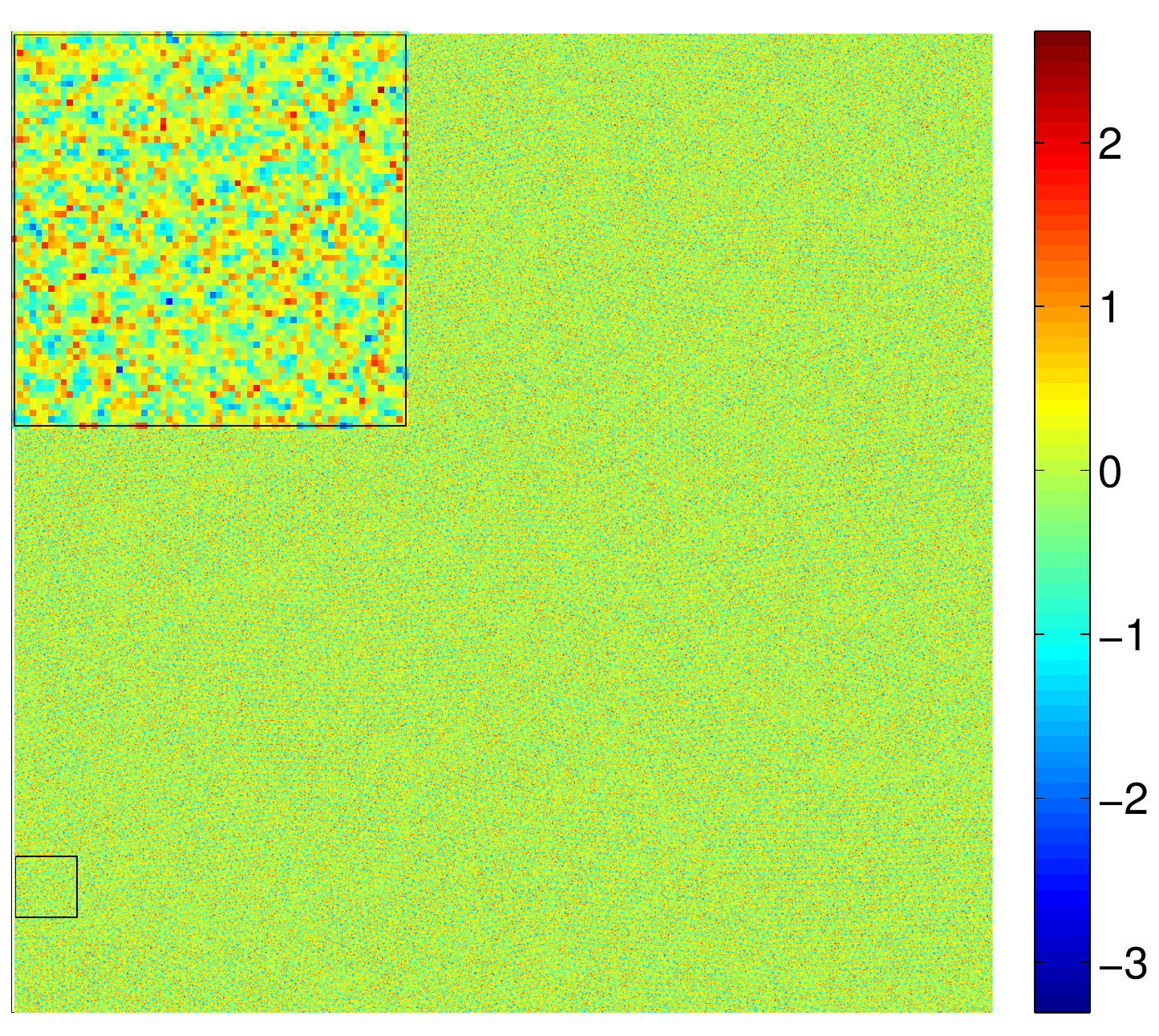}&  \includegraphics[height=1.6in]{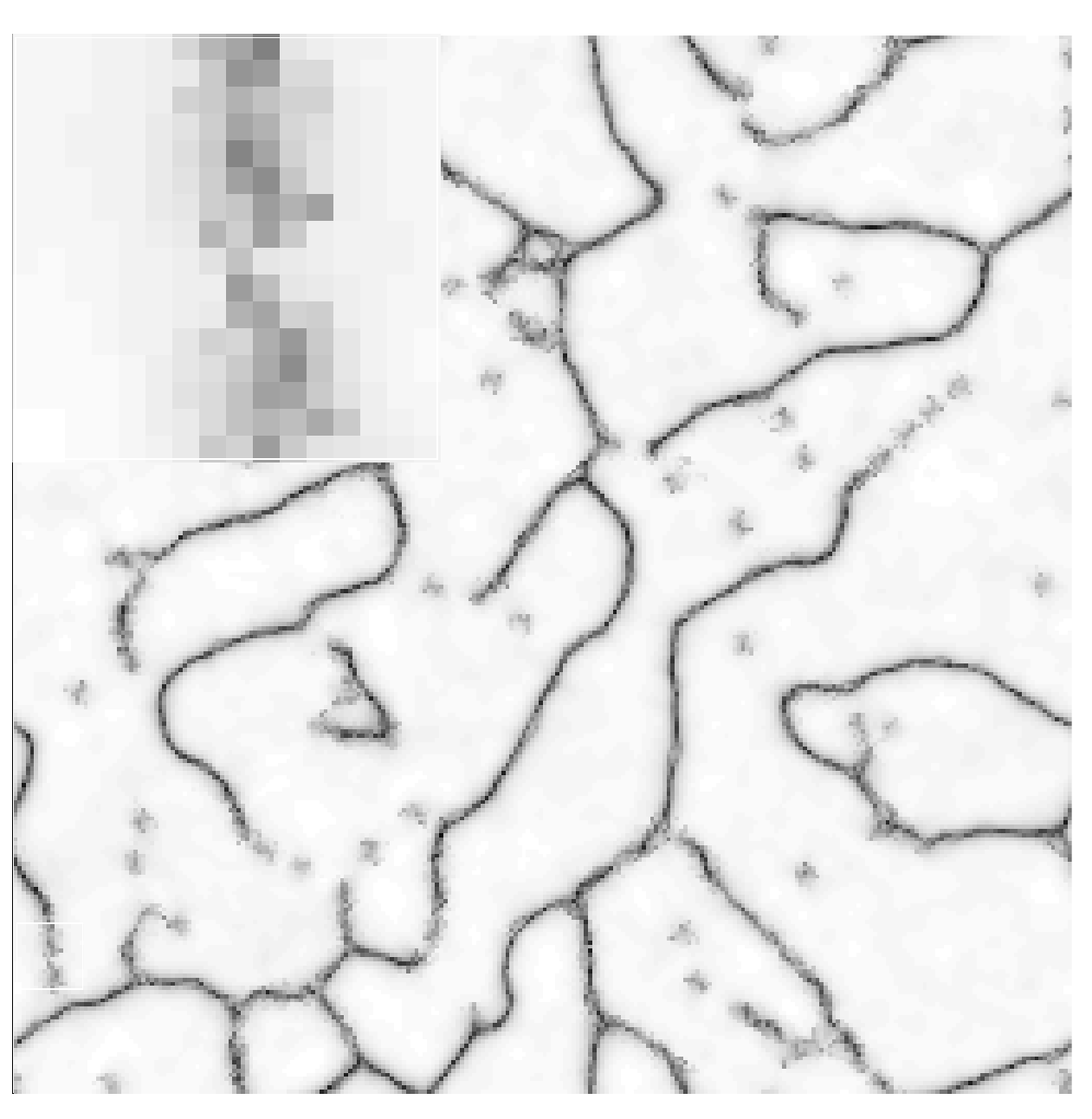}&\includegraphics[height=1.7in]{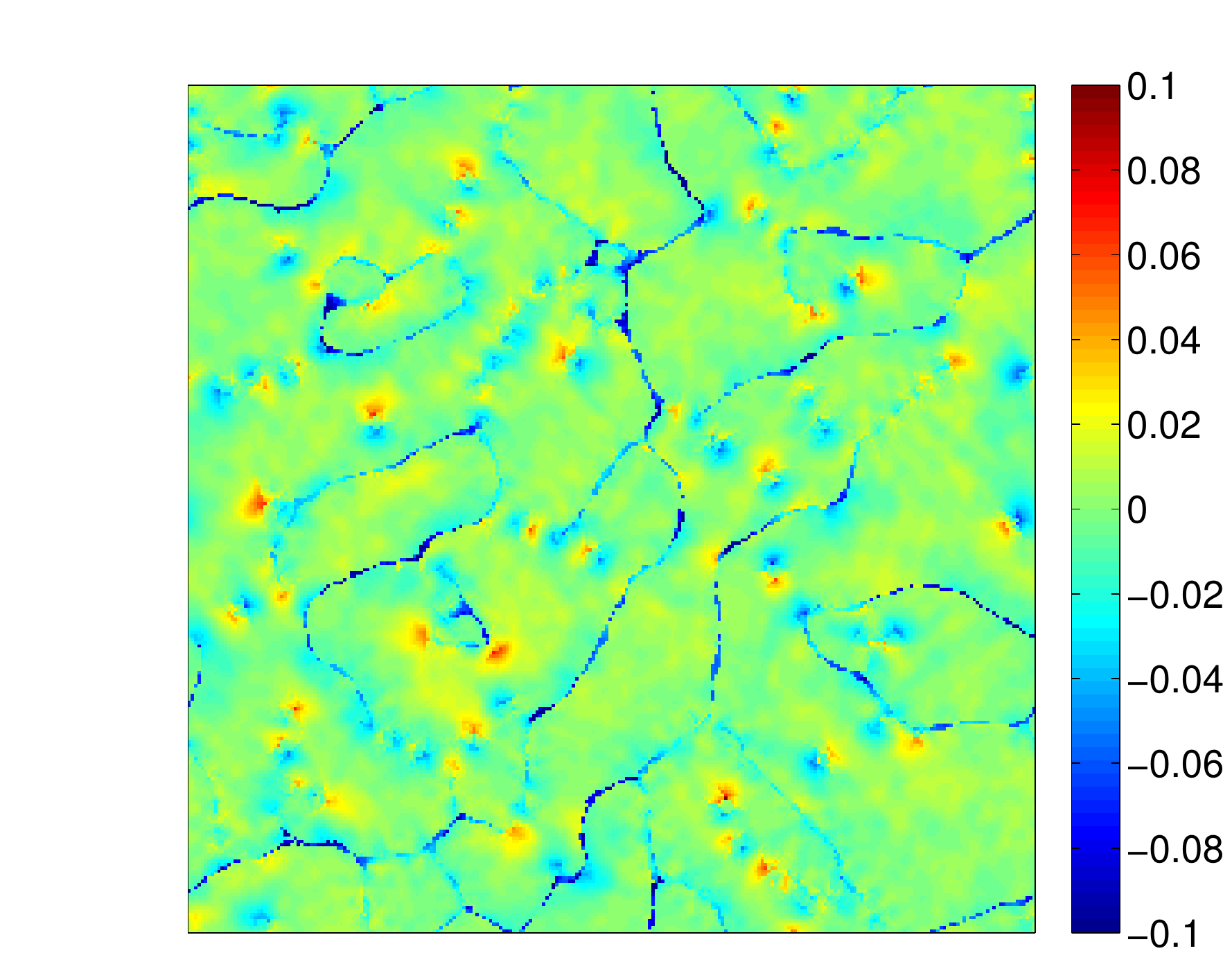}\\
   (a)&(b)&(c)\\
   \includegraphics[height=1.6in]{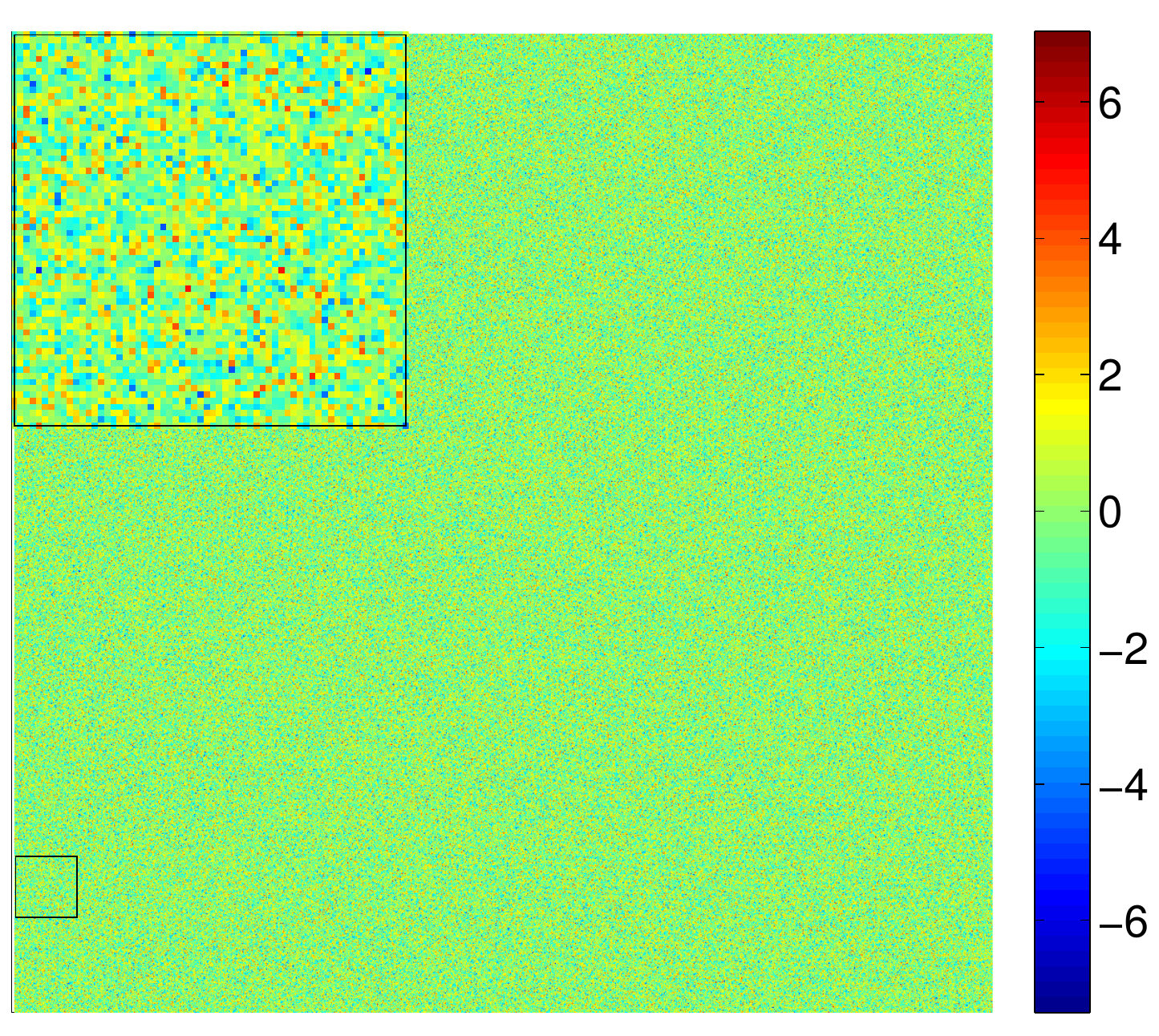}&  \includegraphics[height=1.6in]{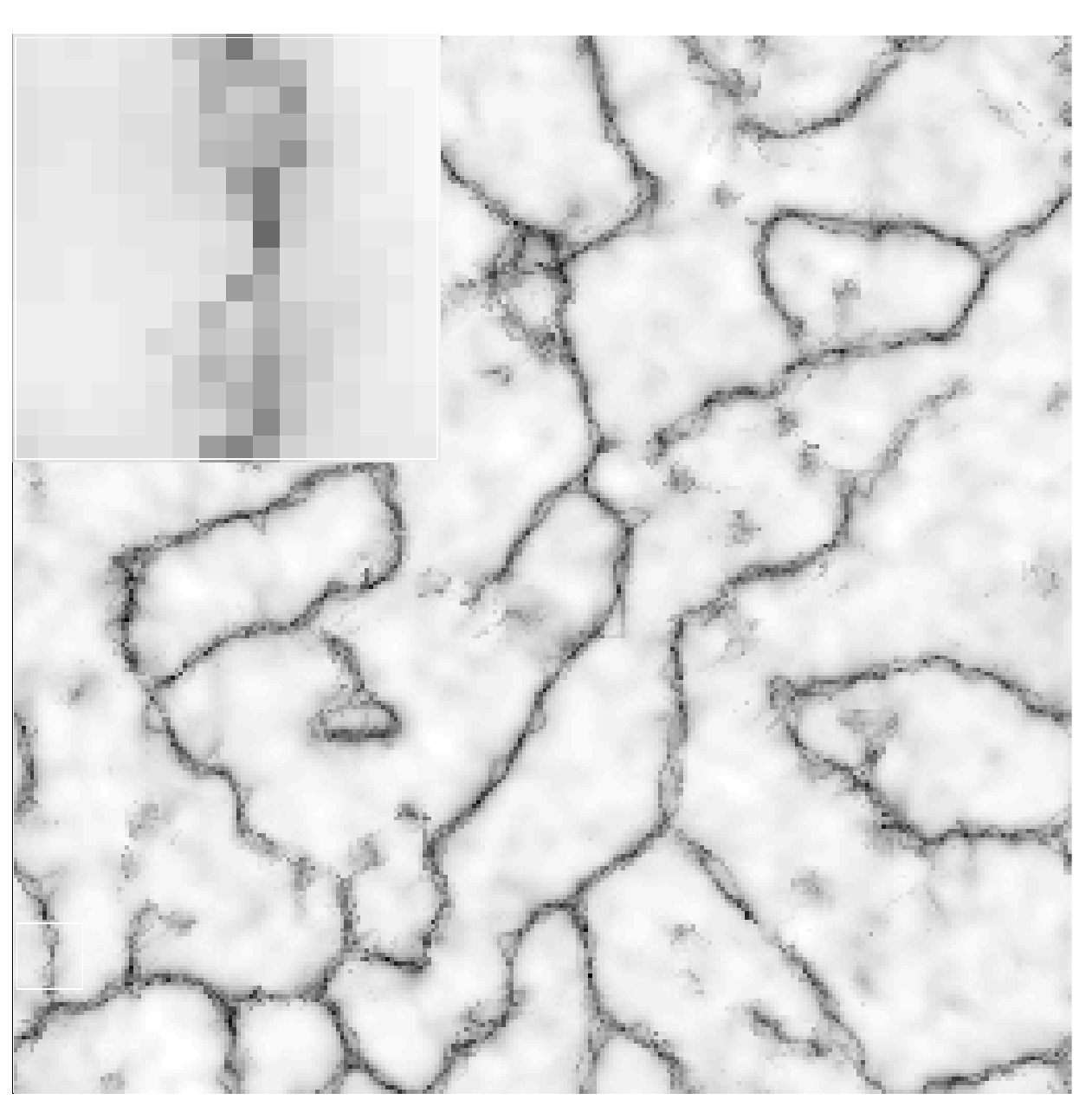}&\includegraphics[height=1.7in]{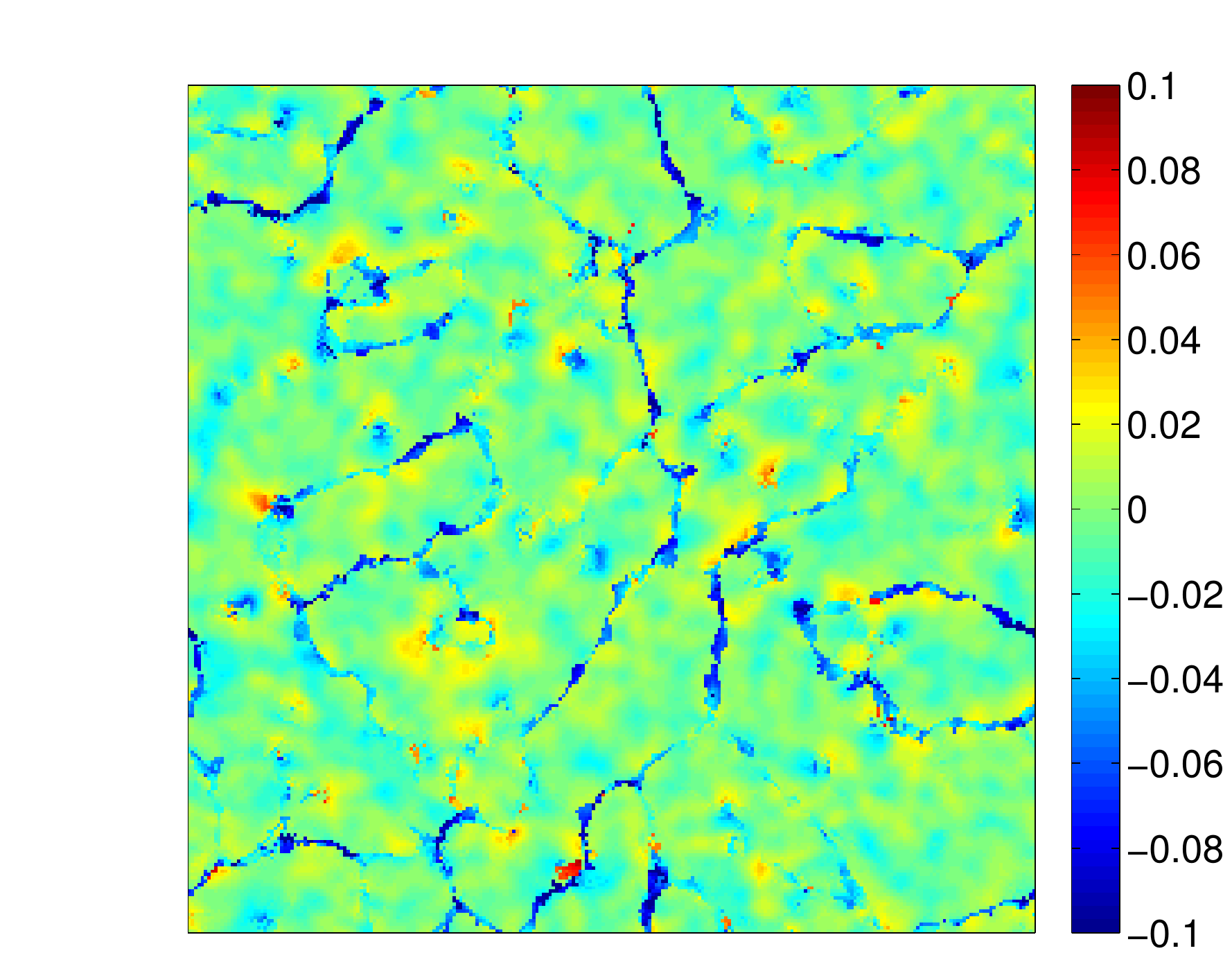}\\
   (d)&(e)&(f)
    \end{tabular}
  \end{center}
  \caption{Analysis results of noisy phase field crystal images provided by Algorithm $\ref{alg:deformed}$. First row: Results of noisy data with Gaussian white noise $0.5\mathcal{N}(0,1)$. Second row: Results of noisy data with Gaussian white noise $1.4\mathcal{N}(0,1)$. First column: input images. Second column: detected grain boundaries and isolated defects. Third column: distortion volume. Zoomed-in images show that our method can still identify isolated defects even if noise is heavy.}
\label{fig:ns}
\end{figure}

We also revisit the example shown in Figure~\ref{fig:GB3_fast} and apply Algorithm $\ref{alg:deformed}$ to identify point defects. Different to the point dislocations in the previous phase field crystal example, some point defects here do not destroy the crystal lattice nearby. For example, the point defect in the bottom-left part of the bubble raft image looks like a missing bubble and it does not influence the positions of neighbor bubbles. Therefore, methods depending on identifying neighbor bubbles might fail to detect such a kind of local defects. 
Since missing bubbles would reduce the synchrosqueezed energy at their places, Algorithm $\ref{alg:deformed}$ is still able to detect them. As an illustration, Figure \ref{fig:GB3_elastic} shows the results provided by Algorithm $\ref{alg:deformed}$ and the positions of those local defects in the left figure are accurately depicted in the right figure.

\begin{figure}[ht!]
  \begin{center}
\hspace{-4em}
\includegraphics[height=1.6in]{GB3_image.pdf}  \includegraphics[height=1.6in]{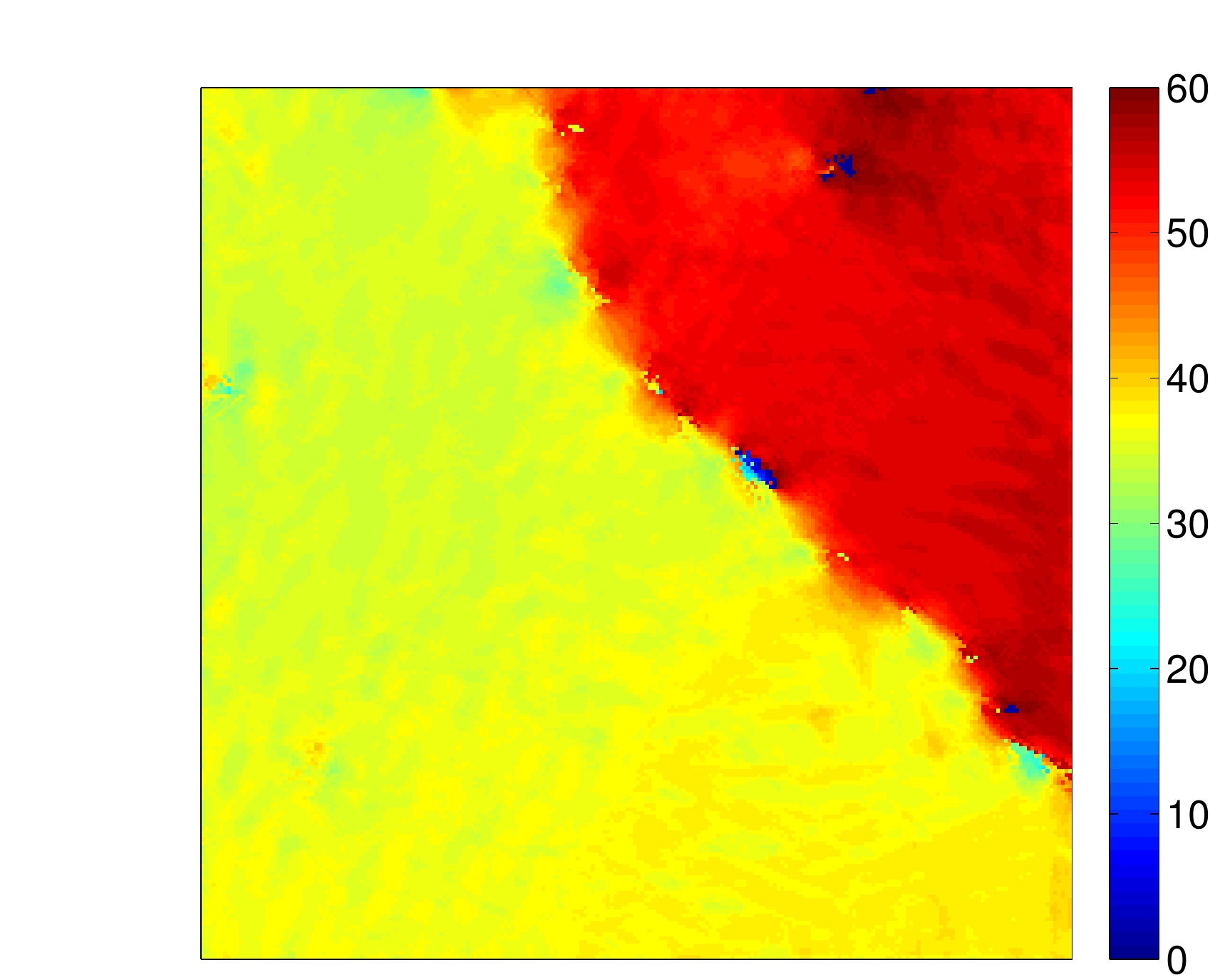}  \includegraphics[height=1.6in]{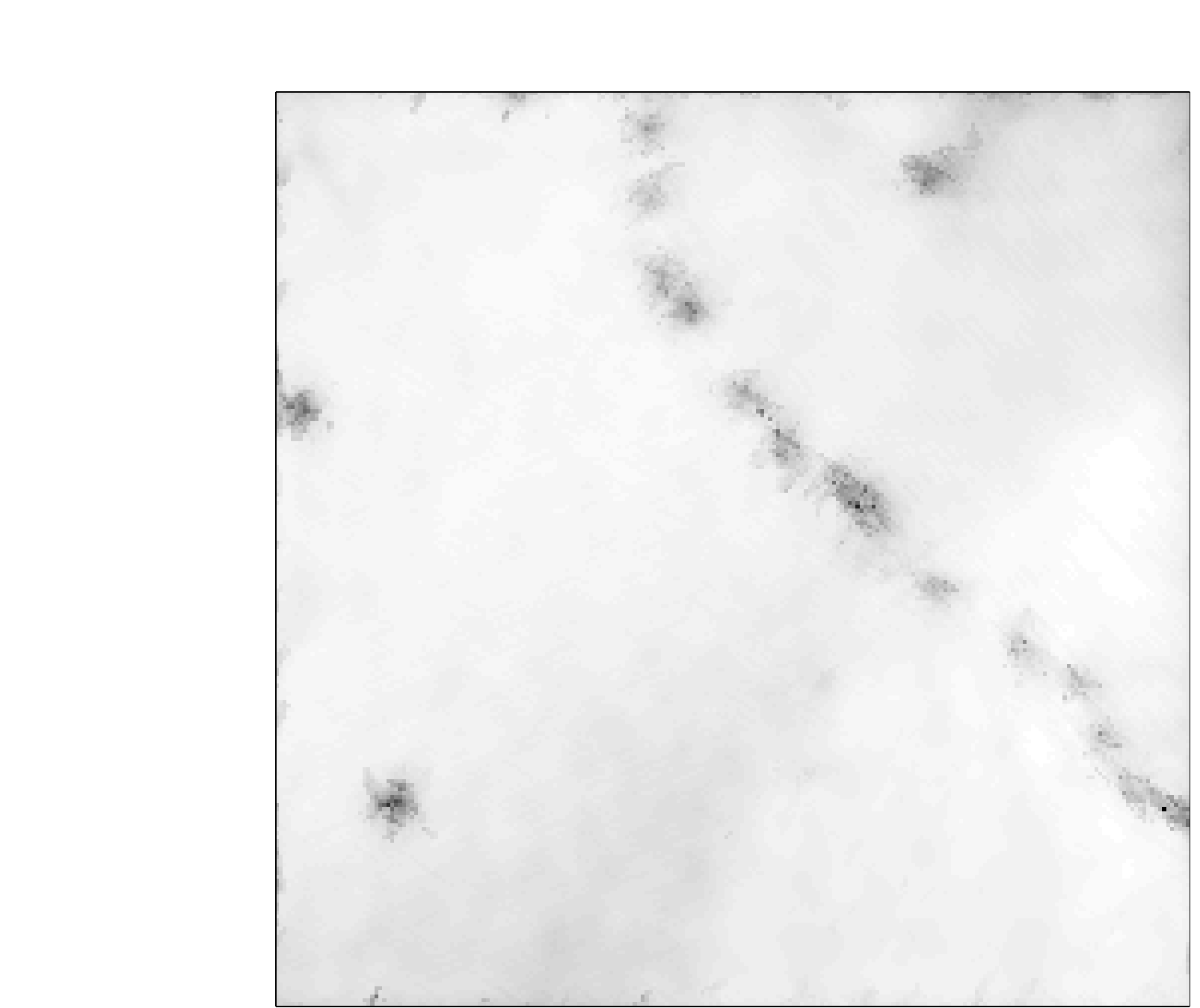}
  \end{center}
  \caption{Revisit the example in
    Figure~\ref{fig:GB3_fast}. Left: A photograph of a bubble
    raft with strong reflections and point dislocations. Middle and
    right: The weighted average angle $\Angle(b)$ and the weighted
    boundary indicator function provided by Algorithm
    $\ref{alg:deformed}$. Primary implementation parameters: $s=t=1$,
    $d=1$.}
  \label{fig:GB3_elastic}
\end{figure}

\subsection{Quantitative analysis for statistical stability}

We quantitatively analyze the performance of our algorithms in this subsection in terms of statistical stability using the example in Figure \ref{fig:GB1_elastic}. Suppose $g=f+e$ is the noisy crystal image, where $e$ is white Gaussian noise with a distribution $\sigma^2\mathcal{N}(0,1)$ and $f$ is the noiseless crystal image. We introduce the signal-to-noise ratio ({\SNR}) of the input data $g=f+e$ as follows:
\[
\SNR[dB](g)=10\log_{10}\left( \frac{\VAR (f)}{\VAR(e)}\right).
\]
We test Algorithm $\ref{alg:undeformed}$ and Algorithm $\ref{alg:deformed}$ on these noisy examples with {\SNR}'s ranging from $35$ to $-10$. All results are summarized in Figure \ref{fig:quant}.

Since the ground truth for grain boundaries and the crystal rotations is not available, we compare the results of noisy examples with those in the noiseless case. To quantify the statistical stability of the grain boundary estimation, we apply the earth mover's distance (EMD) \cite{EMDD2} to measure the distance between a boundary indicator function of a noisy example and the one of the noiseless example. The boundary indicator functions are converted to grayscale images. They can be considered as two-dimensional histograms discribing the distribution of defects. At each pixel, the image intensity is from 0 to 255. One unit in the grayscale image is one unit of the mass of the distribution. The EMD in this paper is the total mass per pixel that we need to move from one distribution to match the other one. A smaller EMD indicates better statistical stability. Figure \ref{fig:quant} (left) plots the EMD as a function of {\SNR} for the boundary indicator functions given by Algorithm $\ref{alg:undeformed}$ and Algorithm $\ref{alg:deformed}$. Although the EMD functions gradually increase as the {\SNR} decreases, the values of these functions remain reasonably small, which means that our algorithms give similar grain boundary estimates. 

To quantify the statistical stability of the crystal rotation estimation, we compute the difference of the rotation estimations in noisy and noiseless cases and measure the difference of its mean and standard deviation. The mean of the difference as a function of {\SNR} is shown in Figure \ref{fig:quant} (middle) and the standard deviation function is shown in Figure \ref{fig:quant} (right). Even though noise is heavy, most estimation errors are bounded by a small degree. This shows that our algorithms are statistically stable against noise.

\begin{figure}[ht!]
  \begin{center}
  \begin{tabular}{ccc}
\includegraphics[height=1.7in]{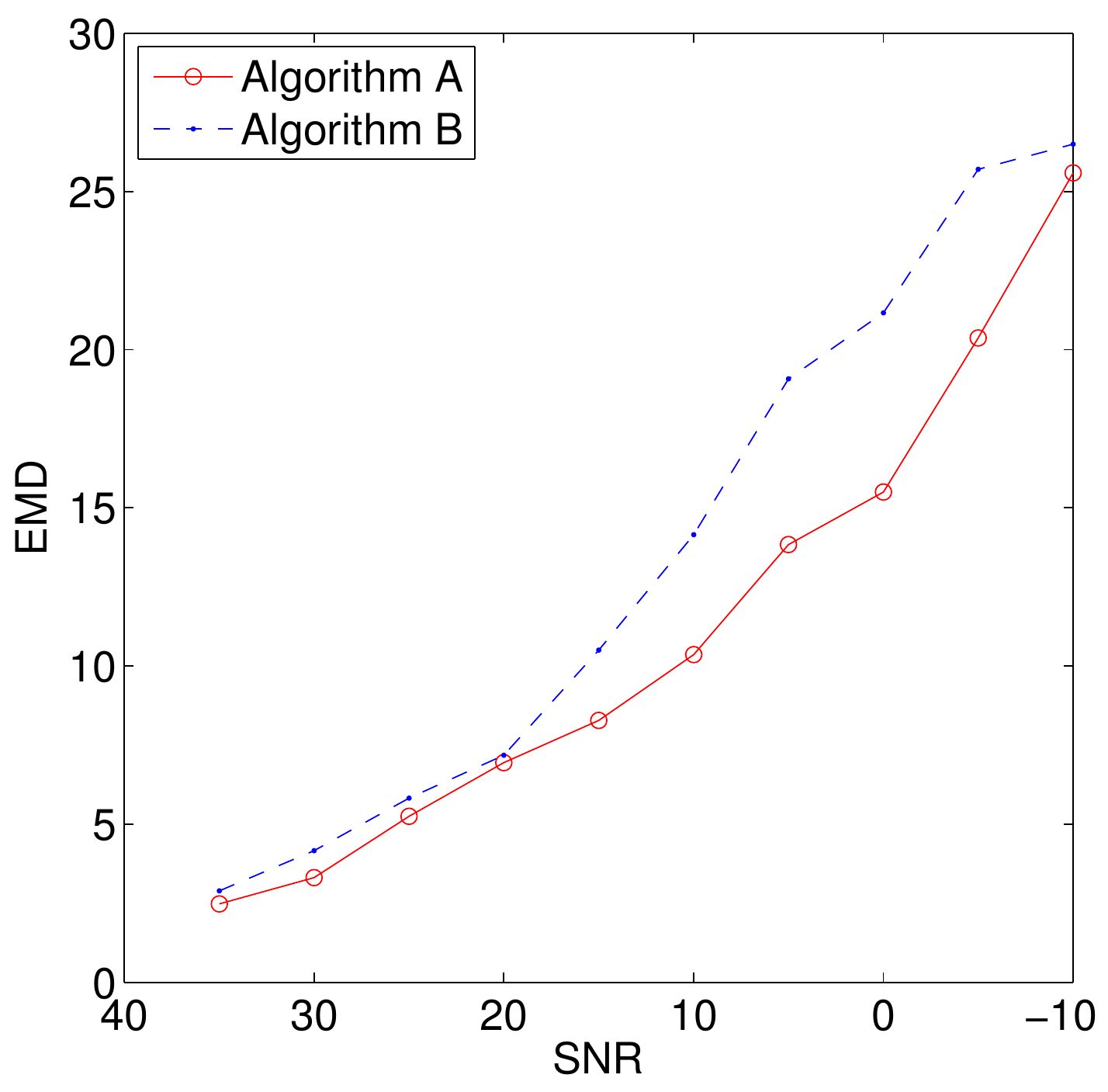} & \includegraphics[height=1.7in]{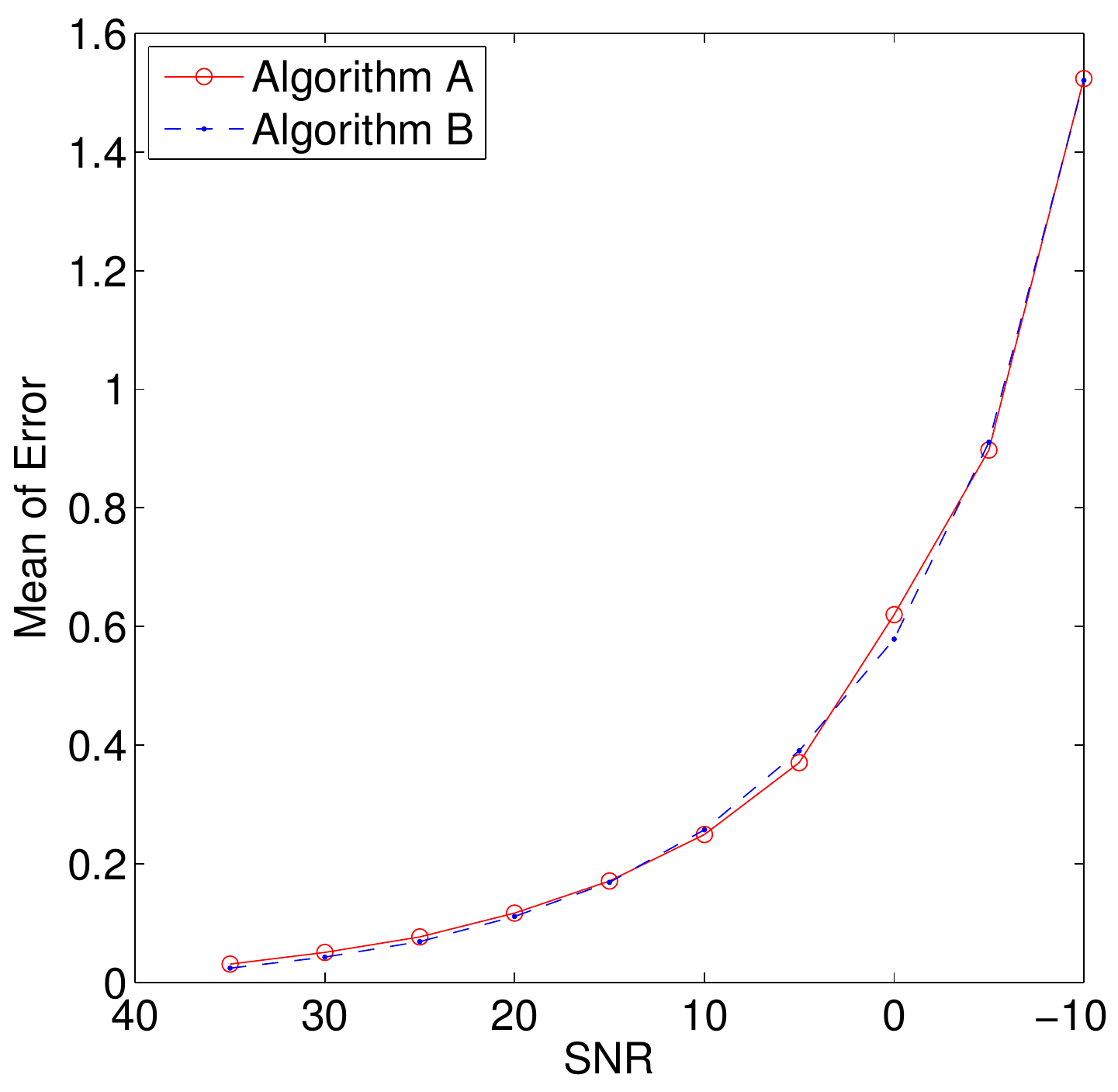} & \includegraphics[height=1.7in]{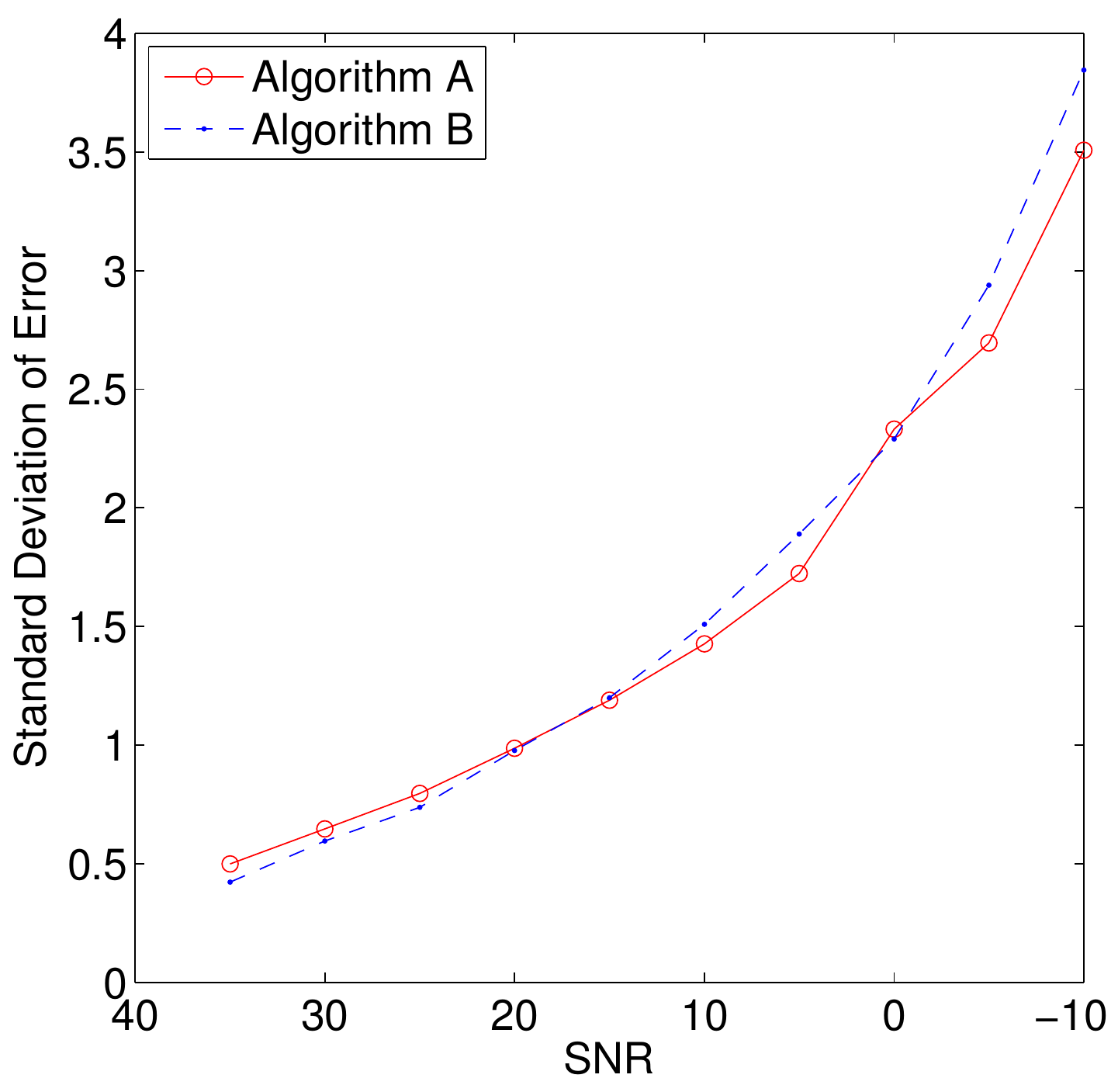}
\end{tabular}
  \end{center}
  \caption{Left: The earth mover's distance (EMD) between the estimated grain boundaries of a noisy image and that of a clean image. Middle: The mean of the estimation error of crystal rotations. Right: The standard deviation of the estimation error of crystal rotations. Data plotted above is the average of $10$ independent realizations.}
  \label{fig:quant}
\end{figure}

\section{Conclusion}
\label{sec:conclusion}

This paper has proposed a new model for atomic crystal images and an efficient tool for their multiscale analysis. Through various synthetic and real data, it has been shown that the proposed methods are able to provide robust and reliable estimates of mesoscopic properties, e.g., crystal defects, rotations, elastic deformations and grain boundaries, in a short time. Since these methods are well suitable for parallelization, the runtime will be considerably reduced by parallel computing. This would be appealing in the analysis of a series of large crystal images to study the time evolution of crystals on a microscopic length scale.

Note that the recovery of local deformation gradients is not smooth and its local distortion volume is not sparsely supported in some cases, due to the estimate error of local wave vectors and the lack of regularization. This inspires future work by combining synchrosqueezed transform with regularization and optimization to find a more accurate and smoother local deformation gradient with a sparser local distortion volume.

We focus on the analysis of images with the presence of only one type of crystal and without solid and liquid interfaces in this paper. The extension is not difficult. In fact, in the presence of liquid, the solid-liquid interface can be identified as ``boundary between grains'' by our method. One could use imaging methods for detecting geometric objects in the cartoon part of images \cites{Jung:12,Cai:13} to identify the liquid part immediately  after grains are identified by our method. Moreover, when the given image  consists of multiple types of crystals, local Fourier transforms taken at enough sampling positions can identify reference crystals. Fixing one type of reference crystals, we apply our method to extract the boundaries, the rotations, the defects and the deformations of grains of this type. A complete analysis can be obtained by combining the results of each type of reference crystals.

Another interesting and challenging future direction is to analyze crystal images corresponding to complex lattices. First, it might be difficult to automatically identify reference crystals directly from a given image. The information hidden in the image is very redundant and hence feature extraction and dimension reduction techniques are necessary. Second, the well-separation condition for synchrosqueezed transforms may not hold due to a large number of underlying wave-like components of each grain. A $2D$ generalization of the $1D$ diffeomorphism based spectral analysis method in \cite{Yang:preprint} may provide a solution to this problem.
 
The current methods can be easily extended to $3D$ crystal analysis by
designing a $3D$ synchrosqueezed transform. This should be relevant for
applications. 

\medskip

\noindent{\bf Acknowledgments.} J.L. was partially supported by the Alfred P.~Sloan foundation and National Science Foundation under award DMS-1312659. H.Y. and L.Y. were partially supported by the National Science Foundation under award DMS-0846501 and the U.S. Department of Energy's Advanced Scientific Computing Research program under award DE-FC02-13ER26134/DE-SC0009409. H.Y. also thanks the support of National Science Foundation under award ACI-1450372 and AMS-Simons Travel Award. We are grateful to Benedikt Wirth for inspiring discussions and for providing us image data from PFC simulations.
We thank Matt Elsey for helpful comments on an earlier version of the manuscript. J.L. is also grateful to Robert V.~Kohn who brought his attention to this problem originally. 

\bibliographystyle{amsxport}
\bibliography{Synchrosqueeze}

\begin{bibdiv}
\begin{biblist}

\bib{Abraham2002}{article}{
      author={Abraham, F.~F.},
      author={Walkup, R.},
      author={Gao, H.},
      author={Duchaineau, M.},
      author={Diz De La~Rubia, T.},
      author={Seager, M.},
       title={Simulating materials failure by using up to one billion atoms and
  the world's fastest computer: {W}ork-hardening},
        date={2002},
     journal={Proc. Nat. Acad. Sci. USA},
      volume={99},
       pages={5783\ndash 5787},
}

\bib{Arias:2005}{article}{
      author={Arias-Castro, E.},
      author={Donoho, D.~L.},
      author={Huo, X.},
       title={Near-optimal detection of geometric objects by fast multiscale
  methods},
        date={2005},
     journal={IEEE Trans. Inf. Theory},
      volume={51},
       pages={2402\ndash 2425},
}

\bib{Aujol2006}{article}{
      author={Aujol, J.-F.},
      author={Chan, T.~F.},
       title={Combining geometrical and textured information to perform image
  classification},
        date={2006},
     journal={J. Vis. Commun. Image. R.},
      volume={17},
       pages={1004\ndash 1023},
}

\bib{Berkels:08}{article}{
      author={Berkels, B.},
      author={R{\"a}tz, A.},
      author={Rumpf, M.},
      author={Voigt, A.},
       title={Extracting grain boundaries and macroscopic deformations from
  images on atomic scale},
        date={2008},
     journal={J. Sci. Comput.},
      volume={35},
       pages={1\ndash 23},
}

\bib{Berkels:10}{inproceedings}{
      author={Boerdgen, M.},
      author={Berkels, B.},
      author={Rumpf, M.},
      author={Cremers, D.},
       title={Convex relaxation for grain segmentation at atomic scale},
        date={2010},
   booktitle={{Vision, Modeling, and Visualization}},
   publisher={Eurographics Association},
       pages={179\ndash 186},
}

\bib{Cai:13}{article}{
      author={Cai, X.},
      author={Chan, R.},
      author={Zeng, T.},
       title={A two-stage image segmentation method using a convex variant of
  the {M}umford--{S}hah model and thresholding},
        date={2013},
     journal={SIAM J. Imaging Sci.},
      volume={6},
       pages={368\ndash 390},
}

\bib{CandesII}{article}{
      author={Cand{\`e}s, E.~J.},
      author={Donoho, D.~L.},
       title={Continuous curvelet transform. {II}. {D}iscretization and
  frames},
        date={2005},
     journal={Appl. Comput. Harmon. Anal.},
      volume={19},
       pages={198\ndash 222},
}

\bib{Chan:01}{article}{
      author={Chan, T.~F.},
      author={Vese, L.~A.},
       title={Active contours without edges},
        date={2001},
     journal={IEEE Trans. Image Process.},
      volume={10},
       pages={266\ndash 277},
}

\bib{DaubechiesLuWu:11}{article}{
      author={Daubechies, I.},
      author={Lu, J.},
      author={Wu, H.-T.},
       title={Synchrosqueezed wavelet transforms: an empirical mode
  decomposition-like tool},
        date={2011},
     journal={Appl. Comp. Harmonic Anal.},
      volume={30},
       pages={243\ndash 261},
}

\bib{DaubechiesMaes:96}{incollection}{
      author={Daubechies, I.},
      author={Maes, S.},
       title={A nonlinear squeezing of the continuous wavelet transform based
  on auditory nerve models},
        date={1996},
   booktitle={{Wavelets in Medicine and Biology}, {Ed.} {A. Aldroubi and M.
  Unser}},
   publisher={CRC Press},
       pages={527\ndash 546},
}

\bib{ElderGrantMartin:04}{article}{
      author={Elder, K.~R.},
      author={Grant, Martin},
       title={Modeling elastic and plastic deformations in nonequilibrium
  processing using phase field crystals},
        date={2004},
     journal={Phys. Rev. E},
      volume={70},
       pages={051605},
         url={http://link.aps.org/doi/10.1103/PhysRevE.70.051605},
}

\bib{ElseyWirth:13}{incollection}{
      author={Elsey, M.},
      author={Wirth, B.},
       title={Segmentation of crystal defects via local analysis of crystal
  distortion},
        date={2013},
   booktitle={{Proceedings of the 8th International Conference on Computer
  Vision Theory and Applications (VISAPP)}},
       pages={294\ndash 302},
}

\bib{ElseyWirth:12}{article}{
      author={Elsey, M.},
      author={Wirth, B.},
       title={A simple and efficient scheme for phase field crystal
  simulation},
        date={2013},
     journal={ESAIM: Math. Mod. Num. Anal.},
      volume={47},
       pages={1413\ndash 1432},
}

\bib{ElseyWirth:MMS}{article}{
      author={Elsey, M.},
      author={Wirth, B.},
       title={Fast automated detection of crystal distortion and crystal
  defects in polycrystal images},
        date={2014},
     journal={Multi. Model. Simul.},
      volume={12},
       pages={1\ndash 24},
}

\bib{GotzD95}{article}{
      author={G\"{o}tz, W.~A.},
      author={Druckm\"{u}ller, H.~J.},
       title={A fast digital {R}adon transform -- {A}n efficient means for
  evaluating the {H}ough transform.},
        date={1995},
     journal={Pattern Recogn.},
      volume={28},
       pages={1985\ndash 1992},
}

\bib{Illingworth:1988}{article}{
      author={Illingworth, J.},
      author={Kittler, J.},
       title={A survey of the {H}ough transform},
        date={1988},
        ISSN={0734-189X},
     journal={Comput. Vision Graph. Image Process.},
      volume={44},
       pages={87\ndash 116},
}

\bib{Jung:12}{article}{
      author={Jung, M.},
      author={Peyr\'e, G.},
      author={Cohen, L.},
       title={Nonlocal active contours},
        date={2012},
     journal={SIAM J. Imaging Sci.},
      volume={5},
       pages={1022\ndash 1054},
}

\bib{King:98}{article}{
      author={King, W.~E.},
      author={Campbell, G.~H.},
      author={Foiles, S.~M.},
      author={Cohen, D.},
      author={Hanson, K.~M.},
       title={Quantitative {HREM} observation of the $\sum 11(113)/[\bar{1}10]$
  grain-boundary structure in aluminium and comparison with atomistic
  simulation},
        date={1998},
     journal={J. Microsc.},
       pages={131\ndash 143},
}

\bib{Bravais}{book}{
      author={Kittel, C.},
       title={Introduction to solid state physics (7th {Ed}.)},
   publisher={Wiley},
        date={1995},
        ISBN={0471111813},
}

\bib{Meyer:2001}{book}{
      author={Meyer, Y.},
       title={Oscillating patterns in image processing and nonlinear evolution
  equations: {The Fifteenth Dean Jacqueline B. Lewis Memorial Lectures}},
   publisher={American Mathematical Society},
     address={Boston, MA, USA},
        date={2001},
        ISBN={0821829203},
}

\bib{Mumford}{article}{
      author={Mumford, D.},
      author={Shah, J.},
       title={Optimal approximations by piecewise smooth functions and
  associated variational problems},
        date={1989},
        ISSN={1097-0312},
     journal={Comm. Pure Appl. Math.},
      volume={42},
       pages={577\ndash 685},
         url={http://dx.doi.org/10.1002/cpa.3160420503},
}

\bib{EMDD2}{inproceedings}{
      author={Pele, O.},
      author={Werman, M.},
       title={{Fast and robust Earth Mover's Distances}},
        date={2009},
   booktitle={{Computer Vision, 2009 IEEE 12th International Conference on}},
       pages={460\ndash 467},
}

\bib{Sandberg:02}{techreport}{
      author={Sandberg, B.},
      author={Chan, T.~F.},
      author={Vese, L.~A.},
       title={A level-set and {G}abor-based active contour algorithm for
  segmenting textured images},
 institution={UCLA Department of Mathematics CAM report},
        date={2002},
}

\bib{SingerSinger:06}{article}{
      author={Singer, H.~M.},
      author={Singer, I.},
       title={Analysis and visualization of multiply oriented lattice
  structures by a two-dimensional continuous wavelet transform},
        date={2006},
     journal={Phys. Rev. E},
      volume={74},
       pages={031103},
}

\bib{Strekalovskiy:11}{inproceedings}{
      author={Strekalovskiy, E.},
      author={Cremers, D.},
       title={Total variation for cyclic structures: Convex relaxation and
  efficient minimization},
        date={2011},
   booktitle={{Computer Vision and Pattern Recognition (CVPR), 2011 IEEE
  Conference on}},
       pages={1905\ndash 1911},
}

\bib{StukowskiAlbe1:10}{article}{
      author={Stukowski, A.},
      author={Albe, K.},
       title={Dislocation detection algorithm for atomistic simulations},
        date={2010},
     journal={Model. Simul. Mater. Sci. Eng.},
      volume={18},
       pages={025016},
         url={http://stacks.iop.org/0965-0393/18/i=2/a=025016},
}

\bib{StukowskiAlbe2:10}{article}{
      author={Stukowski, A.},
      author={Albe, K.},
       title={Extracting dislocations and non-dislocation crystal defects from
  atomistic simulation data},
        date={2010},
     journal={Model. Simul. Mater. Sci. Eng.},
      volume={18},
       pages={085001},
         url={http://stacks.iop.org/0965-0393/18/i=8/a=085001},
}

\bib{Unser:95}{article}{
      author={Unser, M.},
       title={Texture classification and segmentation using wavelet frames},
        date={1995},
        ISSN={1057-7149},
     journal={IEEE Trans. Image Proc.},
      volume={4},
       pages={1549\ndash 1560},
         url={http://dx.doi.org/10.1109/83.469936},
}

\bib{Vese:02}{article}{
      author={Vese, L.~A.},
      author={Chan, T.~F.},
       title={A multiphase level set framework for image segmentation using the
  {M}umford and {S}hah model},
        date={2002},
        ISSN={0920-5691},
     journal={Int. J. Comput. Vision},
      volume={50},
       pages={271\ndash 293},
}

\bib{Vese:03}{article}{
      author={Vese, L.~A.},
      author={Osher, S.~J.},
       title={Modeling textures with total variation minimization and
  oscillating patterns in image processing},
    language={English},
        date={2003},
        ISSN={0885-7474},
     journal={J. Sci. Comput.},
      volume={19},
       pages={553\ndash 572},
}

\bib{Hau-Tieng:2013}{article}{
      author={Wu, H.-T.},
       title={Instantaneous frequency and wave shape functions ({I})},
        date={2013},
     journal={Appl. Comp. Harmonic Anal.},
      volume={35},
       pages={181\ndash 199},
}

\bib{YangYingPreprint:2014}{article}{
      author={Yang, H.},
       title={Robustness analysis of synchrosqueezed transforms},
        date={2014},
     journal={arXiv:1410.5939 [math.ST]},
         url={http://arxiv.org/abs/1410.5939},
        note={preprint},
}

\bib{Yang:preprint}{article}{
      author={Yang, H.},
       title={Synchrosqueezed wave packet transforms and diffeomorphism based
  spectral analysis for 1{D} general mode decompositions},
        date={2015},
        ISSN={1063-5203},
     journal={Appl. Comput. Harmon. Anal.},
      volume={39},
      number={1},
       pages={33 \ndash  66},
  url={http://www.sciencedirect.com/science/article/pii/S1063520314001146},
}

\bib{Canvas}{article}{
      author={Yang, H.},
      author={Lu, J.},
      author={Brown, W.P.},
      author={Daubechies, I.},
      author={Ying, L.},
       title={Quantitative canvas weave analysis using 2-{D} synchrosqueezed
  transforms: Application of time-frequency analysis to art investigation},
        date={2015July},
        ISSN={1053-5888},
     journal={Signal Processing Magazine, IEEE},
      volume={32},
      number={4},
       pages={55\ndash 63},
}

\bib{YangYing:2013}{article}{
      author={Yang, H.},
      author={Ying, L.},
       title={Synchrosqueezed wave packet transform for 2{D} mode
  decomposition},
        date={2013},
     journal={SIAM J. Imaging Sci.},
      volume={6},
       pages={1979\ndash 2009},
}

\bib{YangYing:preprint}{article}{
      author={Yang, H.},
      author={Ying, L.},
       title={Synchrosqueezed curvelet transform for two-dimensional mode
  decomposition},
        date={2014},
     journal={SIAM J. Math. Anal.},
      volume={46},
      number={3},
       pages={2052\ndash 2083},
         url={http://dx.doi.org/10.1137/130939912},
}

\bib{Yu04detectingcircular}{article}{
      author={Yu, Z.},
      author={Bajaj, C.},
       title={Detecting circular and rectangular particles based on geometric
  feature detection},
        date={2004},
     journal={J. Struct. Biol.},
      volume={145},
       pages={168\ndash 180},
}

\bib{ZengGipson:07}{article}{
      author={Zeng, X.},
      author={Gipson, B.},
      author={Zheng, Z.~Y.},
      author={Renault, L.},
      author={Stahlberg, H.},
       title={Automatic lattice determination for two-dimensional crystal
  images},
        date={2007},
     journal={J. Struct. Biol.},
      volume={160},
       pages={353\ndash 361},
}

\end{biblist}
\end{bibdiv}

\end{document}